\documentclass[10pt,reqno]{amsart}
\usepackage{amscd,amsmath,amsopn,amssymb,amsthm,multicol}
\usepackage{hyperref}
\usepackage[all]{xy}
\usepackage{amscd}
\usepackage{graphicx}

\DeclareMathOperator{\Ad}{Ad}
\DeclareMathOperator{\ad}{ad}
\DeclareMathOperator{\Aut}{Aut}

\DeclareMathOperator{\Id}{Id}

\DeclareMathOperator{\Hgt}{ht}

\DeclareMathOperator{\Ric}{Ric}

\DeclareMathOperator{\Span}{span}
\newcommand{\fr}{\mathfrak}
\newcommand{\al}{\alpha}
\newcommand{\be}{\beta}

\newtheorem{theorem}{Theorem}
\newtheorem{lemma}{Lemma}
\newtheorem{remark}{Remark}
\newtheorem{prop}{Proposition}
\newtheorem{definition}{Definition} 

\newtheorem{example}{Example}

\begin{document}

\title
{Invariant Einstein metrics  on  flag manifolds with four isotropy summands}

\author{Andreas Arvanitoyeorgos and Ioannis Chrysikos}
\address{University of Patras, Department of Mathematics, GR-26500 Rion, Greece}
\email{arvanito@math.upatras.gr}
\email{xrysikos@master.math.upatras.gr}
 
\medskip
\noindent
\thanks{The authors  were partially supported   
  by the C.~Carath\'{e}odory grant \#C.161 2007-10, 
  University of Patras.}

\begin{abstract}
A  generalized flag manifold  is a 
homogeneous space  of the form $G/K$, 
where $K$ is the centralizer of a torus 
in a compact connected semisimple Lie group $G$. 
We classify all flag manifolds 
with four isotropy summands by the use of $\fr{t}$-roots.  
We present new  $G$-invariant Einstein metrics 
 by solving explicity the Einstein equation.  We also examine the isometric problem 
 for these Einstein metrics.



\end{abstract}

\maketitle

\section*{Introduction}
\markboth{Andreas Arvanitoyeorgos and Ioannis Chrysikos}{Invariant Einstein metrics  on  
flag manifolds with four isotropy summands}

A Riemannian manifold  $(M,g)$   is called Einstein if the Ricci tensor
 $\Ric_{g}$ is a constant multiple of the metric $g$, i.e.  
 ${\rm  Ric}_{g}=e\cdot g $, 
  for some $e\in\mathbb{R}$.   If $M$ is compact, 
  Einstein metrics of   volume 1 
  can be characterized variationally  as the critical 
  points of the scalar curvature functional 
    $T(g)=\int_{M}S_{g}d {\rm vol}_{g}$ 
  on the space  $\mathcal{M}_{1}$ of Riemannian metrics of 
  volume 1.  At present no general existence results 
  of Einstein metrics are known, 
  except of some important classes of Einstein metrics, 
  such as K\"ahler-Einstein metrics (\cite{Ti}), 
  Sasakian-Einstein metrics (\cite{BGa}), homogeneous Einstein metrics 
  (see \cite{BWZ} for the compact case and \cite{Heb} for the noncompact case).
  For more details on Einstein manifolds we refer to \cite{Be}.

Important progress has been made in the homogeneous case, 
that is  when a Lie group $G$ acts 
transitively on the given manifold  $M$.  
In this case we are interested  in $G$-invariant
Einstein metrics on $M$.  Such a metric is precisely a critical 
point  of $T$ restricted to
$\mathcal{M}_{1}^{G}$, the set  of $G$-invariant metrics of volume 1 (cf. \cite[p.~120]{Be}).
As a consequence, the Einstein equation reduces to a system 
of non-linear algebraic equations, 
which in some  cases can be solved explicity.  
Most known examples of Einstein manifolds are
homogeneous. For example, all compact and simply 
connected homogeneous manifolds of dimension less or equal to 11, 
admit a  homogeneous Einstein metric (\cite{Bom}).
 However, we are far from a classification of homogeneous Einstein manifolds. 
For more recent results on homogeneous Einstein metrics see the 
 syrveys \cite{LW}  and \cite{NRS}.

    An important class of compact homogeneous spaces  consists of the
    {\it generalized flag manifolds}.
      These are   adjoint orbits of a compact connected semisimple Lie group $G$, and 
     are homogeneous spaces of the form  $M=G/ C(S)$,  where $C(S)$ is the centralizer 
   of a torus $S$ in $G$.    If $S$ is a maximal torus in $G$, then $C(S)=S$ and 
   $M=G/S$  is called a {\it full flag manifold}. 
   Excellent references for the structure and the geometry of flag manifolds  are 
    \cite[Chapter 8]{Be} and the articles  \cite{AP}, \cite{B}, and \cite{Bor}.  
   
  Generalized flag manifolds
   have important applications in the physics of  elementary particles,
   where they give rise to a broad class  of supersymmetric sigma models.  
   Their importance arises from their rich complex geometry,
    since they  exhaust  the compact  
simply connected  homogeneous K\"ahler manifolds.
A flag manifold $M=G/K$ admits  a finite number 
of $G$-invariant complex structures and the first Chern class is positive.
In particular, for any $G$-invariant complex structure on $M$, 
there is a compatible K\"ahler-Einstein metric (cf. \cite{AP}, \cite{Bor}).  
 Morever, one can construct
a holomorphic embedding of $M$ into a complex projective space,
and for this reason generalized flag manifolds are also referred to as 
{\it K\"ahler $C$-spaces}, a term that we will use as well.

 In this paper we present new invariant Einstein metrics 
for all flag manifolds for which the isotropy representation 
decomposes into four inequivalent irreducible submodules.   
The problem of finding new non-K\"ahler  Einstein metrics on   flag manifolds was 
initially stydied by D. V. Alekseevsky in \cite{Ale1}.
For some of these spaces the normal  metric (i.e. the metric induced from the Killing form)
 is Einstein, as they  appear in the 
work  \cite{Wa1}  
of M. Wang and W. Ziller, where they classified all normal 
homogeneous Einstein manifolds. 
Such an  example is the flag manifold $SU(nk)/S(U(k)\times \cdots\times U(k))$ $(n \ \mbox{times}),$ $k\geq 2, \ n\geq 3$.  
In the last two decades  progress 
has been made by  Y. Sakane \cite{Sak}, M. Kimura \cite{Kim},   
J.C. Negreiros  \cite{DSN},  the authors \cite{Arv}, \cite{Chry}, and others. 
There are no general theorems about the existence or non existence of invariant (non-K\"ahler) Einstein metrics
on flag manifolds.   The only general result states  that if
$T$ is a maximal 
  torus in a compact semisimple connected Lie group whose local factors are of 
  type $A_{\ell}, D_{\ell}, E_{6}, E_{7},$ or $E_{8}$, 
then the 
  normal metric on $M=G/T$ is Einstein (\cite{Wa1}).

  Concerning the full flag 
  $\mathbb{F}_{n}=SU(n)/T^n$, it is  known (cf. \cite{Arv}) that  for $n\geq 4$, it
  admits at least $n!/2+n+1$ invariant Einstein metrics (the $n!/2$ metrics are K\"ahler-Einstein and one is the normal metric).  
  A few years later, Sakane (\cite{Sak}) proved 
  that the space $\mathbb{F}_{2m}$ ($m\geq 2$) admits another family of Einstein metrics. 
  Recently, Dos Santos and 
  Negreiros (\cite{DSN}) found a third class of homogeneous Einstein metrics on $\mathbb{F}_{n}$ 
  (for $n=2m$, or $n=2m+1$ and $m\geq 6$).  However, a complete classification 
  of the Einstein metrics on this full flag is  unknown and it is not 
  even known if the number of such metrics is finite (the B\"ohm-Wang-Ziller conjecture). 
  We denote the above three   different classes of invariant (non-K\"ahler and non-normal) Einstein metrics  
  on $\mathbb{F}_{n}$, by $\mathcal{E}_1, \mathcal{E}_2$ and $\mathcal{E}_3$. 
  
   In the following Table 1  we list, to the best of our knowledge, all   
known results about the number of invariant Einstein metrics  on generalized flag manifolds  $M=G/K$,
including our new ones.

  \medskip
  \begin{center}
   \Small{\begin{tabular}{|l|l|c|c|l|}
 \hline
  Generalized  flag manifold $M=G/K$ & $\fr{m}=\bigoplus_{i=1}^{s}\fr{m}_{i}$          & K-E & Normal   &   Number of Einstein metrics  \\
  \hline\hline
 $SO(2\ell +1)/U(\ell-m)\times SO(2\ell+1) \ \ (\ell-m\neq 1)$                 & $s=2$     & 1   &  -   &  $=2$   \cite{Ker}, \cite{Chry}\\
 $Sp(\ell)/U(\ell-m)\times Sp(m) \ \ (m\neq 0)$                                & $s=2$     & 1   &  -   &  $=2$   \cite{Ker}, \cite{Chry}\\
 $SO(2\ell)/U(\ell-m)\times SO(2m) \ \ (\ell-m\neq 1, \ m\neq 0)$                      &  $s=2$     & 1   &  -   &  $=2$  \cite{Ker}, \cite{Chry} \\ 
 $G_{2}/U(2) \ \  (U(2) \ \ \mbox{represented by the short root})$ &  $s=2$ & 1 & - & $=2$ \cite{Ker}, \cite{Chry} \\
 $ F_{4}/SO(7)\times U(1)$  & $s=2$ &1 &- & $=2$ \cite{Ker}, \cite{Chry} \\
 $F_{4}/Sp(3)\times U(1)$   & $s=2$ &1 &- & $=2$ \cite{Ker}, \cite{Chry}  \\ 
 $E_{6}/SU(6)\times U(1)$  & $s=2$ &1 &- & $=2$  \cite{Ker}, \cite{Chry}   \\
 $E_{6}/SU(2)\times SU(5)\times U(1)$ & $s=2$ &1 &- & $=2$  \cite{Ker}, \cite{Chry}   \\
 $E_{7}/SU(7)\times U(1)$ & $s=2$ &1 &- & $=2$  \cite{Ker}, \cite{Chry}   \\
 $E_{7}/ SU(2)\times SO(10)\times U(1)$ & $s=2$ &1 &- & $=2$ \cite{Ker}, \cite{Chry}   \\
 $E_{7}/ SO(12)\times U(1)$ & $s=2$ &1 &- & $=2$  \cite{Ker}, \cite{Chry}  \\
 $E_{8}/E_{7}\times U(1)$ & $s=2$  &1 &- & $=2$  \cite{Ker}, \cite{Chry}   \\
 $E_{8}/SO(14)\times U(1)$ & $s=2$ &1 &- & $=2$  \cite{Ker}, \cite{Chry}   \\
 \hline  
 $SU(\ell_{1}+\ell_{2}+\ell_{3})/S(U(\ell_{1})\times U(\ell_{2})\times U(\ell_{3}))$ &$s=3$  & 3   &  -   &  $=4$ \cite{Kim}, \cite{Arv}  \\ 
 $SU(3\ell)/S(U(\ell)\times U(\ell)\times U(\ell))$                           & $s=3$     & 3   & $\checkmark$  & $=4$  \cite{Kim}, \cite{Arv},  \cite{Wa1}\\
 $SO(2\ell)/U(1)\times U(\ell-1) \ \ (\ell\geq 4)$      & $s=3$  & 3 & - & $=4$   \cite{Kim}, \cite{LNF}\\ 
 $G_{2}/U(2) \ \ (U(2) \ \ \mbox{represented by the long root})$ & $s=3$ & 1 &- & $=3$ \cite{Kim},   \cite{Arv} \\
 $F_{4}/U(2)\times SU(3)$ & $s=3$ &1 &- & $=3$  \cite{Kim}\\
 $E_{6}/U(1)\times U(1)\times SO(8)$ & $s=3$ &3 &$\checkmark$ & $=4$  \cite{Kim}, \cite{Wa1}\\
 $E_{6}/U(2)\times SU(3)\times SU(3)$ & $s=3$ &1 &- &$=3$ \cite{Kim} \\
 $E_{7}/U(3)\times SU(5)$ & $s=3$ &1 &- &$=3$ \cite{Kim} \\ 
 $E_{7}/U(2)\times SU(6)$ & $s=3$ &1 &- &$=3$ \cite{Kim} \\ 
 $E_{8}/U(2)\times E_{6}$ & $s=3$ &1 &- &$=3$  \cite{Kim}\\
 $E_{8}/U(8) $ & $s=3$ &1 &- &$=3$ \cite{Kim} \\ 
 \hline  
 $(1) \ F_{4}/SU(3)\times SU(2)\times U(1)$             & $s=4$ & 1 & - &  $=3 \ \mbox{new}$     \\
 $(2) \ E_{7}/SU(4)\times SU(3)\times SU(2)\times U(1)$ & $s=4$ & 1 & - &  $=3 \ \mbox{new}$    \\
 $(3) \ E_{8}/SU(7)\times SU(2)\times U(1)$             & $s=4$ & 1 & - &  $=3 \ \mbox{new}$    \\  
 $(4) \ E_{8}/SO(10)\times SU(3)\times U(1)$            & $s=4$ & 1 & - &  $=5 \ \mbox{new}$    \\ 
 $(5) \ E_{6}/SU(5)\times U(1)\times U(1)$              & $s=4$ & 4 & - &  $=8 \ \mbox{new}$   \\
 $(6) \ E_{7}/SO(10)\times U(1)\times U(1)$             & $s=4$ & 4 & - &  $=8 \ \mbox{new}$   \\  
 $(7) \ SO(2\ell+1)/U(1)\times U(1)\times SO(2\ell-3) \ \ (\ell\geq 2)$        & $s=4$ & 4 & - &  $=8\ \mbox{new} \ \ (\ell\geq 3)$   \\
 $(8) \ SO(2\ell)/U(1)\times U(1)\times SO(2\ell-4) \ \ (\ell\geq 3)$        & $s=4$ & 4 & - &  $=8 \ \mbox{new}$   \\
 $(9) \ SO(2\ell)/U(p)\times U(\ell-p)   \ \ (\ell\geq 4, \ 2\leq p\leq \ell-2)$                      & $s=4$ & 4 & - &  $\geq 6 \ \mbox{new}$  \\
 $(10) \ Sp(\ell)/U(p)\times U(\ell-p) \ \ (\ell\geq 2, \ 1\leq p\leq \ell-1)$                        & $s=4$ & 4 & - &  $\geq 4 \ \mbox{new}$   \\
 $(11) \ SO(4p)/U(p)\times U(p) $       & $s=4$ & 4 & - &  $\geq 6 \ \mbox{new} \ \ (p\geq 2)$    \\
 $(12) \ Sp(2p)/U(p)\times U(p)$                         & $s=4$ & 4 & - &  $=6 \ \mbox{new} \ \ (p\geq 1)$     \\
   \hline
    \end{tabular}}
 \end{center}
  \medskip
 
 \begin{center} 
 {\sc Table 1.} \ The number of invariant Einstein metrics on generalized flag manifolds  
 \end{center}
 \medskip

The second column of Table 1 gives the number of irreducible, non equivalent 
components of the isotropy representation of $M=G/K$.  
The column   K-E  gives the number of the corresponding 
K\"ahler-Eistein metrics on $M$ (up to scalar).  
If $M$ is a normal homogeneous Einstein manifold  (according to  \cite{Wa1}) 
we make a note in the column ``Normal".   In the fifth column  
when we write ``$=m$"  we mean that there exist precisely $m$ $G$-invariant Einstein 
metrics on $M$ (up to isometry).  When we write ``$\geq m$" we mean that $M$ admits 
at least $m$ 	$G$-invariant metrics.    

In the following Table 2 we give the corresponding known results (we  don't include  the column  K-E) 
for  full flag manifolds.
 
  \medskip
  \begin{center}
 \Small{
\begin{tabular}{|l|l|c|l|}
   \hline
  Full flag manifold $M=G/K$ & $\fr{m}=\bigoplus_{i=1}^{s}\fr{m}_{i}$          & Normal   &  Number of Einstein metrics  \\
  \hline\hline
  $SU(3)/T$                & $s=3$       & $\checkmark$   & $=4$ \ \cite{Arv}, \cite{Sak}\\
 $SU(n)/T  \ (n\geq 4) $  &  $s=n(n-1)/2$  & $\checkmark$  &     $\geq n!/2+1+\mathcal{E}_1$ \cite{Arv}\\      
 $SU(2m)/T \ (m\geq 2)$  & $s=m(2m-1)$  & $\checkmark$ &  $\geq (2m)!/2+1+\mathcal{E}_1+\mathcal{E}_2$ \ \cite{Sak} \\ 
  $SU(2m)/T \ (m\geq 6)$  & $s=m(2m-1)$  & $\checkmark$ &  $\geq (2m)!/2+1+\mathcal{E}_1+\mathcal{E}_2+\mathcal{E}_3$ \ \cite{DSN} \\    
   $SU(2m+1)/T \ (m\geq 6)$  & $s=m(2m+1)$  & $\checkmark$ &  $\geq (2m+1)!/2+1+\mathcal{E}_1+\mathcal{E}_2+\mathcal{E}_3$ \  \\  
   $SO(5)/T $ & $s=4$ &  - & $\geq 6$ \ \cite{Sak} \\
   $SO(2n+1)/T \ (n\geq 12)$ & $s=n^2$  & - & $\geq 2$ \  \\
   $Sp(n)/T \ (n\geq 8)$ & $s=n^2$  & - & $\geq 2$ \  \\ 
   $SO(2n)/T \ (n\geq 8)$ & $s=n(n-1)$  & $\checkmark$ & $\geq 2$ \   \\ 
           \hline
 \end{tabular}}
 \end{center}
\medskip

    \begin{center}
 {\sc Table 2.} \ The number of invariant Einstein metrics on full flag manifolds  
\end{center}

   In a  recent work \cite{Gr}, M. Graev  used Newton polytopes of certain compact 
homogeneous spaces $M=G/K$ with simple spectrum, to give the number $\mathcal{E}(M)$ 
of isolated  $G$-invariant  {\it complex}  Einstein metrics on $M$.  
Among other results, he confirms some known results about $\mathcal{E}(M)$ on  flag manifolds.  
For example,  he proves that  the spaces $SU(3)/T$ and $E_{6}/T^{2}\times SO(8)$ 
admit four (real) homogeneous Einstein metrics (cf. \cite{Kim}, \cite{Arv}). 
He also gives new results for the families
$SO(N)/U(1)\times U(1)\times SO(N-4)$, $(N=2n,\ \mbox{or} \ N=2n+1)$,   $SO(4n)/U(n)\times U(n)$,  
$Sp(2n)/U(n)\times U(n)$ and others.  In particular, he proves that 
if $M$ is one of the above spaces, then 
$\mathcal{E}(M)=10$, 
and if $M$ is one of the flag manifolds $SO(2p+2q)/U(p)\times U(q)$ $(p>q\geq 2)$,  
$Sp(p+q)/U(p)\times U(q)$ $(p>q\geq 1)$, $E_{6}/T^2\times SU(5)$, or $E_7/T^2\times SO(10)$ then $\mathcal{E}(M)=12$ (\cite[pp.~1053]{Gr}).
However, as it will be shown in the present work,  the number of {\it real} invariant Einstein metrics on these flag manifolds  is smaller.
 
 \medskip
 
   The paper is organized as follows: 
   In Section 1 we recall some known facts about 
     reductive homogeneous spaces $M=G/K$, 
   and the Ricci curvature.  
   In  Section 2 we study the  structure 
   of  a generalized flag manifold  $M=G/K$ in terms of painted Dynkin diagrams.
   By use of the notion of $\fr{t}$-roots we 
   classify all flag manifolds
    with four isotropy summands.  
    In particular, we prove the following theorem (see Propositions \ref{classification} and \ref{cor1}):
    
    \medskip
    { \sc{Theorem A.}}
    {\it Let $M=G/K$ be a generalized flag manifold whose isotropy 
    representation decomposes into four inequivalent real irreducible 
    $\ad(\fr{k})$-submodules.  Then $M$ is localy 
    isomorphic to one of the spaces (1)-(10) given in Table 1.} 
   
   \medskip 
     Note that spaces  (11) and (12) are special cases of (9) and (10) (by setting  $\ell=2p$), and that    
        the full flag manifold $SO(5)/T$  in Table 2 is obtained from (7) for $\ell=2$.  
   
    In order to simplify  our study,  we separate  
   flag manifolds with four isotropy summands into two classes, 
     those of {\it Type I} and these of {\it Type  II}, depending 
    on the number of simple black roots in the corresponding Dynkin 
    diagram.  In Section  3 we discuss $G$-invariant complex structures 
    and $G$-invariant K\"ahler-Einstein metrics
    on flag manifolds of these two types. It turns out that finding
    K\"ahler-Einstein metrics for flag manifolds with several isotropy summands is a demanding task.
    Furthermore, the knowledge of K\"ahler-Einstein metrics is of central use towards finding other
    invariant Einstein metrics in the following sections.  
   In Section 4 we use the twistor 
   fibration $G/K\to G/U$ of a flag manifold  over an irreducible symmetric space 
   of compact type,   in order to obtain an explicit form of
   the Einstein equation for flag manifolds of Type I, 
   and give solutions.  In Section 5  we investigate 
   homogeneous Einstein metrics for   flag manifolds of Type II.   
   As a result, we sharpen the number $\mathcal{E}(M)$ of $G$-invariant {\it complex} Einstein metrics
   found by M. Graev, to {\it real} solutions.
    Finally, in Section 6 
   we examine   the isometric problem for the Einstein metrics obtained in Sections 4 and 5.  
   Our main theorem is the following and refers to Table 1:
   
       \medskip
    { \sc{Theorem B.}}
   {\it Let $M=G/K$ be a generalized flag manifold with four isotropy summands. 
   \begin{enumerate}
   \item If $M$ is one of the spaces (1), (2), or (3), then $M$ admits exactly three non-isometric $G$-invariant 
   Einstein metrics. One  is a K\"ahler-Einstein metric, 
   and the other two are non-K\"ahler Einstein metrics (cf. Theorem \ref{The5}).    
    If $M$ is the space (4), then $M$ admits five non-isometric $G$-invariant  Einstein metrics. 
    One   is K\"ahler, and the other four are non-K\"ahler Einstein metrics (cf. Theorem \ref{The5}).
   \item If $M$ is either (5), (6),(7), or (8),  then $M$ admits exactly eight $G$-invariant 
   Einstein metrics.  There are two pairs of isometric  K\"ahler-Einstein metrics, and four  non-isometric, non-K\"ahler  Einstein metrics,
    (cf. Theorem \ref{The6},   \ref{The7} and   \ref{The77}). 
      \item If $M$ is the space (11), then $M$ admits at least six $G$-invariant   
       Einstein metrics.  There are two pairs of isometric K\"ahler-Einstein metrics,
    and   two  non-isometric and non-K\"ahler Einstein metrics (cf. Theorem  \ref{The8}).
   If $2\leq p\leq 6$ then $M$  admits precisely eight $G$-invariant    Einstein metrics. 
   The  new  non-K\"ahler  Einstein metrics are given  explicity in Theorem \ref{The8}.  
   If $M$ is the space (12) then  $M$ admits exactly six $G$-invariant Einstein metrics. 
   There are two pairs of isometric K\"ahler-Einstein metrics, and  two   non-K\"ahler   Einstein metrics. 
   These metrics are given explicity
   in Theorem \ref{The9}.
   \end{enumerate}}
    
    \medskip
    We also give the following existence theorem.
    
   \medskip
    { \sc{Theorem C.}} {\it Let $M=G/K=SO(2\ell)/U(p)\times U(\ell-p)$ with $\ell\geq 4$ and $2\leq p\leq \ell-2$ (space (9)). 
    Then $M$ admits at least two real
    (non-K\"ahler)   $G$-invariant Einstein metrics (cf. 	Theorem \ref{existence}). }

  \markboth{Andreas Arvanitoyeorgos and Ioannis Chrysikos}{Homogeneous Einstein metrics  on  flag manifolds with four isotropy summands}
\section{The Ricci tensor of reductive  homogeneous spaces}
\markboth{Andreas Arvanitoyeorgos and Ioannis Chrysikos}{Homogeneous Einstein metrics  on  flag manifolds with four isotropy summands}

A Riemannian manifold $(M, g)$ is called $G$-homogeneous if there exists
a closed subgroup $G$ of ${\rm Isom}(M, g)$ which acts transitively on $M$.  
Let 
$K=\{g\in G : gp=p\}$ be the isotropy subgroup 
at $p$.  Then  $M\cong G/K$.  Note that $K$ is compact since $K\subset O(T_{p}M)$, 
where $T_{p}M$ is the tangent space of $M$ at $p$.   

Let $M=G/K$ be a homogeneous manifold with $G$  a 
compact, connected and semisimple Lie group and $K$ a closed 
subgroup. Let   $o=eK$  be the identity coset of $G/K$. 
We denote by $\fr{g}$ and $\fr{k}$ the corresponding Lie 
algrebras of $G$ and $K$. Let $B$ denote the  
Killing form of $\fr{g}$.  Recall that $-B$ is a positive 
definite inner product on $\fr{g}$, and we consider the orthogonal 
decomposition $\fr{g}=\fr{k}\oplus\fr{m}$ with respect to $-B$.
This is a reductive decomposition 
of $\fr{g}$, i.e.  $\Ad(K)\fr{m}\subset\fr{m}$, and the tangent 
space $T_{o}M$ is identified with $\fr{m}$.  The last equation 
implies the relation $[\fr{k}, \fr{m}]\subset\fr{m}$, and the 
converse is true if $K$ is connected.     Let  $\chi : K\to \Aut(T_{o}M)$ 
be the isotropy representation of $K$ on $T_{o}M$. 
Then $\Ad^{G}\big|_{K}=\Ad^{K}\oplus\chi$, where  $\Ad^{G}$ and 
$\Ad^{K}$ are the adjoint representations of $G$ and $K$ respectively.  
It follows that $\chi$   is equivalent to the adjoint representation   
of $K$ restricted on $\fr{m}$, i.e.  $\chi(K) = \Ad^{K}\big|_{\fr{m}}$.  
Therefore, the set of  all $G$-invariant symmetric (covariant) 2-tensors on $G/K$  
can be identified with the set of all $\Ad(K)$-invariant symmetric bilinear forms on $\fr{m}$.  
In particular, the set of  $G$-invariant metrics on $G/K$ 
is identified with the set of $\Ad(K)$-invariant inner products  on $\fr{m}$.

Let $\fr{m}=\fr{m}_1\oplus\cdots\oplus\fr{m}_{s}$ be a $(-B)$-orthogonal 
$\Ad(K)$-invariant decomposition of $\fr{m}$ into  its irreducible 
$\Ad(K)$-modules $\fr{m}_{i}$  $(i=1, \ldots, s)$, and assume that 
$\fr{m}_{i}$ are mutually inequivalent  $\Ad(K)$-representations.  
Then the space of $G$-invariant symmetric (covariant) 2-tensors on 
$M=G/K$ is given by 
\begin{equation}\label{Ri}
\{y_1\cdot (-B)|_{\fr{m}_1}+\cdots+y_s\cdot (-B)|_{\fr{m}_s} : y_1, \ldots, y_s \in \mathbb{R}\},
\end{equation}
and the space of $G$-invariant Riemannian metrics on $M$ is given by 
\begin{equation}\label{Inva}
\{x_1\cdot (-B)|_{\fr{m}_1}+\cdots+x_s\cdot (-B)|_{\fr{m}_s} : x_1>0, \ldots, x_s>0\}.
\end{equation}
According to (\ref{Ri}), the Ricci tensor $\Ric_{g}$ of a $G$-invariant metric $g$ on $M$
  is given by
  \[
\Ric_{g} = r_1\cdot (-B)|_{\fr{m}_1}+\cdots+r_s\cdot (-B)|_{\fr{m}_s}. 
\]
Here $r_{1}, \ldots, r_{s}$  are the components of the Ricci 
tensor on each $\fr{m}_{i}$.

   We now recall the notation $[ijk]$  from \cite{Wa2}.   
   Let $\{X_{\al}\}$ be a $(-B)$-orthogonal basis adapted to the 
   decomposition of $\fr{m}$, that is  $X_{\al}\in \fr{m}_{i}$ for some $i$, 
   and $\al<\be$ if $i<j$ (with $X_{\al}\in \fr{m}_{i}$ and $X_{\be}\in\fr{m}_{j}$).  
   Set $A_{\al\be}^{\gamma}=B([X_{\al}, X_{\be}], X_{\gamma})$ so 
   that $[X_{\al}, X_{\be}]_{\fr{m}}=\sum_{\gamma}A_{\al\be}^{\gamma}X_{\gamma}$, 
   and   $[ijk]=\sum(A_{\al\be}^{\gamma})^{2}$, where the sum is taken over all 
   indices $\al, \be, \gamma$ with $X_{\al}\in \fr{m}_{i}, X_{\be}\in\fr{m}_{j}, X_{\gamma}\in\fr{m}_{k}$
   (where $[ \ , \ ]_{\fr{m}}$ denotes the $\fr{m}$-component).  Then $[ijk]$ is nonnegative, symmetric in all 
   three entries, and independent of the $(-B)$-orthogonal bases choosen for 
   $\fr{m}_{i}, \fr{m}_{j}$ and $\fr{m}_{k}$ (but it depends on the choise 
   of the decomposition of $\fr{m}$).   
   
 \begin{prop}\label{Ricc}{\textnormal{(\cite{SP})}} 
Let $M=G/K$ be a reductive homogeneous space of a compact semisimple Lie group $G$ and let
 $\fr{m}=\bigoplus_{i=1}^{s}\fr{m}_{i}$ be a decomposition of $\fr{m}$ into mutually
 inequivalent irreducible $\Ad(K)$-submodules.  
Then the components $r_{1}, \ldots, r_{s}$ of the Ricci tensor of a $G$-invariant 
metric  (\ref{Inva}) on $M$ are given by
   \begin{equation}\label{ricc}
   r_{k}=\frac{1}{2x_{k}}+\frac{1}{4d_{k}}\sum_{i, j}\frac{x_{k}}{x_{i}x_{j}}[ijk]-\frac{1}{2d_{k}}\sum_{i, j}\frac{x_{j}}{x_{k}x_{i}}[kij], \qquad (k=1, \ldots, s).
 \end{equation}
\end{prop}

\markboth{Andreas Arvanitoyeorgos and Ioannis Chrysikos}{Invariant Einstein metrics  on  flag manifolds with four isotropy summands}
\section{Generalized flag manifolds}
\markboth{Andreas Arvanitoyeorgos and Ioannis Chrysikos}{Invariant Einstein metrics  on  flag manifolds with four isotropy summands}

  In this section our main goal is to classify  generalized flag manifolds $M=G/K$ 
  of a compact connected  simple Lie group $G$ with four 
  irreducibles components.  First we recall the Lie-theoretic
  description of $M=G/K$.

\subsection{Description of flag manifolds in terms of painted Dynkin diagrams} 
Let   $\fr{g}$ be the  Lie algebra of $G$ and  $\fr{g}^{\mathbb{C}}$ 
be its complexification. We choose a maximal torus  $T$ in $G$, 
and let $\fr{h}$ be the Lie algebra of $T$.  The complexification   
$\fr{h}^{\mathbb{C}}$    is a Cartan subalgebra of $\fr{g}^{\mathbb{C}}$.  
We denote by $R\subset(\fr{h}^{\mathbb{C}})^*$ the root system of $\fr{g}^{\mathbb{C}}$ 
relative to $\fr{h}^{\mathbb{C}}$,  and we consider the root space decomposition
$\fr{g}^{\mathbb{C}}=\fr{h}^{\mathbb{C}}\oplus\sum_{\al\in R}\fr{g}_{\al}^{\mathbb{C}}$,
where  $\fr{g}_{\al}^{\mathbb{C}}=\mathbb{C}E_{\al}$ are  the 1-dimensional 
(complex) root spaces.  Since $\fr{g}^{\mathbb{C}}$ is semisimple, the Killing form $B$ of $\fr{g}^{\mathbb{C}}$
is non-degenerate, and we establish a natural 
isomorphism  between $\fr{h}^{\mathbb{C}}$ 
and the dual space $(\fr{h}^{\mathbb{C}})^{*}$ as follows:  For every $\al\in (\fr{h}^{\mathbb{C}})^{*}$ we define 
 $H_{\al}\in\fr{h}^{\mathbb{C}}$   
by the equation $B(H, H_{\al})=\al(H),$ \ for all $H\in\fr{h}^{\mathbb{C}}$.  
   
  Choose  a fundamental system 
$\Pi=\{\al_{1}, \ldots, \al_{\ell}\}$ \ $(\dim\fr{h}^{\mathbb{C}}=\ell)$  
of $R$ and let $\{\Lambda_1, \ldots, \Lambda_{\ell}\}$ 
be the fundamental weights of $\fr{g}^{\mathbb{C}}$ 
corresponding to $\Pi$, that is 
     $
     \displaystyle\frac{2(\Lambda_{i}, \al_{j})}{(\al_{j}, \al_{j})}=\delta_{ij},  \ (1\leq i, j\leq \ell).
     $
  We fix a lexicographic ordering on $(\fr{h}^{\mathbb{C}})^*$,  
and let $R^{+}$ be the set of positive roots with respect to $\Pi$.   
Choose a subset $\Pi_{K}$ of $\Pi$ and set $
\Pi_{M}=\Pi\backslash \Pi_{K}=\{\al_{i_{1}}, \ldots, \al_{i_{r}}\},  \ (1\leq i_{1}\leq \cdots\leq i_{r}\leq \ell).
$
Let
\begin{equation}\label{Eq: roots}
R_{K}=R\cap\left\langle\Pi_{K} \right\rangle, \quad R_{K}^{+}=R^{+}\cap\left\langle\Pi_{K}\right\rangle, \quad  R_{M}=R\backslash R_{K}, \quad R_{M}^{+}=R^{+}\backslash R_{K}^{+},
\end{equation}
where   $\left\langle\Pi_{K}\right\rangle$   
denotes the set of roots generated by $\Pi_{K}$.   
The  set $R_{M}$ is such that $R=R_{K}\sqcup R_{M}$ (disjoint union) and
is called the set of {\it complementary roots} of $M$.  The subalgebra  
 \begin{equation}\label{parab}
 \fr{p}=\fr{h}^{\mathbb{C}}\oplus\sum_{\al\in R_{K}}\fr{g}_{\al}^{\mathbb{C}}\oplus\sum_{\al\in R_{M}^{+}}\fr{g}^{\mathbb{C}}_{\al}
\end{equation}
is a parabolic subalgebra of $\fr{g}^{\mathbb{C}}$, since it contains the Borel subalgebra 
$\fr{b}=\fr{h}^{\mathbb{C}}\oplus\sum_{\al\in R^+}\fr{g}^{\mathbb{C}}_{\al}$.  
It is well known that there is a one-to-one correspondence between parabolic 
subalgebras of  $\fr{g}^{\mathbb{C}}$ and pairs  
$(\Pi, \Pi_{K})$  (e.g. \cite[p.~187]{GOV}, \cite{Ale1}).  

Let $G^{\mathbb{C}}$ be a simply connected complex simple Lie 
group whose Lie algebra is $\fr{g}^{\mathbb{C}}$, and let  $P$ be
the parabolic subgroup of $G^{\mathbb{C}}$ generated 
by $\fr{p}$. Then the complex  
homogeneous space $G^{\mathbb{C}}/P$  is   compact and 
simply connected, and   $G$ acts transitively on $G^{\mathbb{C}}/P$ (cf. \cite{Akh}).
The intersection $K=G\cap P$ is a connected and closed subgroup of $G$.  
The canonical embedding $G\to G^{\mathbb{C}}$ gives a diffeomorphism 
of the compact homogeneous space $M=G/K$ with $G^{\mathbb{C}}/P$,  
 and $M$ admits a $G$-invariant 
K\"ahler metric (cf. \cite{Brl}). The   homogeneous space $M$ 
is called {\it generalized flag manifold} 
(or  {\it K\"ahler $C$-space}).
            
 We choose a Weyl basis   $E_{\al}\in\fr{g}_{\al}^{\mathbb{C}}$ $(\al\in R)$ 
  of $\fr{g}^{\mathbb{C}}$ with $B(E_{\al}, E_{-\al})=-1$, $[E_{\al}, E_{-\al}]=-H_{\al}$, and
  \begin{equation}\label{str}
 [E_{\al}, E_{\be}]=
 \left\{
\begin{array}{ll}
  N_{\al, \be}E_{\al+\be}  & \mbox{if} \ \ \al, \be,\al+\be\in R \\
  0 & \mbox{if} \ \ \al, \be\in R,  \al+\be\notin R,
\end{array} \right.
\end{equation}
where the constants $N_{\al, \be}$ are such that $ N_{\al, \be}=N_{-\al, -\be}\in \mathbb{R}$ $(\al, \be, \al+\be\in R)$.
Then we obtain that (cf. \cite{Hel}) $\fr{g}=\fr{h}\oplus \sum_{\al\in R^{+}}(\mathbb{R}A_{\al}+\mathbb{R}B_{\al})$, 
where $A_{\al}=E_{\al}+E_{-\al}, B_{\al}=\sqrt{-1}(E_{\al}-E_{-\al}), \al\in R^{+}$.  
The intersection $\fr{k}=\fr{p}\cap\fr{g}\subset\fr{g}$    
is the Lie subalgebra corresponding to $K$, 
 and is given by $\fr{k}=\fr{h}\oplus\sum_{\al\in R_{K}^{+}}(\mathbb{R}A_{\al}+\mathbb{R}B_{\al})$.  
According to (\ref{parab}),  we easily obtain the direct 
decomposition $\fr{p}=\fr{k}^{\mathbb{C}}\oplus\fr{n}$, 
where $\fr{k}^{\mathbb{C}}=\fr{h}^{\mathbb{C}}\oplus\sum_{\al\in R_{K}}\fr{g}_{\al}^{\mathbb{C}}$ and $\fr{n}=\sum_{\al\in R_{M}^{+}}\fr{g}^{\mathbb{C}}_{\al}$. 
The complexification $\fr{k}^{\mathbb{C}}$ of $\fr{k}$ is   
a maximal reductive subalgebra of $\fr{g}^{\mathbb{C}}$,  
and $\fr{n}$ is the maximal nilpotent ideal in $\fr{p}$ (nilradical).  
Morever,  $\fr{k}^{\mathbb{C}}$ (as  a reductive 
complex  subalgebra)   admits the decomposition $\fr{k}^{\mathbb{C}}=
\fr{z}(\fr{k}^{\mathbb{C}})\oplus\fr{k}^{\mathbb{C}}_{s}$,  
where $\fr{k}^{\mathbb{C}}_{s}=[\fr{k}^{\mathbb{C}}, \fr{k}^{\mathbb{C}}]$ 
denotes the semisimple part of $\fr{k}^{\mathbb{C}}$, 
and $\fr{z}(\fr{k}^{\mathbb{C}})$ its center.

 With respect to the negative of the Killing form $B$ we consider 
 the reductive decomposition $\fr{g}=\fr{k}\oplus\fr{m}$ of $\fr{g}$ 
 with $[\fr{k}, \fr{m}]\subset\fr{m}$. Then,  according to (\ref{Eq: roots}) we obtain 
 that
\begin{equation}\label{submodule}
\fr{m}=\sum_{\al\in R_{M}^{+}}(\mathbb{R}A_{\al}+\mathbb{R}B_{\al}).
\end{equation}
We define a tensor $J_{o}$ on $\fr{m}\cong T_{o}M$ given by 
$J_{o}A_{\al}=B_{\al}$, $J_{o}B_{\al}=-A_{\al}$, $\al\in R_{M}^{+}$. 
In this way we provide $M=G/K$ with a $G$-invariant complex structure $J$, which coincides 
with the {\it canonical complex structure} induced from the complex homogeneous 
space $G^{\mathbb{C}}/P$. We can extend this complex structure  
on the complexification of $\fr{m}$,
given by  $\fr{m}^{\mathbb{C}}=\sum_{\al\in R_{M}}\fr{g}_{\al}^{\mathbb{C}}$. The set $\{E_{\al} : \al\in R_{M}\}$ forms a basis of $\fr{m}^{\mathbb{C}}$ (\cite{AP}). 
 Note   that since $[\fr{k}, \fr{m}]\subset\fr{m}$, the set $R_{M}$ is invariant under $R_{K}$ in the sense that if $\al\in R_{M}, \be\in R_{K},$ and $\al+\be\in R$, then  $\al+\be\in R_{M}$. 

An {\it invariant ordering} $R_{M}^{+}$ in $R_{M}$ is a choise of a subset $R_{M}^{+}$
such that 
 \[
\begin{tabular}{ll}
$(i)$ & $R=R_{K}\sqcup R_{M}^+\sqcup R_{M}^{-}, \ \ \mbox{where} \ \ R_{M}^{-}=\{-\al: \al\in R_{M}^{+}\},$\\
$(ii)$ & $\mbox{If} \ \ \al\in R_{K}\sqcup R_{M}^{+}, \ \be\in R_{M}^{+} \ \ 
\mbox{and} \ \al+\be\in R, \ \ \mbox{then} \ \al+\be\in R_{M}^{+}.$ 
\end{tabular}
\]
We say that $\al>\be$ if and only if $\al-\be\in R_{M}^{+}$.  
 Note that the choise of $R_{M}^{+}=R^{+}\backslash R_{K}^+$ determined by (\ref{Eq: roots}),   
 satisfies conditions $(i)$ and $(ii)$.  We will call this ordering {\it natural invariant ordering}.

 Generalized flag manifolds $M=G/K$ of a compact connected simple Lie group $G$
 can be classified by using the Dynkin diagram of $G$, as follows:  
   Let   $\Gamma=\Gamma(\Pi)$ be the Dynkin diagram of the set of simple roots  $\Pi$ for
   the root system $R$.  
    By painting black the nodes of $\Gamma$ corresponding to 
    the simple roots   of $\Pi_{M}$ we obtain the {\it painted Dynkin diagram}  
    of $M=G/K$.  The subdiagram of white nodes with the connecting lines between them 
    determines the semisimple part of the Lie algebra of $K$, and  each black 
    node  gives rise to one $\fr{u}(1)$-summand. Thus  the painted Dynkin diagram determines 
    the reductive decomposition 
    and the space $M$ completely.  Note that if we paint all nodes black, 
    that is $\Pi_{K}=\emptyset, \Pi=\Pi_{M}$, 
then we will obtain a full flag manifold.
    \begin{example}\textnormal{
     Let  $G=F_{4}$ and set  $\Pi_{M}=\{\al_4\}$. 
    This determines the painted Dynkin diagram $
 \begin{picture}(62,10)(85, 4)
\put(87,5){\circle{4}}
\put(87, 11.5){\makebox(0,0){$\al_1$}}
\put(89,5){\line(1,0){14}}
\put(105,5){\circle{4}}
\put(105, 11.5){\makebox(0,0){$\al_2$}}
\put(107, 6.1){\line(1,0){12.3}}
\put(107, 4.5){\line(1,0){12.3}}
\put(116, 3.2){\scriptsize $>$}
\put(123.5,5){\circle{4}}
\put(123.5, 11.5){\makebox(0,0){$\al_3$}}
\put(126,5){\line(1,0){16}}
\put(142,5){\circle*{4.4}}
\put(142, 11.5){\makebox(0,0){$\al_4$}}
\end{picture}.
$
 The set $\Pi_{K}=\{\al_1, \al_2, \al_3\}$ is a system of simple roots  for  $\fr{k}^{\mathbb{C}}_{s}=\fr{so}(7, \mathbb{C})$
    and thus $\fr{k}^{\mathbb{C}}=\fr{so}(7, \mathbb{C})\oplus \fr{u}(1)$, so  $M=G/K=F_{4}/SO(7)\times U(1)$.}  
  \end{example}
  
\begin{prop}  \label{Eq: DD}  \textnormal{(\cite{Ale1})}
Different painted Dynkin diagrams $\Gamma$ and $\Gamma_1$ 
(except for the case of $D_{\ell}=SO(2\ell)$) define isomorphic flag 
manifolds $G/K$ and $G/K'$, i.e. there is an automorphism $\phi\in\Aut(G)$
such that $\phi(K)=K'$, if the subdiagrams $\Gamma'$ and $\Gamma'_1$ of
white roots corresponding to $\Pi_K$ and $\Pi_{K}'$ are isomorphic.
\end{prop}
 
By using Proposition \ref{Eq: DD}, it is possible to give a complete list of
all flag manifolds $G/K$, where $G$ is either a classical or an 
exceptional Lie group (up to isomorphism). 
For  the classification of generalized flag manifolds in terms of  
painted Dynkin diagrams,  we refer to \cite{AA} and \cite{Bor}.

 \subsection{$\fr{t}$-roots and irreducible submodules.}

  An important invariant of a generalized flag manifold $M=G/K$ 
  is the set $R_{\fr{t}}$ of $\fr{t}$-roots. The notion of $\fr{t}$-roots was first
  introduced by  J.~Siebenthal in \cite{Sie}.  In the present form they are due to D.~V.~Alekseevsky (\cite{AP}, \cite{Ale1}). 
  Their importance arises from the fact
  that the knowledge of  $R_{\fr{t}}$
   gives us   crucial information  about  the decomposition of  the isotropy representation of $M$.

   For convenience, we fix  a system of simple roots $\Pi=\{\al_1,\ldots, \al_r, \phi_1, \ldots, \phi_k\}$   
   of $R$,  so that $\Pi_{K}=\{\phi_1, \ldots, \phi_k\}$ is a basis of the root system 
   $R_{K}$ and $\Pi_{M}=\Pi\backslash \Pi_{K}=\{\al_{1}, \ldots, \al_{r}\}$ $(r+k=\ell)$.  
   We consider the decomposition $R=R_{K}\sqcup R_{M}$,   and
    we define the set 
  \[ \fr{t} =\fr{z}(\fr{k}^{\mathbb{C}}) \cap i\fr{h}=\{X\in\fr{h} : \phi(X)=0, \ \mbox{for all} \ \phi\in R_{K}\},\]
where     $\fr{h}$ is the real $\ad$-diagonal  subalgebra  $\fr{h} = \fr{h}^{\mathbb{C}}\cap i\fr{k}$ (\cite{Ale1}, \cite{Arv}).  
The space $\fr{t}$ is a real form of the center $\fr{z}(\fr{k}^{\mathbb{C}})$, 
and thus $\fr{k}^{\mathbb{C}} =  \fr{t}^{\mathbb{C}}\oplus \fr{k}^{\mathbb{C}}_{s}$.  
The fundamental weights  $\Lambda_{1}, \ldots, \Lambda_{r}$  corresponding to the 
simple roots of $\Pi_{M}$ form a basis of the space $\fr{t}^{*}$ 
(isomorphic to $\fr{t}$ as vector space via the Killing form), 
thus  $\dim\fr{t}^*=\dim \fr{t}=r$.


Let $\fr{h}^*=\Span_{\mathbb{R}}\{\al :\al\in R\}$ and $\fr{t}^*$ 
be the dual spaces of $\fr{h}$ and $\fr{t}$ respectively.
Consider  the linear restriction map  $\kappa : \fr{h}^{*}\to \fr{t}^{*}$ defined by $\kappa(\al)=\al|_{\fr{t}}$,
and set  $R_{\fr{t}} = \kappa(R)=\kappa(R_{M})$.  Note that $\kappa(R_K)=0$ and   $\kappa(0)=0$.
\begin{definition}
The elements of $R_{\fr{t}}$ are called {\it $\fr{t}$-roots}.
\end{definition}

 Although  the set $R_{\fr{t}}$ is not
 a root system, it is possible to generalize some known
 notions of  root systems theory.  An element $Y\in\fr{t}$ 
 is called {\it regular} 
 if any $\fr{t}$-root $\kappa(\al) =\xi$ $(\al\in R_{M})$ 
 has non zero value on $Y$, i.e. $\xi(Y)\neq 0$.  
 A regular element defines an ordering
  in $\fr{t}^*$. This means that we obtain a polarization on $R_{\fr{t}}$, 
  that is $R_{\fr{t}}=R_{\fr{t}}^{+}\sqcup  R_{\fr{t}}^{-}$, 
  where $R_{\fr{t}}^{+}=\{\xi\in R_{\fr{t}} : \xi(Y)>0\}$
  and $R_{\fr{t}}^{-}=\{\xi\in R_{\fr{t}} : \xi(Y)<0\}$.  
  The $\fr{t}$-toots $\xi\in R_{\fr{t}}^{+}$ (resp. $\xi\in R_{\fr{t}}^{-}$) will be called {\it positive}
  (resp. {\it negative}). Since $R_{\fr{t}}=\kappa(R_{M})$ it follows that
   $R_{\fr{t}}^{+}=\kappa(R_{M}^{+})$.  Note that a regular element $Y\in\fr{t}$  
   does not lie on any of the hyperplanes orthogonal to $\fr{t}$-roots,
  that is $ Y\in \fr{t} \ \backslash \bigcup _{\xi\in R_{\fr{t}}} \fr{t}_{\xi}$,  where $\fr{t}_{\xi}=\{X\in\fr{t} : \xi(X)=0\}$.
  We will denote by $\fr{t}_{\mbox{reg}}$ the open dense subset 
 $\fr{t} \ \backslash \bigcup _{\xi\in R_{\fr{t}}} \fr{t}_{\xi}$ 
 of all regular elements in $\fr{t}$.
   A connected component $C$ of the set 
  $\fr{t}_{\mbox{reg}}$ 
  is called a {\it $\fr{t}$-chamber}, generalizing the known   Weyl chambers.
 The hyperplanes  $\fr{t}_{\xi}$  are called  the {\it walls} in $\fr{t}$.
  Clearly, in order to specify a  $\fr{t}$-chamber, it suffices to specify 
  on which side of a hyperplane $\fr{t}_{\xi}$ the $\fr{t}$-chamber lies, 
  for every hyperplane $\fr{t}_{\xi}$. 
  Thus, a $\fr{t}$-chamber is defined by a 
  system   of  inequalities of the form $\pm \xi(X)>0$. The number of these inequalities must be finite, since there
  is one inequality for each positive $\fr{t}$-root.
  
  The above description shows that any $\fr{t}$-chamber $C$ defines an ordering in $R_{\fr{t}}$.
  Conversely, given a  polarization  $R_{\fr{t}}=R_{\fr{t}}^{+}\sqcup  R_{\fr{t}}^{-}$  
  we can define the corresponding
  {\it positive $\fr{t}$-chamber} $C_{+}$ by 
  \[
  C_{+}=\{W\in\fr{t} : \xi(W)>0 \ \mbox{for all} \ \xi\in R_{\fr{t}}^{+}\}.
  \]
  Thus we obtain a bijection between the set of all polarizations of $R_{\fr{t}}$ and the set of $\fr{t}$-chambers.
  In particular, one can show that invariant orderings $R_{M}^{+}$ in $R_{M}$,
  are in one-to-one correspondence with $\fr{t}$-chambers (\cite[p.~621]{Bor}). 
  Namely,  a complementary root $\al\in R_{M}$ is positive with respect some given invariant ordering, 
  if and only if it takes strictly positive values on the corresponding $\fr{t}$-chamber.

   A fundamental result about $\fr{t}$-roots  is the following:

 \begin{prop} \label{isotropy} 
  \textnormal{(\cite{Sie}, \cite{AP})} There exists a one-to-one corespondence 
 between $\fr{t}$-roots and complex irreducible $\ad(\fr{k}^{\mathbb{C}})$-submodules 
 $\fr{m}_{\xi}$ of $\fr{m}^{\mathbb{C}}$.  This correspondence is given by  
\[R_{\fr{t}}\ni\xi \ \ \leftrightarrow \ \ \fr{m}_{\xi} =\sum_{\al\in R_{M}: \kappa(\al)=\xi}\mathbb{C}E_{\al}.
\] 	
Thus $\fr{m}^{\mathbb{C}} = \sum_{\xi\in R_{\fr{t}}} \fr{m}_{\xi}.$  Moreover, these submodules are inequivalent as $\ad(\fr{k}^{\mathbb{C}})$-modules.		 		  
  \end{prop}

  Since  the complex conjugation  
$\tau : \fr{g}^{\mathbb{C}}\to\fr{g}^{\mathbb{C}}, \ X+iY\mapsto X-iY \ (X, Y\in\fr{g})$ 
of $\fr{g}^{\mathbb{C}}$ with respect to  the compact 
real form $\fr{g}$ interchanges the root spaces, i.e.  
  $\tau(E_{\al})=E_{-\al}$ and $\tau(E_{-\al})=E_{\al}$, 
 a decomposition of the real $\ad(\fr{k})$-module $\fr{m}=(\fr{m}^{\mathbb{C}})^{\tau}$ 
 into  real  irreducible submodules is given by
\begin{equation}\label{tangent}
 \fr{m} = \sum_{\xi\in R_{\fr{t}}^{+}=\kappa(R_{M}^{+})}(\fr{m}_{\xi}\oplus  \fr{m}_{-\xi})^{\tau},
 \end{equation}
  where  $\fr{n}^{\tau}$   denotes the set of fixed points of the 
    complex conjugation $\tau$ in  a vector subspace $\fr{n}\subset \fr{g}^{\mathbb{C}}$.    
If, for simplicity,  we  set  $R_{\fr{t}}^{+}=\{\xi_1, \ldots, \xi_s\}$, then according to (\ref{submodule})  each real irreducible $\ad(\fr{k})$-submodule $\fr{m}_{i}=(\fr{m}_{\xi_{i}}\oplus  \fr{m}_{-\xi_{i}})^{\tau}$ $(1\leq i\leq s)$
corresponding to the positive $\fr{t}$-root $\xi_i$,   is given by
\begin{equation}\label{bas}
\fr{m}_{i}=\sum_{\al\in R_{M}^+ \ :\ \kappa(\al)=\xi_{i}}(\mathbb{R}A_{\al}+\mathbb{R}B_{\al}).
\end{equation}

\begin{remark}\label{rem}
\textnormal{
An immediate consequence of (\ref{bas}) is that, the (real) dimension of each irreducible $\ad(\fr{k})$-submodule $\fr{m}_i$ 
which corresponds to the positive $\fr{t}$-root $\xi_i$, is equal to
the cardinality of the set $\{E_{\pm\al} : \kappa(\pm\al)=\pm\xi_{i}\}$.}
\end{remark}


\begin{definition}
A $\fr{t}$-root is called simple if is not a sum of two positive $\fr{t}$-roots.
\end{definition}

 The set $\Pi_{\fr{t}}$ of all simple $\fr{t}$-roots  is called a {\it $\fr{t}$-base} and   
  it is a basis of $\fr{t}^*$, in the sense that any $\fr{t}$-root can be written 
  as a linear combination of its elements with integer coefficients of the same sign.
 As we will see in the next proposition, a $\fr{t}$-base $\Pi_{\fr{t}}$ is obtained 
 by restricting the simple roots of $\Pi_{M}$ on $\fr{t}$.  We provide a proof of this fact, as it is
  not given in the literature.
  
\begin{prop} \label{simple}
Let $\Pi_{M}=\Pi\backslash \Pi_{K}=\{\al_{1}, \ldots, \al_{r}\}$.  
Then   the set $\{\overline{\al}_{i}=\al_{i}|_{\fr{t}} : \al_{i}\in \Pi_{M}\}$ 
is a $\fr{t}$-base of $\fr{t}^*$.
\end{prop}

 \begin{proof} 
 It suffices to prove  that the set $\{\overline{\al}_{i}=\al_{i}|_{\fr{t}} :  \al_{i}\in \Pi_{M}\}$ 
 consists of $r=\dim\fr{t}^*$  linearly independent simple $\fr{t}$-roots (co-vectors in $\fr{t}^*$).
Note that the linear map $\kappa :\fr{h}^{*}\to\fr{t}^*$ is not   an injective map in general,
that is different complementary roots could be mapped to the same $\fr{t}$-root.
Indeed, 
\[
{\rm Ker}\ \kappa = \{\al\in\fr{h}^* : \kappa(\al)=0\} = R_{K}\cup\{0\}.
\]
 However,  it is $R_{K}=R\cap\left\langle\Pi_{K} \right\rangle$ and $\Pi_{M}=\Pi\backslash\Pi_{K}$, thus $\Pi_{M}\cap R_{K}=\emptyset$, 
 and  $\kappa$ always maps the  simple roots $\al_{i}\in\Pi_{M}$ $(i=1, \ldots, r)$ into different $\fr{t}$-roots $\overline{\al}_{i}=\al_{i}|_{\fr{t}}$,
 i.e.  $\overline{\al}_{i}\neq\overline{\al}_{j}$  for all $1\leq i\neq j\leq r$.   

  We set  $\Pi_{\fr{t}}=\{\overline{\al}_{1}, \ldots, \overline{\al}_{r}\}$, and   
 let $\mu_1, \ldots, \mu_r$ be  real numbers such that  $\mu_1\overline{\al}_1+\cdots+\mu_r\overline{\al}_r=0$.  By  the definition of $\overline{\al}_i$ and the linearity of $\kappa$ we obtain  that
 $\kappa(\mu_1{\al}_1+\cdots+\mu_r{\al}_r)=0$, which implies that $\mu_1{\al}_1+\cdots+\mu_r{\al}_r\in R_{K},$ 
 or $\mu_1{\al}_1+\cdots+\mu_r{\al}_r=0$.  
 But ${\al}_{1}, \ldots, {\al}_{r}$ belong to $\Pi_M$, and a linear combination of these simple roots
 can not be a root of $R_K$.  So $\mu_1{\al}_1+\cdots+\mu_r{\al}_r=0$ and since $\Pi_{M}\subset\Pi$, we conclude that $\mu_1=\mu_2=\cdots =\mu_r=0$.
  Therefore $\Pi_{\fr{t}}$ consists of  $r$ linear independent $\fr{t}$-roots.

  In order to prove that $\overline{\al}_{i}$ $(i=1, \ldots, r)$  are simple 
  $\fr{t}$-roots,  we need to show that every $\overline{\al}_{i}$ can  
  not be expressed as a sum of two positive $\fr{t}$-roots.   
  Assume on the contrary that $\overline{\al}_{i}$ is not  simple, so there exist
  $\xi, \zeta\in R^{+}_{\fr{t}}$ such that $\overline{\al}_{i}=\xi+\zeta$.  
  Without loss of generality  we may take  $\zeta=\overline{\al}_{j}$.   
  But then $\overline{\al}_{i}-\overline{\al}_{j}=\xi$, 
  or equivalently   $\kappa(\al_{i}-\al_{j})=\xi\in R_{\fr{t}}^{+}$.  
  But this is impossible, because $\al_{i}, \al_{j}$ are simple roots of  $R$, 
  and their difference is never a root of $\fr{g}^{\mathbb{C}}$ (cf. \cite[p.~458]{Hel}).
   \end{proof}
     
Proposition \ref{simple}  provides  us with a useful
 method to determine the $\fr{t}$-roots, as follows: 
 Fix a positive root $\al\in R^+$ and let $\al=\sum_{i=1}^{r}f_{i}\al_{i}+\sum_{j=1}^{k}g_{j}\phi_{j}$ 
 be  its expression in terms of simple roots, with respect to $\Pi$.  
 The  coefficients $f_{i}, g_{j}$ are such 
 that $0\leq f_{i}\leq m_{i}$ and $0\leq g_{j}\leq m_{j}$, where $m_i, m_j$ are determined by the {\it highest 
 root} $\tilde{a}=\sum_{i=1}^{\ell}m_{i}\al_{i}$ $( m_{i}\in\mathbb{Z}^+)$  in $R$.  
 Then  
 
  
  \begin{eqnarray}\label{troots}
  \kappa(\al) = \kappa(\sum_{i=1}^{r}f_{i}\al_{i}+\sum_{j=1}^{k}g_{j}\phi_{j}) = \kappa(\sum_{i=1}^{r}f_{i}\al_{i})  = \sum_{i=1}^{r}f_{i}\kappa(\al_{i})
  =f_{1}\overline{\al}_{1}+\cdots+f_{r}\overline{\al}_{r}.
  \end{eqnarray}
  
 By using the expressions of  the complementary roots in terms of simple roots,
  and applying formula (\ref{troots}), we can easily obtain 
 the set $R_{\fr{t}}$.
 This method was first applied in \cite{AA} 
 for certain flag manifolds of exceptional Lie groups.
     In the present work we will use it also for flag manifolds of classical Lie groups.

 \subsection{Classification of flag manifolds with four isotropy summands.} 
 The aim here is to classify all generalized flag manifolds $M=G/K$ 
 whose isotropy representation decomposes into four irreducible submodules. 
 Recall that the {\it height} of a simple root $\al_{i}$ is the positive integer  $m_{i}$ so that
 $\tilde{a}=\sum_{i=1}^{\ell}m_{i}\al_{i}$.
  We define the   function $ \Hgt : \Pi\to\mathbb{Z}$, $\Hgt(\al_{i})=m_{i}$. 
  In \cite{Chry2}  the authors classified all  
  generalized  flag manifolds $M=G/K$
   with two isotropy summands.  This was done by painting black in the 
   Dynkin diagram 
   $\Gamma(\Pi)$ of $G$ a simple root of height 2, that is $\Pi_{M}=\Pi\setminus\Pi_{K}=\{\al_{i} :\Hgt(\al_{i})=2\}$.  Also in  
    \cite{Kim} Kimura 
 obtained all flag manifolds with three isotropy summands, 
 by setting $\Pi_{M}= \{\al_{i} :  \Hgt(\al_{i})=3\}$, or  $\Pi_{M}=\{\al_{i}, \al_{j} :  \Hgt(\al_{i})=\Hgt(\al_{j})=1\}$.

  It will be shown that pairs $(\Pi, \Pi_{K})$ for generalized flag manifolds with four isotropy 
   summands, are divided into
  two different types as follows:
  \medskip
\begin{center}
 \begin{tabular}{r|l}
     Type  & $(\Pi, \Pi_{K})$\\
 \hline  
 I & $ \Pi\setminus\Pi_{K}=\{\al_{i} :  \Hgt(\al_{i})=4\}$ \\
 \hline 
  II & $\Pi\setminus\Pi_{K}=\{\al_{i}, \al_{j} : \Hgt(\al_{i})=1,  \Hgt(\al_{j})=2\}$\\ 
 \hline
  \end{tabular}
 \end{center}
 \medskip
 
 Pairs $(\Pi, \Pi_{K})$ of  Type I   always define a flag 
manifold with four isotropy summands.  However, this is not always 
true for pairs  $(\Pi, \Pi_{K})$ of  Type II.  These may define
flag manifolds with four or five isotropy summands.  This 
 depends on the form of the complementary roots.
 
 \begin{prop}\label{classification}
Let  $G$ be a compact and connected simple Lie group with Dynkin diagram $\Gamma=\Gamma(\Pi)$, 
where $\Pi$ is a system of simple roots of $G$. 
 Let $M=G/K$ be a generalized flag manifold corresponding to one of the
 pairs $(\Pi, \Pi_{K})$ of Type  I or Type II, presested in  Tables  3 and 4 respectively.  
  Then  the isotropy representation of $M$
  decomposes into four inequivalent irreducible $\ad(\fr{k})$-submodules.   
\end{prop}
 
 \medskip
\begin{center}
  \small{
\begin{tabular}{|c|c|c|}
  \hline
 \begin{picture}(20,30)(0,0)
\put(10, 12){\makebox(0,0){$G$}}\end{picture}
 &  \begin{picture}(30,30)(0,0)\put(10, 12){
\makebox(0,0){$( \Pi, \Pi_K ) $ \ of \ Type \ I}}\end{picture}
& \begin{picture}(30,30)(0,0)\put(15, 12){
\makebox(0,0){$K$}}\end{picture}
   
\\
  \hline \hline
  

  \begin{picture}(15,30)(0,0)
\put(10, 13){\makebox(0,0){$F^{}_4$}}\end{picture}
& 

\begin{picture}(160,25)(38, -3)

\put(87,10){\circle{4 }}
\put(87,18){\makebox(0,0){$\al_1$}}
\put(87,2){\makebox(0,0){2}}
\put(89,10.5){\line(1,0){14}}
\put(105,10){\circle{4 }}
\put(105,18){\makebox(0,0){$\al_2$}}
\put(105,2){\makebox(0,0){3}}
\put(107, 11.3){\line(1,0){12.4}}
\put(107, 9.1){\line(1,0){12.6}}
\put(117.6, 8){\scriptsize $>$}
\put(124.5,10){\circle*{4 }}
\put(124.5, 18){\makebox(0,0){$\al_3$}}
\put(124.5, 2){\makebox(0,0){4}}
\put(126,10.5){\line(1,0){16}}
\put(144,10){\circle{4 }}
\put(144,18){\makebox(0,0){$\al_4$}}
\put(144,2){\makebox(0,0){2}}
\end{picture}
  & 
\begin{picture}(110,25)(0,2)\put(50, 10){
\makebox(10,10){$SU(3)\times SU(2)\times U(1)$}}\end{picture}
\\  \hline
\begin{picture}(15, 35)(0,0)
\put(10, 13){\makebox(0,0){$E^{}_7$}}\end{picture}

&
\begin{picture}(160,25)(-25, -10)

\put(15, 9.5){\circle{4 }}
\put(15, 18){\makebox(0,0){$\al_1$}}
\put(15,2){\makebox(0,0){1}}
\put(17, 10){\line(1,0){14}}
\put(33.5, 9.5){\circle{4 }}
\put(33.5, 18){\makebox(0,0){$\al_2$}}
\put(33.5,2){\makebox(0,0){2}}
\put(35, 10){\line(1,0){13.6}}
 \put(51, 9.5){\circle{4 }}
 \put(51, 18){\makebox(0,0){$\al_3$}}
\put(51,2){\makebox(0,0){3}}
\put(69,-6){\line(0,1){14}}
\put(53,10){\line(1,0){14}}
\put(69,9.5){\circle*{4 }}
\put(69.9, 18){\makebox(0,0){$\al_4$}}
\put(65,2){\makebox(0,0){4}}
\put(69,-8){\circle{4}}
\put(79, -10){\makebox(0,0){$\al_7$}}
\put(62,-10){\makebox(0,0){2}}
\put(71,10){\line(1,0){16}}
\put(89,9.5){\circle{4 }}
\put(89, 18){\makebox(0,0){$\al_5$}}
\put(89,2){\makebox(0,0){3}}
\put(90.7,10){\line(1,0){16}}
\put(109,9.5){\circle{4 }}
\put(109, 18){\makebox(0,0){$\al_6$}}
\put(109, 2){\makebox(0,0){2}}
\end{picture}
  & 
\begin{picture}(140,30)(0,0)\put(60, 10){
\makebox(10,10){$SU(4)\times SU(3)\times SU(2)\times U(1)$}}\end{picture}
 \\  \hline

 
\begin{picture}(15,35)(0,0)
\put(10, 13){\makebox(0,0){${E_{8}}_{(i)}$}}\end{picture}

&
\begin{picture}(160,25)(-5, -10)

\put(15, 9.5){\circle{4 }}
\put(15, 18){\makebox(0,0){$\al_1$}}
\put(15,2){\makebox(0,0){2}}
\put(17, 10){\line(1,0){14}}
\put(33.5, 9.5){\circle{4 }}
\put(33.5, 18){\makebox(0,0){$\al_2$}}
\put(33.5,2){\makebox(0,0){3}}
\put(35, 10){\line(1,0){13.6}}
 \put(51, 9.5){\circle*{4 }}
 \put(51, 18){\makebox(0,0){$\al_3$}}
\put(51,2){\makebox(0,0){4}}
\put(53,10){\line(1,0){14}}
\put(69,9.5){\circle{4 }}
\put(69, 18){\makebox(0,0){$\al_4$}}
\put(69,2){\makebox(0,0){5}}
\put(89,-8){\circle{4}}
\put(99, -9.5){\makebox(0,0){$\al_8$}}
\put(82,-9.5){\makebox(0,0){3}}
\put(89,-6){\line(0,1){14}}
\put(71,10){\line(1,0){16}}
\put(89,9.5){\circle{4 }}
\put(89, 18){\makebox(0,0){$\al_5$}}
\put(84,2){\makebox(0,0){6}}
\put(90.7,10){\line(1,0){16}}
\put(109,9.5){\circle{4 }}
\put(109, 18){\makebox(0,0){$\al_6$}}
\put(109,2){\makebox(0,0){4}}
\put(111,10){\line(1,0){16}}
\put(129,9.5){\circle{4 }}
\put(129, 18){\makebox(0,0){$\al_7$}}
\put(129,2){\makebox(0,0){2}}
\end{picture}
  & 
\begin{picture}(140,30)(0,0)\put(60, 10){
\makebox(10,10){$SO(10)\times SU(3)\times U(1)$}}\end{picture}
 \\  \hline


\begin{picture}(15,35)(0,0)
\put(10, 13){\makebox(0,0){${E_{8}}_{(ii)}$}}\end{picture}

&
\begin{picture}(160,25)(-5, -10)

\put(15, 9.5){\circle{4 }}
\put(15, 18){\makebox(0,0){$\al_1$}}
\put(15,2){\makebox(0,0){2}}
\put(17, 10){\line(1,0){14}}
\put(33.5, 9.5){\circle{4 }}
\put(33.5, 18){\makebox(0,0){$\al_2$}}
\put(33.5,2){\makebox(0,0){3}}
\put(35, 10){\line(1,0){13.6}}
 \put(51, 9.5){\circle{4 }}
 \put(51, 18){\makebox(0,0){$\al_3$}}
\put(51,2){\makebox(0,0){4}}
\put(53,10){\line(1,0){14}}
\put(69,9.5){\circle{4 }}
\put(69, 18){\makebox(0,0){$\al_4$}}
\put(69,2){\makebox(0,0){5}}
\put(89,-8){\circle{4}}
\put(99, -9.5){\makebox(0,0){$\al_8$}}
\put(82,-9.5){\makebox(0,0){3}}
\put(89,-6){\line(0,1){14}}
\put(71,10){\line(1,0){16}}
\put(89,9.5){\circle{4 }}
\put(89, 18){\makebox(0,0){$\al_5$}}
\put(84,2){\makebox(0,0){6}}
\put(90.7,10){\line(1,0){16}}
\put(109,9.5){\circle*{4 }}
\put(109, 18){\makebox(0,0){$\al_6$}}
\put(109,2){\makebox(0,0){4}}
\put(111,10){\line(1,0){16}}
\put(129,9.5){\circle{4 }}
\put(129, 18){\makebox(0,0){$\al_7$}}
\put(129,2){\makebox(0,0){2}}
\end{picture} 
 & 
\begin{picture}(140,30)(0,0)\put(60, 10){
\makebox(10,10){$SU(7)\times SU(2)\times U(1)$}}\end{picture}

  \\  \hline
    \end{tabular}}
 \end{center}
\smallskip

\begin{center}
 {\sc Table 3.} \ Pairs $(\Pi, \Pi_{K})$ of Type I.
\end{center}
  \begin{center}
  \small{
\begin{tabular}{|c|c|c|}
  \hline
 \begin{picture}(20,30)(0,0)
\put(10, 12){\makebox(0,0){$G$}}\end{picture}
 &  \begin{picture}(30,30)(0,0)\put(10, 12){
\makebox(0,0){$( \Pi, \Pi_K ) $ \ of \ Type \ II}}\end{picture}
& \begin{picture}(30,30)(0,0)\put(15, 12){
\makebox(0,0){$K$}}\end{picture}
   
\\
  \hline \hline

\begin{picture}(40,40)(0,5)
\put(20, 25){\makebox(0,0){$SO(2\ell+1)$}}\end{picture}
 &
\begin{picture}(160,40)(-15,-17)
\put(0, 0){\circle*{4}}
\put(0,8.5){\makebox(0,0){$\al_1$}}
\put(0,-8){\makebox(0,0){1}}
\put(2, 0.5){\line(1,0){14}}
\put(18, 0){\circle*{4}}
\put(18,8.5){\makebox(0,0){$\al_2$}}
\put(18,-8){\makebox(0,0){2}}
\put(20, 0.5){\line(1,0){13.5}}
\put(35.5, 0){\circle {4}}
\put(35.5,-8){\makebox(0,0){2}}
 \put(37.5, 0.5){\line(1,0){13}}
\put(58.5, 0){\makebox(0,0){$\ldots$}}
\put(65, 0.5){\line(1,0){13.5}}
 \put(80.6, 0){\circle{4}}

\put(80.6,-8){\makebox(0,0){2}}
   \put(82, 0.5){\line(1,0){14}}
\put(98, 0){\circle{4}}
\put(99,8.5){\makebox(0,0){$\al_{\ell-1}$}}
\put(98,-8){\makebox(0,0){2}}
\put(100, 1.3){\line(1,0){15.5}}
\put(100, -0.9){\line(1,0){15.5}}
\put(113.5, -1.9){\scriptsize $>$}
\put(121, 0){\circle{4}}
 \put(121.8,8.5){\makebox(0,0){$\al_\ell$}}
\put(121,-8){\makebox(0,0){2}}
\end{picture}
 & 
\begin{picture}(110,40)(0,3)
\put(52, 15){\makebox(8,15){$SO(2\ell-3)\times U(1)\times U(1)$}}
 \end{picture} 
  
 \\
  \hline  
  
\begin{picture}(15,45)(0,0)
\put(10, 25){\makebox(0,0){$Sp(\ell)$}}\end{picture}
 &

\begin{picture}(160,45)(-15,-25)
\put(0, 0){\circle{4}}
\put(0,8.5){\makebox(0,0){$\al_1$}}
\put(0,-8){\makebox(0,0){2}}
\put(2, 0.3){\line(1,0){14}}
\put(18, 0){\circle{4}}
\put(18,8.5){\makebox(0,0){$\al_2$}}
\put(18,-8){\makebox(0,0){2}}
\put(20, 0.3){\line(1,0){10}}
\put(40, 0){\makebox(0,0){$\ldots$}}
\put(50, 0.3){\line(1,0){10}}
\put(60, -20){\makebox(0,0){$( 1 \leq p \leq \ell - 1 )$}}
\put(60, 0){\circle*{4.4}}
\put(60,8.5){\makebox(0,0){$\al_p$}}
\put(60,-8){\makebox(0,0){2}}
\put(60, 0.3){\line(1,0){10}}
\put(80, 0){\makebox(0,0){$\ldots$}}
\put(90, 0.3){\line(1,0){10}}
\put(102, 0){\circle{4}}
\put(102,8.5){\makebox(0,0){$\al_{\ell-1}$}}
\put(102,-8){\makebox(0,0){2}}
\put(107.2, 1.2){\line(1,0){14.8}}
\put(107.2, -1){\line(1,0){14.8}}
\put(103.46, -2){\scriptsize $<$}
\put(123.5, 0){\circle*{4}}
\put(124,8.5){\makebox(0,0){$\al_\ell$}}
\put(123,-8){\makebox(0,0){1}}
\end{picture}
 &

\begin{picture}(110,45)(5,15)\put(50, 20){
\makebox(8,35){$U(p)\times U(\ell-p)$}}\end{picture} 
  
  \\
  \hline
  
\begin{picture}(15,45)(0,0)
\put(10, 25){\makebox(0,0){$SO(2\ell)_{(i)}$}}\end{picture}
 &
\begin{picture}(160,40)(-15,-23)
\put(0, 0){\circle*{4}}
\put(0,8.5){\makebox(0,0){$\al_1$}}
\put(0,-8){\makebox(0,0){1}}
\put(2, 0){\line(1,0){14}}
\put(18, 0){\circle*{4}}
\put(18,8.5){\makebox(0,0){$\al_2$}}
\put(18,-8){\makebox(0,0){2}}
\put(20, 0){\line(1,0){13.5}}
\put(35.5, 0){\circle {4}}
\put(35.5,-8){\makebox(0,0){2}}
 \put(37.5, 0){\line(1,0){12}}
\put(59.2, 0){\makebox(0,0){$\ldots$}}
\put(67, 0){\line(1,0){13}}
 \put(82, 0){\circle{4}}
 \put(82,-8){\makebox(0,0){2}}
\put(84, 0){\line(1,0){12}}
\put(100.7, 1){\line(2,1){10}}
\put(100.7, -1){\line(2,-1){10}}
\put(98, 0){\circle{4}}
\put(98,-8){\makebox(0,0){2}}
\put(96.5, 8.5){\makebox(0,0){$\al_{\ell -2}$}}
\put(112.5, -6){\circle{4}}
\put(120, 13){\makebox(0,0){$\al_{\ell-1}$}}
\put(120,5.5){\makebox(0,0){1}}
\put(111, -16){$\al_\ell$}
\put(120,-8){\makebox(0,0){1}}
\put(112.5, 6){\circle{4}}
\end{picture}
 &
\begin{picture}(110,45)(1,20)\put(50, 25){
\makebox(8,35){\shortstack{$SO(2(\ell-2))\times U(1)\times U(1)$}}}\end{picture}
  
   \\
  \hline
  
\begin{picture}(15,45)(0,0)
\put(10, 25){\makebox(0,0){$SO(2\ell)_{(ii)}$}}\end{picture}
 &
\begin{picture}(160,40)(-15,-23)
\put(0, 0){\circle{4}}
\put(0,8.5){\makebox(0,0){$\al_1$}}
\put(0,-8){\makebox(0,0){1}}
\put(2, 0){\line(1,0){14}}
\put(18, 0){\circle{4}}
\put(18,8.5){\makebox(0,0){$\al_2$}}
\put(18,-8){\makebox(0,0){2}}
\put(20, 0){\line(1,0){10}}
\put(40, 0){\makebox(0,0){$\ldots$}}
\put(50, 0){\line(1,0){10}}
\put(60, -19){\makebox(0,0){$( 2 \leq p \leq \ell -2 )$}}
\put(60, 0){\circle*{4.4}}
\put(60,8.5){\makebox(0,0){$\al_p$}}
\put(60,-8){\makebox(0,0){2}}
\put(60, 0){\line(1,0){10}}
\put(80, 0){\makebox(0,0){$\ldots$}}
\put(90, 0){\line(1,0){10}}
\put(102, 0){\circle{4}}
\put(102,-8){\makebox(0,0){2}}
\put(103.7, 1){\line(2,1){10}}
\put(103.7, -1){\line(2,-1){10}}
\put(115.5, 6){\circle{4}}
\put(115.5, -6){\circle*{4}}
\put(123.5, 14){\makebox(0,0){$\al_{\ell-1}$}}
\put(123.5,5.5){\makebox(0,0){1}}
\put(115.5, -16){$\al_\ell$}
\put(123.5,-8){\makebox(0,0){1}}
\end{picture}
 &  
\begin{picture}(110,45)(1,20)\put(50, 25){
\makebox(8,35){\shortstack{$U(p)\times U(\ell-p)$}}}\end{picture}
   
 \\  \hline
  
\begin{picture}(15,35)(0,0)
\put(10, 13){\makebox(0,0){$E^{}_6$}}\end{picture}

&

\begin{picture}(100,25)(-5, -10)

\put(15, 9.5){\circle*{4 }}
\put(15,17){\makebox(0,0){$\al_1$}}
\put(15,2){\makebox(0,0){1}}
\put(17, 10){\line(1,0){14}}
\put(33.5, 9.5){\circle*{4 }}
\put(33.5,17){\makebox(0,0){$\al_2$}}
\put(33.5,2){\makebox(0,0){2}}
\put(35, 10){\line(1,0){13.6}}
 \put(51, 9.5){\circle{4 }}
 \put(51,17){\makebox(0,0){$\al_3$}}
\put(47, 2){\makebox(0,0){3}}
\put(51,-6){\line(0,1){14}}
\put(53,10){\line(1,0){14}}
\put(69,9.5){\circle{4 }}
\put(69, 17){\makebox(0,0){$\al_4$}}
\put(69,2){\makebox(0,0){2}}
\put(51,-8){\circle{4}}
\put(60, -8.5){\makebox(0,0){$\al_6$}}
\put(45.5,-8.5){\makebox(0,0){2}}
\put(71,10){\line(1,0){16}}
\put(89,9.5){\circle{4 }}
\put(89, 17){\makebox(0,0){$\al_5$}}
\put(89,2){\makebox(0,0){1}}
\end{picture}
 
   & 
\begin{picture}(140,30)(0,0)\put(60, 10){
\makebox(10,10){$SU(5)\times U(1)\times U(1)$}}\end{picture}
 \\  \hline

\begin{picture}(15,35)(0,0)
\put(10, 13){\makebox(0,0){$E^{}_7$}}\end{picture}

&
\begin{picture}(160,25)(-15, -10)

\put(15, 9.5){\circle*{4 }}
\put(15, 18){\makebox(0,0){$\al_1$}}
\put(15,2){\makebox(0,0){1}}
\put(17, 10){\line(1,0){14}}
\put(33.5, 9.5){\circle*{4 }}
\put(33.5, 18){\makebox(0,0){$\al_2$}}
\put(33.5,2){\makebox(0,0){2}}
\put(35, 10){\line(1,0){13.6}}
 \put(51, 9.5){\circle{4 }}
 \put(51, 18){\makebox(0,0){$\al_3$}}
\put(51,2){\makebox(0,0){3}}
\put(69,-6){\line(0,1){14}}
\put(53,10){\line(1,0){14}}
\put(69,9.5){\circle{4 }}
\put(69.9, 18){\makebox(0,0){$\al_4$}}
\put(65,2){\makebox(0,0){4}}
\put(69,-8){\circle{4}}
\put(79, -10){\makebox(0,0){$\al_7$}}
\put(62,-10){\makebox(0,0){2}}
\put(71,10){\line(1,0){16}}
\put(89,9.5){\circle{4 }}
\put(89, 18){\makebox(0,0){$\al_5$}}
\put(89,2){\makebox(0,0){3}}
\put(90.7,10){\line(1,0){16}}
\put(109,9.5){\circle{4 }}
\put(109, 18){\makebox(0,0){$\al_6$}}
\put(109, 2){\makebox(0,0){2}}
\end{picture}
   & 
\begin{picture}(140,30)(0,0)\put(60, 10){
\makebox(10,10){$SO(10)\times  U(1)\times U(1)$}}\end{picture}
 \\  \hline
    \end{tabular}}
 \end{center}
\smallskip

\begin{center}
 {\sc Table 4.} \ Pairs $(\Pi, \Pi_{K})$ of Type II.  
\end{center}
\smallskip

\begin{proof}
The proof is based on Proposition \ref{isotropy} and the correspondence between $\fr{t}$-roots and irreducible submodules of $\fr{m}^{\mathbb{C}}$. 
 For the calculation of $\fr{t}$-roots we apply   
  relation (\ref{troots}).  
 Due to the decomposition (\ref{tangent}), it is sufficient to compute only 
 the positive $\fr{t}$-roots.   
 For the root systems  of the complex simple Lie algebras
 we use the notation from \cite{AA} (see also \cite{GOV}, \cite{Sam}).  
 For convenience, on the  painted Dynkin diagrams presented in Tables  3 and 4,
  we have assigned the heights of the simple roots  $\al_{i}$, 
  associated to the fixed base $\Pi$ any time.
 
 \medskip
{\bf Pairs $\bold{(\Pi, \Pi_{K})}$  of Type I.}  
Let $M=G/K$ be a generalized flag manifold 
defined by a set $\Pi_{M}=\{\al_{i} : \Hgt(\al_i)=4\}$. 
According to Proposition \ref{simple} it is 
$\Pi_{\fr{t}}=\{ \overline{\al}_{i}\}$,  where
  $\overline{\al}_{i}=\kappa(\al_{i})=\al_{i}|_{\fr{t}}$, 
 and   
$\fr{t}^*=\Span_{\mathbb{R}}\{\overline{\al}_{i}\}$. 
If $\al=\sum_{i=1}^{\ell}c_{i}\al_{i}\in R^{+}$  
with $0\leq c_{i}\leq m_{i}$ $(i=1,\ldots, \ell)$,
then relation (\ref{troots}) implies that 
$\kappa(\al)=c_{i}\overline{\al}_{i}$ ($0\leq c_{i} \leq 4$)
thus $R_{\fr{t}}^{+}=\{\overline{\al}_{i}, 2\overline{\al}_{i}, 3\overline{\al}_{i}, 4\overline{\al}_{i}\}$,
so $|R_{\fr{t}}^{+}| =4$.  Here
$|R_{\fr{t}}^{+}|$ denotes 
the cardinality of  the set $R_{\fr{t}}^{+}$. 
Therefore, any generalized   flag manifold 
which is defined by a pair $(\Pi, \Pi_{K})$
of Type I  has four isotropy summands, i.e. $\fr{m}=\fr{m}_1\oplus\fr{m}_2\oplus\fr{m}_3\oplus\fr{m}_4$.

The only complex simple Lie algebras 
for which the associated basis contains 
simple roots with height 4 are the exceptional 
Lie algebras $F_{4}, E_{7}$ and $E_{8}$ (cf. \cite{GOV}, \cite{Hel}).
For $F_4$ we use the basis  $\Pi=\{\al_1=e_2-e_3, \al_2=e_3-e_4, \al_3=e_4, \al_4=\frac{1}{2}(e_1-e_2-e_3-e_4)\}$ with  highest  root $\tilde{a}=2\al_1+3\al_2+4\al_3+2\al_4$. Thus $\Pi_{M}=\{\al_{3}\}$ so we obtain the flag manifold $F_4/SU(3)\times SU(2)\times U(1)$.  For $E_7$ we set $\Pi=\{e_i-e_{i+1} \ (i<7), \al_7=e_5+e_6+e_7+e_8\}$ with $\tilde{a}=\al_1+2\al_2+3\al_3+4\al_4+3\al_5+2\al_6+2\al_7$.  Thus $\Pi_{M}=\{\al_4\}$ which determines the flag $E_7/SU(4)\times SU(3)\times SU(2)\times U(1)$.  For the root system of $E_8$ a basis is given by
$\Pi=\{\al_{1}=e_{1}-e_{2}, \ldots, \al_{7}=e_{7}-e_{8}, \al_{8}=e_{6}+e_{7}+e_{8}\}$, and $\tilde{\al}=2\al_{1}+3\al_{2}+4\al_{3}+5\al_{4}+6\al_{5}+4\al_{6}+2\al_{7}+3\al_{8}$.  
There are two choises, $\Pi_{M}=\{\al_{3}\}$ or $\Pi_{M}=\{\al_6\}$, which  determine the flag manifolds $E_8/SO(10)\times SU(3)\times U(1)$ and $E_{8}/SU(7)\times SU(2)\times U(1)$ respectively.

\medskip
\medskip
{\bf Pairs $\bold{(\Pi, \Pi_{K})}$  of Type II.} \  
Let $M=G/K$ be  a generalized flag manifold  such that 
$\Pi_{M}=\{\al_i, \al_{j} : \Hgt(\al_{i})=1, \Hgt(\al_{j})=2\}$.  
A $\fr{t}$-base is given by $\Pi_{\fr{t}}=\{\overline{\al}_{i}, \overline{\al}_{j}\}$, 
so
$\fr{t}^*=\Span_{\mathbb{R}}\{\overline{\al}_{i}, \overline{\al}_{j}\}$.
For $\al\in R^+$, let 
$\al=\sum_{k=1}^{\ell}c_{k}\al_{k}$
 with 
$0\leq c_{k}\leq m_{k}$ $(k=1\ldots, \ell)$.  
Then (\ref{troots})  implies  that 
\begin{equation}\label{typeII}
\kappa(\al)=c_{i}\overline{\al}_{i}+c_j\overline{\al}_{j},
\end{equation}
where $0\leq c_{i} \leq m_{i}=1$, and $0\leq c_{j}\leq m_{j}=2$.  
Therefore, we obtain at most  five different positive $\fr{t}$-roots: 
$\overline{\al}_{i}, \ \overline{\al}_{j}, \ \overline{\al}_{i}+\overline{\al}_{j}, 
 \ \overline{\al}_{i}+2\overline{\al}_{j},$ and $2\overline{\al}_{j}$. 
 The appearence of the fifth  $\fr{t}$-root in the previous sequence
depends on the choise of $(\Pi, \Pi_{K})$. 
As we will see later on,  there are cases where the 
system $R_{\fr{t}}^+$ contains only  four elements.
  
Pairs $(\Pi, \Pi_{K})$ of type II exist only for the Lie groups 
  $SO(2\ell+1)$,  $Sp(\ell)$, $SO(2\ell)$, $E_{6}$,  and $E_{7}$.
In order to describe the  root systems $R, R_K$ for the corresponding 
flag manifolds of the classical groups we follow the method
of \cite{AA} (see also \cite{AP}).   We will only examine the cases of
 $B_{\ell}=SO(2\ell+1)$ and $E_6$,
and the remaining cases given in Table 4 are obtained by a similar procedure. 
We remark that the generalized flag manifolds  $SO(2\ell)/U(p)\times U(\ell-p)$ and $E_{7}/SO(10)\times U(1)\times U(1)$, 
 can   also be obtained by setting  $\Pi_{M}=\{\al_{p}, \al_{\ell-1}\}$  and 
 $\Pi_{M}=\{\al_1, \al_6\}$, respectively. For the case of $SO(2\ell)_{(ii)}$ see also \cite{Arv}.

\medskip
{\it Case of $SO(2\ell+1)$.}  
Let $\Pi_{M}=\{\al_{1}, \al_2\}$. 
This choice determines the generalized flag manifold 
$M=G/K=SO(2\ell+1)/U(1)\times U(1)\times SO(2(\ell-2)+1)
=SO(2\ell+1)/U(1)\times U(1)\times SO(2\ell-3)$.

Let $\{e^1_1, e^2_1, \pi_{j}\}$  $(j=1, \ldots, \ell-2),$
be an orthonormal  basis on $\mathbb{R}^{\ell}$. 
Then the root system $R$ of the complex simple Lie algebra 
$B_{\ell}=\fr{so}(2\ell+1, \mathbb{C})$ is described as follows:
\[ 
R=\{\pm e^1_1\pm e^2_1, \ \pm e_1^1\pm \pi_j, \ \pm e^2_1\pm \pi_j, \ \pm\pi_i\pm \pi_j, \ \pm e_1^1, \ \pm e_1^2, \ \pm \pi_{j} : i<j\},
\]
We fix a basis $
\Pi=\{\al_1=e_1^1-e_1^2, \ \al_2=e^2_1-\pi_1, \ \phi_1=\pi_1-\pi_2,  \ldots, \phi_{\ell-3}=\pi_{\ell-3}-\pi_{\ell-2}, \ \phi_{\ell-2}=\pi_{\ell-2}\},
$ and set $R^+=\{e^1_1\pm e^2_1, \ e_1^1\pm \pi_j, \ e^2_1\pm \pi_j, \ \pi_i\pm \pi_j, \  e_1^1, \  e_1^2, \  \pi_{j} : i<j\}$. 
The highest root is given by
$\tilde{\al}=e^1_1+e^2_1=\al_1+2\al_2+2\phi_1\cdots +2\phi_{\ell-2}$.
The root  system $R_{K}$ of the semisimple part 
$\fr{so}(2 \ell-3, \mathbb{C})$  of $\fr{k}^{\mathbb{C}}$,  
is given by  $R_{K}=\{\pm\pi_j, \ \pm\pi_{i}\pm\pi_j : i<j\}$,
while $\Pi_{K}=\{\phi_{1}, \ldots, \phi_{\ell-2}\}$ is a basis. 
Thus $R_{K}^{+}=\{\pi_j, \ \pi_{i}\pm\pi_j : i<j\}$, and 
$R_{M}^+=R^+\backslash R_K^{+}=\{ e^1_1\pm e^2_1, \  e_1^1\pm \pi_j, \   e^2_1\pm \pi_j, \   e_1^1, \  e_1^2 : i<j \}$.
Since $\Pi_{M}=\{\al_1, \al_2\}$, it is $\fr{t}^*=\Span_{\mathbb{R}}\{\overline{\al}_1, \overline{\al}_2\}$,
where $\overline{\al}_{i}=\kappa(\al_{i})$ $(i=1, 2)$.   
Relation  (\ref{typeII}) implies that  for any $\al\in R_{M}^+$ it is 
$\kappa(\al)=c_1\overline{\al}_{1}+c_2\overline{\al}_{2}$, 
with $0\leq c_{1}\leq 1$ and $0\leq c_2\leq 2$. 
By using the expression of each positive complementary root 
in terms of simple roots of $\Pi$,
the linearity of the restriction $\kappa$, 
and the fact that $\kappa(\Pi_{K})=0$, we obtain that
\begin{eqnarray*}
\kappa(e^1_1-e^2_1) &=& \kappa(\al_1)=\overline{\al}_1, \\
\kappa(e^2_1-\pi_1) &=& \kappa(\al_2)=\overline{\al}_2, \\
\kappa(e^1_1+e^2_1) &=& \kappa(\al_1+2\al_2+2\phi_1\cdots +2\phi_{\ell-2})=\kappa(\al_1+2\al_2)=\overline{\al}_1+2\overline{\al}_2, \\
\kappa(e^1_1-\pi_j) &=& \kappa(\al_1+\cdots + \phi_{j-1})=\kappa(\al_1+\al_2)=\overline{\al}_1+\overline{\al}_2 \\
\kappa(e^1_1+\pi_j) &=& \kappa(\al_1+\cdots+\phi_{j-1}+2\phi_{j}+2\phi_{j+1}+\cdots+2\phi_{\ell-2})=\kappa(\al_1+\al_2)=\overline{\al}_1+\overline{\al}_2,\\
\kappa(e^2_1-\pi_j) &=& \kappa(\al_2+\cdots +\phi_{j-1})=\kappa(\al_2)=\overline{\al}_2, \\  
\kappa(e^2_1+\pi_j) &=& \kappa(\al_2+\cdots +\phi_{j-1}+2\phi_{j}+2\phi_{j+1}+\cdots+2\phi_{\ell-2})=\kappa(\al_2)=\overline{\al}_2, \\
\kappa(e^1_1) &=& \kappa(\al_1+\cdots +\phi_{\ell-2})=\kappa(\al_1+\al_2)=\overline{\al}_1+\overline{\al}_2, \\
\kappa(e^2_1) &=& \kappa(\al_2+\cdots +\phi_{\ell-2})=\kappa(\al_2)=\overline{\al}_2.
\end{eqnarray*}
We easily conclude that 
$R_{\fr{t}}^+=\{\overline{\al}_{1}, \ \overline{\al}_{2}, \ \overline{\al}_{1}+\overline{\al}_{2}, 
 \ \overline{\al}_{1}+2\overline{\al}_{2}\}$, thus $\fr{m}=\fr{m}_1\oplus\fr{m}_2\oplus\fr{m}_3\oplus\fr{m}_4$.

\medskip
{\it Case of $E_6$.} The root system $R$ of $E_6$ is given by 
\[
R=\{\pm(e_i-e_j), \ e_{i}+e_{j}+e_{k}\pm e, \ \pm(2e) : 1\leq i<j<k\leq 6 \},
\]
where $e$ is a vector orthogonal to all vectors $e_i$.  A basis of $R$ is given by 
$\Pi=\{\al_{i}=e_{i}-e_{i+1} \ (i<5),  \ \al_{6}=e_{4}+e_{5}+e_{6}+e\}$. 
The highest root  is given by 
$\tilde{\al}=\al_{1}+2\al_{2}+3\al_{3}+2\al_{4}+\al_{5}+2\al_{6}$, 
and   any root $\al\in R$
is expressed as 
$\al =c_1\al _1+c_2\al _2+c_3\al _3+c_4\al _4+c_5\al_5+c_6\al_6$, 
with
$|c_1|\leq 1$, $|c_2|\leq 2$, $|c_3|\leq 3$, 
$|c_4|\leq 2$, $|c_5|\leq 1$, $|c_6|\leq 2$.   
We fix the set of positive roots to be $R^+=\{e _i-e _j\ (i<j), \ 2e, \ e _i+e _j+e _k+e\  (i<j<k)\}$. 
There are several  choices of pairs 
$(\Pi, \Pi_{K})$ of Type II. 
Let $\Pi_{M}$ be one of  the
sets $\{\al_1, \al_2\}$, $\{\al_4, \al_5\}$,    $\{\al_1, \al_{6}\}$, or $\{\al_5, \al_6\}$.
  These sets correspond to the following 
  Dynkin diagrams, 
  which define the same flag manifold 
  $M=G/K=E_6/SU(5)\times U(1)\times U(1)$.    
  
\medskip
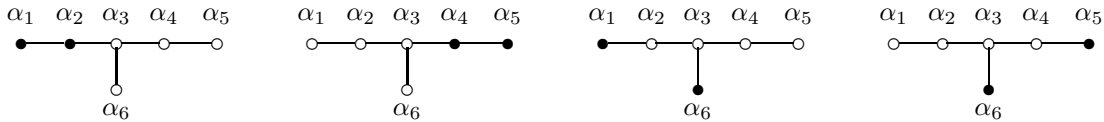
\begin{center}
\begin{picture}(100,25)(0, -10)
\put(15, 9.5){\circle*{4 }}
\put(15,20){\makebox(0,0){$\al_1$}}
\put(17, 10){\line(1,0){14}}
\put(33.5, 9.5){\circle*{4 }}
\put(33.5, 20){\makebox(0,0){$\al_2$}}
\put(35, 10){\line(1,0){13.6}}
 \put(51, 9.5){\circle{4 }}
 \put(51,20){\makebox(0,0){$\al_3$}}
\put(51,-6){\line(0,1){14}}
\put(53,10){\line(1,0){14}}
\put(69,9.5){\circle{4 }}
\put(69,20){\makebox(0,0){$\al_4$}}
\put(51,-8){\circle{4}}
\put(51,-16){\makebox(0,0){$\al_6$}}
\put(71,10){\line(1,0){16}}
\put(89,9.5){\circle{4 }}
\put(89,20){\makebox(0,0){$\al_5$}}
\end{picture}
\ \
 \begin{picture}(100,25)(0, -10)

\put(15, 9.5){\circle{4 }}
\put(15,20){\makebox(0,0){$\al_1$}}
\put(17, 10){\line(1,0){14}}
\put(33.5, 9.5){\circle{4 }}
\put(33.5, 20){\makebox(0,0){$\al_2$}}
\put(35, 10){\line(1,0){13.6}}
 \put(51, 9.5){\circle{4 }}
 \put(51,20){\makebox(0,0){$\al_3$}}
\put(51,-6){\line(0,1){14}}
\put(53,10){\line(1,0){14}}
\put(69,9.5){\circle*{4 }}
\put(69,20){\makebox(0,0){$\al_4$}}
\put(51,-8){\circle{4}}
\put(51,-16){\makebox(0,0){$\al_6$}}
\put(71,10){\line(1,0){16}}
\put(89,9.5){\circle*{4 }}
\put(89,20){\makebox(0,0){$\al_5$}}
\end{picture}
\ \
\begin{picture}(100,25)(0, -10)

\put(15, 9.5){\circle*{4 }}
\put(15,20){\makebox(0,0){$\al_1$}}
\put(17, 10){\line(1,0){14}}
\put(33.5, 9.5){\circle{4 }}
\put(33.5, 20){\makebox(0,0){$\al_2$}}
\put(35, 10){\line(1,0){13.6}}
 \put(51, 9.5){\circle{4 }}
 \put(51,20){\makebox(0,0){$\al_3$}}
\put(51,-6){\line(0,1){14}}
\put(53,10){\line(1,0){14}}
\put(69,9.5){\circle{4 }}
\put(69,20){\makebox(0,0){$\al_4$}}
\put(51,-8){\circle*{4}}
\put(51,-16){\makebox(0,0){$\al_6$}}
\put(71,10){\line(1,0){16}}
\put(89,9.5){\circle{4 }}
\put(89,20){\makebox(0,0){$\al_5$}}
\end{picture}
\ \
 \begin{picture}(100,25)(0, -10)

\put(15, 9.5){\circle{4 }}
\put(15,20){\makebox(0,0){$\al_1$}}
\put(17, 10){\line(1,0){14}}
\put(33.5, 9.5){\circle{4 }}
\put(33.5, 20){\makebox(0,0){$\al_2$}}
\put(35, 10){\line(1,0){13.6}}
 \put(51, 9.5){\circle{4 }}
 \put(51,20){\makebox(0,0){$\al_3$}}
\put(51,-6){\line(0,1){14}}
\put(53,10){\line(1,0){14}}
\put(69,9.5){\circle{4 }}
\put(69,20){\makebox(0,0){$\al_4$}}
\put(51,-8){\circle*{4}}
\put(51,-16){\makebox(0,0){$\al_6$}}
\put(71,10){\line(1,0){16}}
\put(89,9.5){\circle*{4 }}
\put(89,20){\makebox(0,0){$\al_5$}}
\end{picture}
 \end{center}
 \medskip
 
Let  $\Pi_{M} =\{\al_1, \al_2\}$, so
$\Pi_{K}=\{\al_3, \al_4, \al_5, \al_6\}$
is a basis of simple roots for the root system $R_{K}$.  
In particular, we obtain that $R_{K}^{+}=\{\al_3, \al_4,  \al_5, \al_6, \al_3+\al_4, 
\al_4+\al_5, \al_3+\al_4+\al_5, \al_3+\al_4+\al_5+\al_6, \al_3+\al_4+\al_6, \al_3+\al_6\}$, 
so the positive complementary roots   are  the following:
\begin{equation}\label{E6}
R_M^+=\left\{ 
\begin{tabular}{ll} 
$\al_1+2\al_2+3\al_3+2\al_4+\al_5+2\al_6$ &    $\al_2+2\al_3+2\al_4+\al_5+\al_6$  \\
$\al_1+2\al_2+3\al_3+2\al_4+\al_5+\al_6$  &    $\al_2+2\al_3+\al_4+\al_5+\al_6$\\
$\al_1+2\al_2+2\al_3+2\al_4+\al_5+\al_6$  &    $\al_2+2\al_3+\al_4+\al_6$\\
$\al_1+2\al_2+2\al_3+\al_4+\al_5+\al_6$   &    $\al_2+\al_3+\al_4+\al_5+\al_6$\\
$\al_1+2\al_2+2\al_3+\al_4+\al_6$         &    $\al_2+\al_3+\al_4+\al_6$\\
$\al_1+\al_2+2\al_3+2\al_4+\al_5+\al_6$   &    $\al_2+\al_3+\al_4+\al_5$\\
$\al_1+\al_2+2\al_3+\al_4+\al_5+\al_6$    &    $\al_2+\al_3+\al_6$\\
$\al_1+\al_2+2\al_3+\al_4+\al_6$          &    $\al_2+\al_3+\al_4$\\
$\al_1+\al_2+\al_3+\al_4+\al_5+\al_6$     &    $\al_2+\al_3$\\
$\al_1+\al_2+\al_3+\al_4+\al_6$           &    $\al_2$\\
$\al_1+\al_2+\al_3+\al_6$                 &    $\al_1+\al_2+\al_3,$\\
$\al_1+\al_2+\al_3+\al_4+\al_5$           &    $\al_1+\al_2$    \\
$\al_1+\al_2+\al_3+\al_4$                 &    $\al_1$\\
   \end{tabular} \right. 
\end{equation}
It is $\fr{t}=\Span_{\mathbb{R}}\{\overline{\al}_1, \overline{\al}_2\}$,
where $\overline{\al}_{i}=\kappa(\al_{i})$ $(i=1, 2)$.  
By using the above expressions and relation (\ref{typeII}), 
we obtain that $R_{\fr{t}}^+=\{\overline{\al}_1, \overline{\al}_2, 
\overline{\al}_1+\overline{\al}_2, \overline{\al}_1+2\overline{\al}_2\}$, 
thus we  conclude that 
$\fr{m}=\fr{m}_1\oplus\fr{m}_2\oplus\fr{m}_3\oplus\fr{m}_4$, where
\begin{equation}\label{e66}
\begin{tabular}{ll}
$\fr{m}_1=(\fr{m}_{\overline{\al}_1}\oplus  \fr{m}_{-\overline{\al}_1})^{\tau},$  & $\fr{m}_3=(\fr{m}_{(\overline{\al}_1+\overline{\al}_2)}\oplus  \fr{m}_{-(\overline{\al}_1+\overline{\al}_2)})^{\tau}$, \\
 $\fr{m}_2=(\fr{m}_{\overline{\al}_2}\oplus  \fr{m}_{-\overline{\al}_2})^{\tau}$, & $\fr{m}_4=(\fr{m}_{(\overline{\al}_1+2\overline{\al}_2)}\oplus  \fr{m}_{-(\overline{\al}_1+2\overline{\al}_2)})^{\tau}$.
\end{tabular}
\end{equation}

\end{proof}
  
 \begin{prop}\label{cor1}
The only generalized flag manifolds with four isotropy summands 
are those obtained in Proposition \ref{classification}.
\end{prop}
\begin{proof}
As there is no general known method, we have to proceed into a case by case examination, 
by painting black all possible combinations of simple roots.  
Let  $G$ be a   simple Lie group such that $G\in\{SO(2\ell+1), Sp(\ell), SO(2\ell)\}$, or $G\in\{F_4, E_6, E_7, E_8\}$.
 Let $\Pi_{M}=\{\al_{i}, \al_{j} : \Hgt(\al_{i})=\Hgt(\al_{j})=2\}$.  By  use of (\ref{troots})  we conclude that  $|R_{\fr{t}}^{+}|\in\{5, 6\}$.
 Thus  the isotropy representation of the corresponding flag manifold $M=G/K$ decomposes into more than four irreducible submodules. The same is true if we set
  $\Pi_{M}=\{\al_{i}, \al_{j} : \Hgt(\al_{i})=1, \ \Hgt(\al_{j})=3\}$. In particular, such a choise exists only if $G\in\{E_6, E_7\}$, and we  find that $|R_{\fr{t}}^{+}|=6$.
   It is obvious that all the other possible paintings,
  determine flag manifolds with more than four isotropy summands.
Let us now describe  pairs $(\Pi, \Pi_K)$ of Type II, 
which define exactly five positive $\fr{t}$-roots and 
thus do not belong in our classification.  Such pairs 
appear for the Lie groups $SO(2\ell+1), SO(2\ell), E_6$ 
and $E_7$. The process is the same as in Proposition \ref{classification},  thus we don't insist on  details.

\medskip
{\it Case of $SO(2\ell+1)$.}  
Let $\Pi_{M}=\{\al_1, \al_p : 3\leq p\leq \ell\}$. 
This determines the following painted Dynkin diagram 
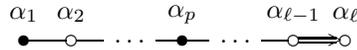
\[
\begin{picture}(160,26)(-5,-5)
\put(0, 0){\circle*{4}}
\put(0,10){\makebox(0,0){$\al_1$}}
\put(2, 0.3){\line(1,0){14}}
\put(18, 0){\circle{4}}
\put(20, 0.3){\line(1,0){10}}
\put(18,10){\makebox(0,0){$\al_2$}}
\put(40, 0){\makebox(0,0){$\ldots$}}
\put(50, 0.3){\line(1,0){10}}
\put(60, 0){\circle*{4.4}}
\put(60, 10){\makebox(0,0){$\al_p$}}
\put(60, 0.3){\line(1,0){10}}
\put(80, 0){\makebox(0,0){$\ldots$}}
\put(90, 0.3){\line(1,0){10}}
\put(102, 0){\circle{4}}
\put(102, 10){\makebox(0,0){$\al_{\ell-1}$}}
\put(104, 1.1){\line(1,0){14}}
\put(104, -0.6){\line(1,0){14}}
\put(114.5, -1.5){\scriptsize $>$}
\put(121.5, 0){\circle{4}}
\put(122, 10){\makebox(0,0){$\al_\ell$}}
\end{picture}
\]
or equivalently the generalized flag manifold 
$M=G/K=SO(2\ell+1)/U(1)\times U(p-1)\times SO(2(\ell-p)+1)$,
with  $R_{\fr{t}}^+=\{\overline{\al}_1, \ \overline{\al}_p, \ 2\overline{\al}_p, 
 \ \overline{\al}_1+\overline{\al}_p, \ \overline{\al}_1+2\overline{\al}_p\}$, thus $|R_{\fr{t}}^{+}|=5$.  
For $p=\ell$ we obtain the space 
$SO(2\ell+1)/U(1)\times U(\ell-1)$  
which has also five isotropy summands.

\medskip
{\it Case of $SO(2\ell)$.}  Let 
$\Pi_{M}=\{\al_1, \al_p : 3\leq p\leq \ell-2\}$.  
This determines the following painted Dynkin diagram
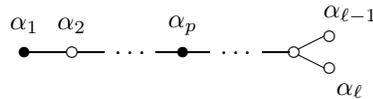
\[
\begin{picture}(160,40)(-15,-23)
\put(0, 0){\circle*{4}}
\put(0,10){\makebox(0,0){$\al_1$}}
\put(2, 0){\line(1,0){14}}
\put(18, 0){\circle{4}}
\put(20, 0){\line(1,0){10}}
\put(18,10){\makebox(0,0){$\al_2$}}
\put(40, 0){\makebox(0,0){$\ldots$}}
\put(50, 0){\line(1,0){10}}
 \put(60, 0){\circle*{4.4}}
\put(60, 10){\makebox(0,0){$\al_p$}}
\put(60, 0){\line(1,0){10}}
\put(80, 0){\makebox(0,0){$\ldots$}}
\put(90, 0){\line(1,0){10}}
\put(102, 0){\circle{4}}
\put(103.7, 1){\line(2,1){10}}
\put(103.7, -1){\line(2,-1){10}}
\put(115.5, 6){\circle{4}}
\put(115.5, -6){\circle{4}}
\put(123.5, 14){\makebox(0,0){$\al_{\ell-1}$}}
\put(118, -16){$\al_\ell$}
\end{picture}
\]
or equivalently the flag manifold 
$M=G/K=SO(2\ell)/U(1)\times U(p-1)\times SO(2(\ell-p))$ 
with $R_{\fr{t}}^+=\{\overline{\al}_1, \overline{\al}_p, 2\overline{\al}_p, 
\overline{\al}_1+\overline{\al}_p, \overline{\al}_1+2\overline{\al}_p\}$, so $|R^+_{\fr{t}}|=5$.

\medskip
{\it Case of $E_6$.}
We follow the notation of Proposition \ref{classification}.  For $E_6$, 
pairs $(\Pi, \Pi_K)$ of Type II 
are also obtaining by choosing $\Pi_{M}=\{\al_1,\al_4\}$, 
or $\Pi_M=\{\al_2, \al_5\}$.  
The corresponding Dynkin diagrams are 
given below, and these determine the same flag manifold 
$M=G/K=E_6/SU(4)\times SU(2)\times U(1)\times U(1)$.
\medskip
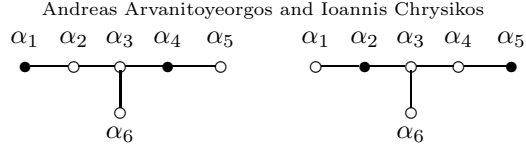
\begin{center}
\begin{picture}(100,25)(0, -10)

\put(15, 9.5){\circle*{4 }}
\put(15,20){\makebox(0,0){$\al_1$}}
\put(17, 10){\line(1,0){14}}
\put(33.5, 9.5){\circle{4 }}
\put(33.5, 20){\makebox(0,0){$\al_2$}}
\put(35, 10){\line(1,0){13.6}}
 \put(51, 9.5){\circle{4 }}
 \put(51,20){\makebox(0,0){$\al_3$}}
\put(51,-6){\line(0,1){14}}
\put(53,10){\line(1,0){14}}
\put(69,9.5){\circle*{4 }}
\put(69,20){\makebox(0,0){$\al_4$}}
\put(51,-8){\circle{4}}
\put(51,-16){\makebox(0,0){$\al_6$}}
\put(71,10){\line(1,0){16}}
\put(89,9.5){\circle{4 }}
\put(89,20){\makebox(0,0){$\al_5$}}
\end{picture}
\ \
 \begin{picture}(100,25)(0, -10)

\put(15, 9.5){\circle{4 }}
\put(15,20){\makebox(0,0){$\al_1$}}
\put(17, 10){\line(1,0){14}}
\put(33.5, 9.5){\circle*{4 }}
\put(33.5, 20){\makebox(0,0){$\al_2$}}
\put(35, 10){\line(1,0){13.6}}
 \put(51, 9.5){\circle{4 }}
 \put(51,20){\makebox(0,0){$\al_3$}}
\put(51,-6){\line(0,1){14}}
\put(53,10){\line(1,0){14}}
\put(69,9.5){\circle{4 }}
\put(69,20){\makebox(0,0){$\al_4$}}
\put(51,-8){\circle{4}}
\put(51,-16){\makebox(0,0){$\al_6$}}
\put(71,10){\line(1,0){16}}
\put(89,9.5){\circle*{4 }}
\put(89,20){\makebox(0,0){$\al_5$}}
\end{picture}
\end{center}
By using the expressions of each positive root
in terms of simple roots (cf. \cite{AA}), we easily
obtain that $|R_{\fr{t}}^+|=5$.  

\medskip
{\it Case of $E_7$.}  Let $\Pi_M=\{\al_1, \al_7\}$ 
with $\Hgt(\al_1)=1$ and $\Hgt(\al_7)=2$.  This choise determines the flag manifold
$M=G/K=E_7/SU(6)\times U(1)\times U(1)$ with $|R_{\fr{t}}^+|=5$.
\end{proof}

 For practical purposes, we refer to spaces of Table 3 (resp. of Table 4) as {\it (generalized) flag manifolds of Type I} (resp.  {\it (generalized) flag manifolds of Type II}),
 depending on the type of the pair $(\Pi, \Pi_K)$.  
For flag manifolds of Type II, we make a further division into
{\it Type IIa} and {\it Type IIb}, depending on
whether the first painted black simple root 
has height 1 or 2 (respectively).
So the flag manifolds $M=G/K$ which correspond  
to the  Lie groups $SO(2\ell+1), SO(2\ell)_{(i)}, E_6,$ and $E_7$ 
are of Type IIa, 
since $\Pi_{M}=\{\al_1, \al_2\}$ with $\Hgt(\al_1)=1$, $\Hgt(\al_2)=2$.
Also, the flag manifolds $M=G/K$ which correspond to the Lie groups $Sp(\ell)$ and $SO(2\ell)_{(ii)}$ 
are of Type IIb, 
since it is $\Pi_{M}=\{\al_{p}, \al_\ell : 1\leq p\leq \ell-1\}$  and  
$\Pi_{M}=\{\al_{p}, \al_\ell : 2\leq p\leq \ell-2\}$ 
respectively, with
$\Hgt(\al_p)=2$ and $\Hgt(\al_\ell)=1$.   

Let now fix notation for later use.   Let $M=G/K$ be a generalized flag manifold of Type I  or II, and let
   $\fr{m}=\fr{m}_1\oplus\fr{m}_2\oplus\fr{m}_3\oplus\fr{m}_4$ be  a
decomposition 
of $\fr{m}=T_{o}M$ into   irreducible inequivalent  $\ad(\fr{k})$-submodules, with respect to the negative of the Killing form $B$  of $G$.
Then according to (\ref{Inva}), a $G$-invariant metric on $M=G/K$ is given by
\begin{equation}\label{metrI}
\left\langle \ , \ \right\rangle=x_1\cdot (-B)|_{\fr{m}_1}+
x_2\cdot (-B)|_{\fr{m}_2}+x_3\cdot (-B)|_{\fr{m}_3}+x_4\cdot (-B)|_{\fr{m}_4},
\end{equation}
 for positive real numbers $x_1, x_2, x_3, x_4$.
Next we will denote such metrics with $g=(x_1, x_2, x_3, x_4)$.

\markboth{Andreas Arvanitoyeorgos and Ioannis Chrysikos}{Invariant Einstein metrics  on  flag manifolds with four isotropy summands}
\section{K\"ahler-Einstein metrics}
\markboth{Andreas Arvanitoyeorgos and Ioannis Chrysikos}{Invariant Einstein metrics  on  flag manifolds with four isotropy summands}

  In order to obtain Einstein equation for flag manifolds with four isotropy summands is crucial to know the corresponding K\"ahler-Einstein metrics.
  These metrics are obtained by an independent procedure  which is described in this section.
 \subsection{$G$-invariant complex stuctures}
  Let $M=G/K$ be a   generalized flag manifold  
  associated to a pair $(\Pi, \Pi_{K})$, where $G$ is a compact simple Lie group, 
  and let $\fr{g}=\fr{k}\oplus\fr{m}$ be a reductive decomposition of $\fr{g}$. 
  A  $G$-invariant  (almost) complex structure $J$ on $M$ 
(as a tensor field of type (1, 1)  satisfying  $J^{2}=-\Id$), 
can be described  by an   endomorphism $J_{o} : \fr{m}\to \fr{m}$   
such that $J^{2}_{o} = -{\rm Id}_{\fr{m}}$ which is $\Ad(K)$-invariant, that 
is $J_{o}(\chi(k)X) = \chi(k)(J_{o}X),$ for all $k\in K$ and $X\in\fr{m}$.  

Let $\mathcal{J}$  denote the set of all $G$-invariant
complex structures  on $M$. It is well known (\cite{Wan}) that  $\mathcal{J}$ is a finite set.     
In fact, it can  be shown that there is a 
one-to-one  
correspondence between elements in $\mathcal{J}$ and parabolic  subgroups 
$P$ of $G^{\mathbb{C}}$ such that $G\cap P=K$.
 In \cite[Prop.~ 13.8]{B} it was shown that if $M=G/K$ is a flag manifold such that
the center of the subgroup $K$ is one-dimensional, then $M$ admits a unique (up to equivalence) 
 $G$-invariant complex structure. This case arises by painting black only one simple root in  the   Dynkin diagram of $G$.
 Thus, all flag manifolds $M=G/K$ of Type I admit  a unique $G$-invariant complex structure. 

The set $\mathcal{J}$ is also related to invariant orderings and $\fr{t}$-chambers.  
 
 \begin{prop}\label{chambers}\textnormal{(\cite[p.~625]{Bor})} 
Let $M=G/K$ be a generalized flag manifold defined by a pair $(\Pi, \Pi_{K})$. 
Then there is a one-to-one correspondence between  $\fr{t}$-chambers, 
invariant orderings $R_{M}^{+}$ in $R_{M}$,
and $G$-invariant complex structures  given by 
$J_{o}E_{\pm\al} = \pm i E_{\pm\al},$ for all $\al\in R_{M}^{+}$.
\end{prop}

 By using   Proposition \ref{chambers}, 
we obtain the following theorem.
 
\begin{theorem}\label{The2}
Let $M=G/K$ be a generalized flag manifold of Type II.  Then $M$ admits precisely four invariant complex structures.
\end{theorem}
 
\begin{proof}
We assume that $M=G/K$ is a generalized flag manifold of Type IIa.
It is $\Pi_{M}=\{\al_1, \al_2 : \Hgt(\al_1)=1,  \Hgt(\al_2)=2\}$,
and   from Proposition \ref{classification} it follows that
$R_{\fr{t}}^+=\{\overline{\al}_{1}, \ \overline{\al}_{2}, \ \overline{\al}_{1}+\overline{\al}_{2}, 
 \ \overline{\al}_{1}+2\overline{\al}_{2}\}$. 
 In $\fr{t}$ there are four walls, 
 namely the hyperplanes (cf. Figure 1(a))
\[
 \begin{tabular}{ll}
 $\fr{a}_1=\{X\in\fr{t} : \overline{\al}_{1}(X)=0\},$ & $\fr{a}_2=\{X\in\fr{t} : \overline{\al}_{2}(X)=0\}$, \\
 $\fr{a}_3=\{X\in\fr{t} : (\overline{\al}_{1}+\overline{\al}_{2})(X)=0\},$  & $\fr{a}_4=\{X\in\fr{t} :  (\overline{\al}_{1}+2\overline{\al}_{2})(X)=0\}.$ 
   \end{tabular}
 \] 
These walls devide $\fr{t}$ into  eight $\fr{t}$-chambers.
The positive $\fr{t}$-chambers are determined by the following inequalities 
\begin{eqnarray*}
C_1&=&\{\overline{\al}_{1}>0, -\overline{\al}_{2}>0,  -(\overline{\al}_{1}+\overline{\al}_{2})>0, -(\overline{\al}_{1}+2\overline{\al}_{2})>0\},\\
C_2&=&\{\overline{\al}_{1}>0, \  \overline{\al}_{2}>0, \ \overline{\al}_{1}+\overline{\al}_{2}>0, \ \overline{\al}_{1}+2\overline{\al}_{2}>0\},\\
C_3&=&\{\overline{\al}_{1}>0, -\overline{\al}_{2}>0, \ \overline{\al}_{1}+\overline{\al}_{2}>0, -(\overline{\al}_{1}+2\overline{\al}_{2})>0\}, \\
C_4&=&\{\overline{\al}_{1}>0, -\overline{\al}_{2}>0, \ \overline{\al}_{1}+\overline{\al}_{2}>0, \ \overline{\al}_{1}+2\overline{\al}_{2}>0\}.
\end{eqnarray*}
These $\fr{t}$-chambers   induce four  invariant orderings $R_{M}^{+}$ in $R_{M}$, which in turn  
determine four invariant complex structures   on   $M$. The four negative $\fr{t}$-chambers are obtained by reversing the inequalities and these chambers induce the invariant orderings $R_{M}^{-}=-R_{M}^{+}$, which in turn determine the conjugate complex structures. 
An invariant complex structure on $M$ and its conjugate, are equivalent complex structures under an automorphism of $G$ (\cite{B}), hence we identify them.
 
A similar analysis  is applied to flag manifolds of Type IIb, and gives four invariant complex structures which  are determined by the $\fr{t}$-chambers 
$C_{1}^{'}, C_2^{'}, C_3^{'} C_4^{'}$ given in  Figure 1(b). We mention  that the shaded $\fr{t}$-chambers $C_2$ and $C'_2$ in Figure 1 are those which induce the natural invariant ordering, or equivalently the canonical  complex structure on the associated flag manifold $M=G/K$.  
 \end{proof}
 
 We remark that it may exists  some automorphism of $G$ carrying one of the above invariant complex structures onto another.
The complete classification of equivalent  invariant complex structures on flag manifolds was obtained in  \cite{N}, where stated that
  all flag manifolds  of Type II  admit  two pairs of equivalent  complex structures.   
 
 \begin{figure}
 \begin{center}
 \begin{picture}(200,80)(50, 0)
\put(100, 0){\line(1,0){100}}
\put(210,0){\makebox(0,0){$\fr{a}_2 $}}
\put(150,-50){\line(0,1){100}}
\put(155,-60){\makebox(0,0){$\fr{a}_1 $}}
\put(150, 0){\line(1,-1){50}}
\put(100, 50){\line(1,-1){50}}
\put(210,-55){\makebox(0,0){$\fr{a}_3 $}}
\put(150, 0){\line(2,-1){50}}
\put(100, 25){\line(2,-1){50}}
\put(210,-25){\makebox(0,0){$\fr{a}_4 $}}
\put(160,80){\makebox(0,0){$(a)  \ M=G/K \ \mbox{of Type IIa}$}}
\put(190, 44){\makebox(0,0){$C_2 $}}
\put(190,-8){\makebox(0,0){$C_4 $}}
\put(193,-31){\makebox(0,0){$C_3 $}}
\put(170,-38){\makebox(0,0){$C_1 $}}
\put(155, 5){\line(1,0){40}}
\put(155, 10){\line(1,0){40}}
\put(155, 15){\line(1,0){40}}
\put(155, 20){\line(1,0){40}}
\put(155, 25){\line(1,0){40}}
\put(155, 30){\line(1,0){40}}
\put(155, 35){\line(1,0){40}}
 \end{picture}
 \begin{picture}(200,80)(50, 0)
\put(100, 0){\line(1,0){100}}
\put(210,0){\makebox(0,0){$\fr{a}'_2 $}}
\put(150,-50){\line(0,1){100}}
\put(155,-60){\makebox(0,0){$\fr{a}'_1 $}}
\put(150, 0){\line(1,-1){50}}
\put(100, 50){\line(1,-1){50}}
\put(210,-45){\makebox(0,0){$\fr{a}'_3 $}}
\put(150, 0){\line(1,-2){30}}
 \put(120,60){\line(1,-2){30}}
 \put(195,-60){\makebox(0,0){$\fr{a}'_4 $}}
 \put(160,80){\makebox(0,0){$(b) \   M=G/K \ \mbox{of Type IIb}$}}
 \put(190, 44){\makebox(0,0){$C'_2 $}}
 \put(190,-15){\makebox(0,0){$C'_3 $}}
 \put(182,-42){\makebox(0,0){$C'_4 $}}
 \put(162,-45){\makebox(0,0){$C'_1 $}}
 \put(155, 5){\line(1,0){40}}
\put(155, 10){\line(1,0){40}}
\put(155, 15){\line(1,0){40}}
\put(155, 20){\line(1,0){40}}
\put(155, 25){\line(1,0){40}}
\put(155, 30){\line(1,0){40}}
\put(155, 35){\line(1,0){40}}
\end{picture}
\end{center}

\vspace{2cm}
\caption{$\fr{t}$-chambers}
\label{fig:1}       
\end{figure}

\subsection{$G$-invariant K\"ahler-Einstein metrics}
 Let $M=G/K$ be a generalized flag manifold defined  by a pair $(\Pi, \Pi_K)$,
 such that
  $\Pi_{M}=\Pi\backslash \Pi_{K}=\{\al_{i_{1}}, \ldots, \al_{i_{r}}\}$ 
where $1\leq i_{1}\leq \cdots\leq i_{r}\leq \ell$.   
Let $J$ be the $G$-invariant 
complex structure   on $M$ corresponding to an 
invariant ordering $R_{M}^{+}$ in $R_{M}$.  Then the  
  metric    given by
\begin{equation}\label{KE}
g=\{g_{\al}=c\cdot (\delta_{\fr{m}}, \al) \ (c\in\mathbb{R}) : \al\in R_{M}^+  \},
\end{equation}
where $\delta_{\fr{m}}=\frac{1}{2}\sum_{\al\in R_{M}^+}\al$, is a K\"ahler-Einstein metric (up to   constant) on $M$ (cf. \cite{Bor}, \cite{Arv1}). 
Here $g_{\al}=g(E_{\al}, E_{-\al})\in \mathbb{R}^+$, where
$E_{\al}$ are the (normalized) root vectors defined by $\al\in R_{M}^+$.  
Note that the positive
numbers $g_{\al}$ satisfy $g_{\al}=g_{\be}$  
 if $\al|_{\fr{t}}=\be|_{\fr{t}}$, $(\al, \be\in R_{M}^{+})$.

The  1-form $\delta_{\fr{m}}$ is   called  {\it Koszul form}, 
and it depends on the choise of $R_{M}^+$.
It is well known that (cf.  \cite{B})  
\begin{equation}\label{delta}
\delta_{\fr{m}}=c_{i_{1}}\Lambda_{i_{1}}+\cdots +c_{i_{r}}\Lambda_{i_{r}},
\end{equation}
where $c_{i_{1}}>0, \ldots, c_{i_{r}}>0$, and $\Lambda_{i_{1}}, \ldots, \Lambda_{i_{r}}$ 
are the fundamental weights   corresponding to the simple roots of $\Pi_{M}$.
The positive coefficients $c_{i_{1}}, \ldots, c_{i_{r}}$ 
are known as {\it Koszul numbers}. For  flag manifolds corresponding to  
classical Lie groups 
Koszul numbers were computed in \cite[p.~176]{AP}  (see also \cite[p.~ 21]{AS}). 
 
Let  $R_{\fr{t}}^+=\{ \xi_1, \ldots,  \xi_{s}\}$, and  consider the decomposition  $\fr{m}=\fr{m}_1\oplus\cdots\oplus\fr{m}_s$, where 
 $\fr{m}_i=(\fr{m}_{\xi_i}\oplus  \fr{m}_{-\xi_i})^{\tau}$ with $1\leq i\leq s$.  
Then by  (\ref{KE}),   a K\"ahler-Einstein metric is expressed as
 \begin{equation}\label{ke}
 g=g_{\xi_1}\cdot (-B)|_{\fr{m}_1}+\cdots + g_{\xi_s}\cdot (-B)|_{\fr{m}_s},
 \end{equation}
 where the positive numbers $g_{\xi_i}$ are given by $g_{\xi_i}=(\delta_{\fr{m}}, \al)$  (up to a constant),
 and the complementary roots $\al\in R_{M}^+$ are such that
 $\al|_{\fr{t}}=\xi_i$.  Note that  it suffices to 
 work only with the  root $\al\in R_{M}^+$ which is the lowest weight of the 
 corresponding $\ad(\fr{k}^{\mathbb{C}})$-irreducible submodule
 $\fr{m}_{\xi_i}$ of $\fr{m}^{\mathbb{C}}$. 
 This is because of a bijection  between $\fr{t}$-roots $\xi$ 
 and lowest weights of the
 irreducible submodules $\fr{m}_{\xi}$
  (\cite{AP}, \cite{Ale1}).  The lowest weight  of 
 the irreducible submodule  $\fr{m}_{\xi_i}$ $(\xi_{i}\in R_{\fr{t}}^+)$ 
 is the positive complementary root $\al\in R_{M}^+$
 such that $\al-\phi\notin R$ for any $\phi\in R_{K}^+$.  Such roots are 
 also called {\it $K$-simple roots}. For example, all roots of $\Pi_{M}$ are  
 $K$-simple (\cite{AP}).

Flag manifolds 
$M=G/K$ of Type I  admit a unique K\"ahler-Einstein metric.
 
   \begin{theorem} \label{The3} \textnormal{(\cite{B})} 
Let $M=G/K$ be a generalized flag manifold
  of Type I, and let $B$ be the Killing form of $G$.  Then $M$ admits a 
unique $G$-invariant K\"ahler-Einstein metric given by
\[
\left\langle \ , \ \right\rangle = 1 \cdot (-B)|_{\fr{m}_1}+
2\cdot (-B)|_{\fr{m}_2}+3\cdot (-B)|_{\fr{m}_3}+4\cdot (-B)|_{\fr{m}_4}
\]
\end{theorem}

It follows from Theorem \ref{The2} that flag manifolds of Type II  
admit four (up to scale) K\"ahler-Einstein metrics.   
 We will find one such metric by using (\ref{ke}), and  
 the other three K\"ahler-Einstein metrics can be obtainded by a similar method.

\begin{theorem}\label{The4}
Let $M=G/K$ be a generalized flag manifold
of Type IIa (resp. Type IIb) given in Table 4.  Let $J$ (resp. $J'$) be the $G$-invariant complex structure  corresponding to  
 the natural invariant ordering  induced by the $\fr{t}$-chamber $C_2$  (resp. $C'_2$) (cf. Theorem \ref{The2}).
 Then  the metric $g=(x_1, x_2, x_3, x_4)$  given below  is (up to scalar) a 
 $G$-invariant K\"ahler-Einstein metric  
 on $M$ with respect to $J$ (resp. to $J'$).
 \begin{eqnarray*}
 M=G/K \ \mbox{of Type IIa}    &:&   \left\{ \begin{tabular}{ll} 
 $  SO(2\ell+1)  :$ & $ g=(1, \ \ell-3/2, \ \ell-1/2, \ 2\ell-2)$ \\ 
$     SO(2\ell)_{(i)} :$  & $g=(1, \ \ell-2, \ \ell-1, \ 2\ell-3)$  \\
$    E_6 :$  &  $g=(1,\  4,\ 5,\ 9)$  \\
 $  E_7 :$ & $g=(1,\  6,\ 7,\ 13)$ 
 \end{tabular}\right. \\
M=G/K \ \mbox{of Type IIb} &:&   \left\{  \begin{tabular}{ll} 
   $    Sp(\ell)  :$ & $g= ( \ell/{2}, \  \ell-p+1, \  (3\ell/2)-p+1, \ 2\ell-p+1)$  \\
 $   SO(2\ell)_{(ii)}     :$ & $ g= (\ell/2, \ \ell-p-1, \ (3\ell/2)-p-1, \ 2\ell-p-1)$
 \end{tabular}\right.
\end{eqnarray*}
  \end{theorem}

\begin{proof}
     The Koszul form is given as follows: 	
     \[
\begin{tabular}{lll}
 $  G=SO(2\ell+1)$  & :  & $\delta_{\fr{m}}=\Lambda_1+(\ell-3/2)\Lambda_2$ \\
 $ G=SO(2\ell)_{(i)}$    & :  & $\delta_{\fr{m}}=\Lambda_1+(\ell-2)\Lambda_2$ \\
  $  G=Sp(\ell)$  & :  & $\delta_{\fr{m}}=(\ell/{2})\Lambda_{p}+((\ell-p+1)/2)\Lambda_{\ell}$ \\
 $  G=SO(2\ell)_{(ii)}$    & :  & $\delta_{\fr{m}}=(\ell/{2})\Lambda_{p}+(\ell-p-1)\Lambda_{\ell}$  \\
 $G=E_6$ & : & $ \delta_{\fr{m}}=\Lambda_1+4\Lambda_2$\\
 $G=E_7$ & : & $ \delta_{\fr{m}}=\Lambda_1+6\Lambda_2$ 
 \end{tabular}
\]
The first four has been computed in \cite[p.~ 21]{AS}. We will give a computation for 
    the exceptional Lie group 
  $E_{6}$ and  the case of $E_7$ can be treated similarly. 
 We follow  the notation of Proposition \ref{classification}.
Let $M=G/K=E_6/SU(5)\times U(1)\times U(1)$ with $\Pi_{M}=\{\al_1, \al_2\}$, 
and $\Pi_{K}=\{\al_3, \al_4, \al_5, \al_6\}$. 
 According to (\ref{delta}) it is $\delta_{\fr{m}}=c_1\Lambda_1+c_2\Lambda_2$. 
By using the natural invariant ordering $R_{M}^+$ given in  (\ref{E6}) we easily obtain that 
\begin{equation}\label{dE}
2\delta_{\fr{m}}=\sum_{\al\in R_{M}^+}\al=16 \al_1+30 \al_2+36 \al_3+24 \al_4+12 \al_5+18 \al_6.
\end{equation}
The Cartan matrix $A=(A_{ij}) =(2(\al_i, \al_j)/(\al_j, \al_j))$ of $E_6$  (with respect to the   base 
$\Pi=\{\al_1, \ldots, \al_6\}$) is given by (cf. \cite[p.~ 82]{Sam})
 \[
A =  \left( \begin{tabular}{cccccc}
 2  & -1 &  0 &  0 &  0 &  0 \\
-1  &  2 & -1 &  0 &  0 &  0 \\
 0  & -1 &  2 & -1 &  0 & -1 \\
 0  &  0 & -1 &  2 & -1 &  0 \\
 0  &  0 & 0 &-1 & 2 & 0 \\
 0 & 0 & -1 & 0 & 0 & 2
\end {tabular}
\right).   
\]
Recall  that  the Cartan matrix $A=(A_{ij})$  depends on the enumeration of the system of simple roots $\Pi$, and it establishes the relation between the simple roots and the
fundamental weights. In particular, it is 
$\al_{i}=\sum_{j=1}^{6}A_{ij}\Lambda_{j}$, where $\Lambda_j$ $(j=1, \ldots, 6)$ are the fundamental weights of $E_6$.
 By using this remark  and  relation (\ref{dE}), we easily obtain that
$ \delta_{\fr{m}}=\Lambda_1+4\Lambda_2$, that is $c_1=1$ and $c_2=4$. 

We now proceed to the computation of the corresponding K\"ahler-Einstein metric.  From Proposition \ref{classification} we have that
$R_{\fr{t}}^{+}= \{\overline{\al}_1, \overline{\al}_2, 
\overline{\al}_1+\overline{\al}_2, \overline{\al}_1+2\overline{\al}_2\}$, thus  according to (\ref{ke})  the K\"ahler-Einstein metric $g$ 
which corresponds to the invariant complex structure $J$ induced by $R_{M}^+=R^+\backslash R_{K}^+$, is given by
\[
g=g_{\overline{\al}_1}\cdot (-B)|_{\fr{m}_1}+g_{\overline{\al}_2}\cdot (-B)|_{\fr{m}_2}+g_{(\overline{\al}_1+\overline{\al}_2)}\cdot (-B)|_{\fr{m}_3}
+g_{(\overline{\al}_1+2\overline{\al}_2)}\cdot (-B)|_{\fr{m}_4},
\]
where $\fr{m}_i$ $(i=1, \ldots, 4)$ are given by (\ref{e66}).
It follows from (\ref{E6}) that the lowest weights of $\fr{m}_1, \fr{m}_2, \fr{m}_3$ and $\fr{m}_4$ (or the corresponding  $K$-simple roots) are the complementary roots
$\al_1, \al_2, \al_1+\al_2,$ and $\al_1+2\al_2+2\al_3+\al_4+\al_6$, respectively.  By using the relation (cf. \cite[p.~ 174]{AP})
\begin{equation}\label{fundamental}
2(\Lambda_{i}, \al_{j})/(\al_j, \al_j)=\left\{
\begin{tabular}{lc}
$\delta_{ij},$ & $\mbox{if} \ \ \al_j\in \Pi_{M}$ \\
$0,$   & $\mbox{if} \ \ \al_j\in \Pi_{K}$
\end{tabular}\right.
\end{equation}
we easily obtain that  
\begin{eqnarray*}
 g_{\overline{\al}_1} &=& (\Lambda_1+4\Lambda_2, \al_1) = (\Lambda_1, \al_1) = (\al_1, \al_1)/2 =1/2, \\
 g_{\overline{\al}_2} &=& (\Lambda_1+4\Lambda_2, \al_2) = 4(\Lambda_2, \al_2) = 2(\al_2, \al_2)=2, \\
 g_{(\overline{\al}_1+\overline{\al}_2)} &=& (\Lambda_1+4\Lambda_2,\al_1+\al_2) = (\Lambda_1, \al_1)+4(\Lambda_2, \al_2)=5/2, \\
 g_{(\overline{\al}_1+2\overline{\al}_2)} &=&  (\Lambda_1+4\Lambda_2, \al_1+2\al_2+2\al_3+\al_4+\al_6) \\
 &=& (\Lambda_1, \al_1)+8(\Lambda_2, \al_2) =  (\al_1, \al_1)/2+4(\al_2, \al_2)= 9/2,
  \end{eqnarray*}
 where we have set $(\al_i, \al_i)=1$ for $i=1, \ldots, 6$ (recall that all simple roots in $E_6$ have the same length).
 Thus an invariant K\"ahler-Einstein metric on $M$ is given (up to a constant) by $g =(1/2, 2, 5/2, 9/2)$, and if we normalize,
  we obtain the metric stated in the theorem.
 
 The other K\"ahler-Einstein metrics presented in  the theorem are obtained by a similar procedure.
 \end{proof}

\markboth{Andreas Arvanitoyeorgos and Ioannis Chrysikos}{Homogeneous Einstein metrics  on  flag manifolds  with four isotropy summands}
\section{Invariant Einstein metrics on generalized flag manifolds of Type I}
\markboth{Andreas Arvanitoyeorgos and Ioannis Chrysikos}{Homogeneous Einstein metrics  on  flag manifolds  with four isotropy summands}

In this section we will find all $G$-invariant Einstein metrics
for generalized flag manifolds $M=G/K$ of Type I. We will use Proposition \ref{Ricc}
and the notation of Section 1 to obtain  the Einstein equation. The  spaces 
 of Type II will be examined in Section 5.

 \begin{prop}\label{componentsI}
Let $M=G/K$ be a  generalized flag manifold of Type I.
Then the components $r_{i}=\Ric|_{\fr{m}_i}$ of the Ricci tensor of a $G$-invariant Riemannian  metric (\ref{metrI}) on $M$
are given by 
\begin{equation}\label{compI}
\left\{ 
\begin{tabular}{l}
$r_1=\displaystyle\frac{1}{2x_1}-\frac{c_{11}^{2}}{2d_1}\frac{x_2}{x_1^2}+
\frac{c_{12}^{3}}{2d_1}\Big(\frac{x_1}{x_2x_3}-\frac{x_2}{x_1x_3}-\frac{x_3}{x_1x_2}\Big)+
\frac{c_{13}^{4}}{2d_1}\Big(\frac{x_1}{x_3x_4}-\frac{x_3}{x_1x_4}-\frac{x_4}{x_1x_3}\Big)$ \\
$r_2=\displaystyle\frac{1}{2x_2}-\frac{c_{22}^{4}}{2d_2}\frac{x_4}{x_2^2}+
\frac{c_{11}^{2}}{4d_2}\Big(\frac{x_2}{x_1^2}-\frac{2}{x_2}\Big)+
\frac{c_{12}^{3}}{2d_2}\Big(\frac{x_2}{x_1x_3}-\frac{x_1}{x_2x_3}-\frac{x_3}{x_1x_2}\Big)$ \\
$r_3=\displaystyle\frac{1}{2x_3}+
\frac{c_{12}^{3}}{2d_3}\Big(\frac{x_3}{x_1x_2}-\frac{x_2}{x_1x_3}-\frac{x_1}{x_2x_3}\Big)+
\frac{c_{13}^{4}}{2d_3}\Big(\frac{x_3}{x_1x_4}-\frac{x_1}{x_3x_4}-\frac{x_4}{x_1x_3}\Big)$ \\
$r_4=\displaystyle\frac{1}{2x_4}+
\frac{c_{22}^{4}}{4d_4}\Big(\frac{x_4}{x_2^2}-\frac{2}{x_4}\Big)+
\frac{c_{13}^{4}}{2d_4}\Big(\frac{x_4}{x_1x_3}-\frac{x_1}{x_3x_4}-\frac{x_3}{x_1x_4}\Big)$, \\
\end{tabular}\right. 
\end{equation}
where $ c_{11}^{2}=[112], c_{12}^{3}=[123], c_{13}^{4}=[134]$, and $c_{22}^{4}=[224]$.
\end{prop}
\begin{proof}
It  suffices to show that the only non-zero triples $[ijk]$ are 
those given above. This is an immediate consequense of the
  following relations (cf. \cite{It}) 
\begin{equation}\label{relations}
\left.
\begin{tabular}{llll}
$ [\fr{m}_1, \fr{m}_1]\subset\fr{k}\oplus\fr{m}_2,$ & 
$[\fr{m}_2, \fr{m}_2]\subset\fr{k}\oplus\fr{m}_4,$ & 
$[\fr{m}_3, \fr{m}_3]\subset\fr{k},$ & 
$[\fr{m}_4, \fr{m}_4]\subset\fr{k},$ \\
$[\fr{m}_1, \fr{m}_2]\subset\fr{m}_1\oplus\fr{m}_3,$ &  
$[\fr{m}_1, \fr{m}_3]\subset\fr{m}_2\oplus\fr{m}_4,$ & 
$[\fr{m}_1, \fr{m}_4]\subset\fr{m}_3,$ & \\
$[\fr{m}_2, \fr{m}_3]\subset\fr{m}_1,$ & 
$ [\fr{m}_2, \fr{m}_4]\subset\fr{m}_2,$ & 
$[\fr{m}_3, \fr{m}_4]\subset\fr{m}_1,$ &
\end{tabular}\right\} 
\end{equation}
and the fact that
$B(\fr{m}_i, \fr{m}_j)=B(\fr{m}_i, \fr{k})=0$.  Then the result  is a 
direct application of Proposition \ref{Ricc}.
\end{proof}

 The real dimensions   of the real irreducible 
$\ad(\fr{k})$-submodules $\fr{m}_i$, $(i=1,\ldots, 4)$ of $\fr{m}$ (corresponding to the positive $\fr{t}$-root $\xi_i\in R_{\fr{t}}^+$) 
can be obtained by  various  methods. 
 One such was indicated in Remark \ref{rem}, and uses the 
 explicit description of $R_{M}$.  Another method uses the  Weyl dimensional formula  (cf. \cite[p.~94]{GOV}, \cite[p.~122]{Sam}), which in our case is given by 
 \[
\dim_{\mathbb{C}}\fr{m}_{i}= \prod_{\al\in R_{K}^{+}}\Big(1+\frac{(\lambda_{i}, \al)}{(\delta_{K},\al)}\Big).
\]
Here,   $\lambda_i$ is  the corresponding 
highest weight of the $\ad(\fr{k})$-submodule $\fr{m}_i$, and $\delta_{K}=\frac{1}{2}\sum_{\be\in R_{K}^{+}}\be$  where $R_{K}^+$ denotes the positive root system of the isotropy  group $K$ corresponding to
 the flag manifold $M=G/K$.
The highest weight $\lambda_i$ 
is the  positive complementary root $\al=\sum_{i=1}^{\ell} c_{i}\al_{i}\in R_{M}^+$, such that $\kappa(\al)=\xi_i$, and $c_{i}\geq c_{i}'$ whenever $\al'=\sum_{i=1}^{\ell}c_{i}'\al_{i}\in R_{M}^+$ with $\kappa(\al')=\xi_{i}$.     
A useful method to work with the above formula is  to express the highest weight  $\lambda_i$ in terms of the fundmental weights  by using the  Cartan matrix of $\fr{g}$, and then use relation (\ref{fundamental}).  Therefore, we obtain the following table:

\bigskip
   \begin{center}
  \begin{tabular}{|l|l|l|l|l|}
\hline 
$G$ & $d_1$ & $d_2$ & $d_3$ & $d_4$\\
\hline 
$F_4$ & $12$ & $18$ & $4$ & $6$\\
\hline
$E_7$ & $48$ & $36$ & $16$ & $6$\\
\hline 
${E_8}_{(i)}$ & $96$ & $60$ & $32$ & $6$\\
\hline
${E_8}_{(ii)}$ & $84$ & $70$ & $28$ & $14$\\
\hline
\end{tabular}
 \end{center}
 \smallskip
 
 \begin{center}
 {\sc Table 5.} \ The  dimensions of the isotropy submodules for $M=G/K$ of Type I. 
 \end{center} 
 
\medskip
 A $G$-invariant invariant  metric 
 $g=(x_1, x_2, x_3, x_4)$ on  a flag manifold $M=G/K$ of Type I is Einstein if and only if,  there is a positive 
 constant $e$  such that $r_1=e, r_2=e, r_3=e, r_4=e,$
 or equivalently,
\begin{equation}\label{systemI}
   r_1-r_2=0, \quad r_2-r_3=0, \quad r_3-r_4=0, 
\end{equation}
where $r_{i} \ (i=1, \ldots, 4)$ are determined by Proposition \ref{componentsI}.

In order to solve system (\ref{systemI}), first we need   to determine the non zero triples $c_{11}^{2}, c_{12}^{3},  c_{13}^{4}$, and $c_{22}^{4}$ of Proposition \ref{componentsI}.  
By Theorem  \ref{The3} the metric $x_1=1, x_2=2, x_3=3$, $x_4=4$ is K\"ahler-Einstein, thus by using (\ref{compI}) and Table 5 we obtain 
the following:
\begin{eqnarray}
 F_4 &:& c_{11}^{2}=c_{22}^{4},  \ c_{12}^{3}=(4-c_{22}^{4})/2, \ c_{13}^{4}=(10-3c_{22}^{4})/6,  \label{eq1}\\ 
 E_7 &:& c_{11}^{2}=6+c_{22}^{4}, \ c_{12}^{3}=(10-c_{22}^{4})/2, \ c_{13}^{4}=(14-3c_{22}^{4})/6, \label{eq2} \\
 {E_{8}}_{(i)} &:& c_{11}^{2}=14+c_{22}^{4}, \  c_{12}^{3}=(18-c_{22}^{4})/2,  \ c_{13}^{4}=(26-5c_{22}^{4})/10, \label{eq3}\\
  {E_{8}}_{(ii)} &:& c_{11}^{2}=(28+3c_{22}^{4})/3, \  c_{12}^{3}=(56-3c_{22}^{4})/6,  \ c_{13}^{4}=(154-15c_{22}^{4})/30. \label{eq4}
\end{eqnarray}

 In order to compute $c_{22}^{4}=[224]$, we  use the twistor fibration of generalized flag manifolds  $M=G/K$  over   symmetric spaces $G/U$ (\cite[p.~48]{Bur}, \cite{Bur1}).

Let $M=G/K$ be a generalized flag manifold of Type I, 
with reductive decomposition $\fr{g}=\fr{k}\oplus\fr{m}$, 
where $\fr{m}=\fr{m}_1\oplus\fr{m}_2\oplus\fr{m}_3\oplus\fr{m}_4$.  
Set $\fr{u}=\fr{k}\oplus\fr{m}_2\oplus\fr{m}_4$  and 
$\fr{f}=\fr{m}_1\oplus\fr{m}_3$.
Relations (\ref{relations}) imply that
\[
[\fr{u}, \fr{u}]\subset\fr{u},  \qquad [\fr{u}, \fr{f}]\subset 
\fr{f}, \qquad [\fr{f}, \fr{f}]\subset\fr{u},
\]
therefore $\fr{u}$ is a subalgebra of $\fr{g}$,  $\fr{g}=\fr{u}\oplus\fr{f}$ is a reductive decomposition of $\fr{g}$, and  the pair $(\fr{g}, \fr{u})$ 
is an irreducible symmetric pair. This determines an irreducible symmetric space  of compact
type  $G/U$, where $U$ is the connected Lie subgroup of $G$ 
corresponding to $\fr{u}$. Since $\fr{k}\subset\fr{u}$, 
we obtain  the natural fibration 
$\pi : G/K\to G/U$, $gK\mapsto gU$ with fiber  $U/K$.

By using the classification of irreducible symmetric spaces 
(e.g. \cite[p.~518]{Hel} or \cite[p.~201]{Be}) and by comparing dimensions, 
we can determine explicity the Lie algebra $\fr{u}$ 
and the associated symmetric space $G/U$.  These are given in the following table:

\medskip
\begin{center}  
 \begin{tabular}{|l|l|l|}
\hline
$G$ & $\fr{u}$ & $\mbox{symmetric space} \ G/U$ \\
\hline\hline
$F_4$ & $\fr{so}(9)$ & $F_4/SO(9)$ \\
$E_7$ & $\fr{so}(12)\oplus \fr{su}(2)$ & $E_7/SO(12)\times SU(2)$\\
${E_8}_{(i)}$ & $\fr{so}(16)$ & $E_8/SO(16) $\\
${E_8}_{(ii)}$ & $\fr{e}_7\oplus\fr{su}(2)$ & $E_8/E_7\times SU(2)$\\
\hline
\end{tabular}
\end{center}

\smallskip
\begin{center}
 {\sc Table 6.} \ The  twistor fibration $\pi: G/K\to G/U$ for flag manifolds of type I. 
 \end{center}
 
 \medskip

Note that on the fiber $U/K$  the Lie group $U$ may 
not act (almost) effectively, 
that is  $\fr{u}$ and $\fr{k}$ may have non-trivial 
(non-discrete) ideals in common.
Let $U'$ be the normal  subgroup of $U$ which acts 
(almost) effectively
on the fiber $U/K$, with isotropy group $K'$. Then $U/K=U'/K'$ (\cite[p.~179]{Be}).  
In our  case the fibers $U'/K'$  are the spaces $SO(9)/U(3)\times SO(3)$, 
$SO(12)/U(3)\times SO(6)$, $SO(16)/U(3)\times SO(10)$, 
and $E_7/SU(7)\times U(1)$  respectively.
These    are generalized flag manifolds with two 
isotropy summands (see \cite{Chry}),
so we can easily compute the triples $(c_{22}^{4})'$ 
for the spaces $U'/K'$.
  
Let $B_{G}=B$ 
and $B_{U'}$ denote the Killing forms  of $G$ and $U'$ respectively, and 
let $\fr{u}'=\fr{k}'\oplus\fr{m}'$ be a 
 reductive decomposition of $\fr{u}'$ with respect to 
$B_{U'}$, where $\fr{m}'\cong T_{eK'}(U'/K')$.  
Then it is evident that $\fr{m}'=\fr{m}_2\oplus\fr{m}_4$,
and $\fr{m}'$ is a linear subspace of $\fr{m}=\oplus_{k=1}^{4}\fr{m}_k$.  
In \cite{Chry} it was shown that 
$(c_{22}^{4})'=\displaystyle\frac{d_2d_4}{d_2+4d_4}$, 
so by using Table 5 we obtain the following values:
\begin{equation}\label{Kill1}
F_4: \ (c_{22}^{4})'=18/7, \quad  E_7 : \ (c_{22}^{4})'=18/5, 
\quad {E_{8}}_{(i)} : \ (c_{22}^{4})'=30/7, \quad {E_{8}}_{(ii)} : \ (c_{22}^{4})'=70/9.
\end{equation}
Since $U'\subset G$ is a simple Lie subgroup of $G$,
there is a positive number $c$ such that $B_{U'}=c\cdot B_{G}$ 
(cf. \cite[p.~260]{Be}). In particular,  we obtain 
the following (cf. Appendix of \cite{Brb}):
\begin{equation}\label{Kill2}
\begin{tabular}{llll}
$ F_4 :$ &  $c=\displaystyle\frac{B_{SO(9)}}{B_{F_4}}=14/18,$ 
& $ E_7:$ & $c=\displaystyle\frac{B_{SO(12)}}{B_{E_7}}=20/36,$\\
$ {E_{8}}_{(i)}:$ & $c=\displaystyle\frac{B_{SO(16)}}{B_{E_8}}=28/60,$ 
& ${E_{8}}_{(ii)}:$ & $c=\displaystyle\frac{B_{E_7}}{B_{E_8}}=36/60.$
\end{tabular}
\end{equation}
The relation between the triples $(c_{22}^{4})'$ and 
$c_{22}^{4}$, for the flag manifolds $U'/K'$ and $G/K$  respectively, 
is now given as follows:
\begin{lemma}\label{Lem1}
Let $c>0$ defined by $B_{U'}=c\cdot B_{G}$. 
Then $c_{22}^{4}=c\cdot (c_{22}^{4})'$.
\end{lemma}
\begin{proof}
Let $\{X_{j}^{(k)}\}_{j=1}^{d_{k}}$ be a $B_{G}$-orthogonal basis on $\fr{m}_{k}$,
 where $d_{k}=\dim\fr{m}_{k}$ and $k\in\{1, 2, 3, 4\}$. 
Then the set
$\{Y_{j}^{(k)}=\displaystyle\frac{1}{\sqrt{c}}X_{j}^{(k)}\}_{j=1}^{d_{k}}$ 
is a $B_{U'}$-orthogonal basis 
on $\fr{m}_{k}$ for $k\in\{2, 4\}$.
 Following  the notation of Section 1, we obtain that
\begin{eqnarray*}
(c_{22}^{4})'&=& \sum_{i, j} \big(B_{U'}([Y_{i}^{(2)}, Y_{i}^{(2)}], Y_{j}^{(4)})\big)^{2}
=\sum_{i, j} \big(B_{U'}([\frac{1}{\sqrt{c}}X_{i}^{(2)}, \frac{1}{\sqrt{c}}X_{i}^{(2)}], \frac{1}{\sqrt{c}}X_{j}^{(4)})\big)^{2}\\
&=&\sum_{i, j}\big(\frac{1}{c\sqrt{c}}B_{U'}([X_{i}^{(2)}, X_{i}^{(2)}], X_{j}^{(4)})\big)^{2}
=\sum_{i, j}\big(\frac{1}{\sqrt{c}}B_{G}([X_{i}^{(2)}, X_{i}^{(2)}], X_{j}^{(4)})\big)^{2}\\
&=& \frac{1}{c}\sum_{i, j}\big(B_{G}([X_{i}^{(2)}, X_{i}^{(2)}], X_{j}^{(4)})\big)^{2}=\frac{1}{c}\cdot c_{22}^{4},
\end{eqnarray*}
where $1\leq i\leq d_2$, $1\leq j\leq d_4$.  Hence $c_{22}^{4}= c\cdot(c_{22}^{4})'$.
\end{proof}

From Lemma \ref{Lem1} and     relations (\ref{Kill1}) and (\ref{Kill2}), 
we can find the value $c_{22}^{4}=[224]$ for all flag manifolds of Type I,
so from (\ref{eq1})-(\ref{eq4}), we obtain the following  table:

\smallskip
  \begin{center}
\begin{tabular}{|l|l|l|l|l|}
\hline 
$G$ & $c_{22}^{4}=[224]$ & $c_{11}^{2}=[112]$ & $c_{12}^{3}=[123]$ & $c_{13}^{4}=[134]$\\
\hline 
$F_4$ & $2$ & $2$ & $1$ & $2/3$\\
\hline
$E_7$ & $2$ & $8$ & $4$ & $4/3$\\
\hline 
${E_8}_{(i)}$ & $2$ & $16$ & $8$ & $8/5$\\
\hline
${E_8}_{(ii)}$ & $14/3$ & $14$ & $7$ & $14/5$\\
\hline
\end{tabular}
\end{center}
\smallskip
\begin{center}
{\sc{Table 7.}} \ Values of the unknown triples in Proposition \ref{componentsI}.
\end{center} 
 \medskip
 
 It is now evident that by using Tables 5 and 7,   the components  
 of the Ricci tensor (\ref{compI}) for a flag manifold $M=G/K$ of Type I  
 are completely determined.  Thus by setting $x_1=1$ and  solving the system (\ref{systemI}), we obtain the following:
 \begin{theorem}\label{The5}
(1) \ Let $M=G/K$ be a generalized flag manifold of Type I associated to the exceptional Lie groups
$F_4, E_7$, and ${E_{8}}_{(ii)}$.
Then $M$ admits (up to scale) three $G$-invariant Einstein metrics.  One is 
a K\"ahler-Einstein metric given by $g=(1, 2, 3, 4)$, and the other two are non-K\"ahler 
given approximatelly as  follows:
  \[
  \begin{tabular}{c|c|c}
  $G$    &  $g_1=(x_1, x_2, x_3, x_4)$ & $g_2=(x_1, x_2, x_3, x_4)$\\
  \hline\hline
  $F_4$  &  $(1, \ 1.2761, \ 1.9578,  \ 2.3178)$ & $(1, \ 0.9704, \ 0.2291, \ 1.0097)$ \\
  $E_7$  &  $(1, \ 0.8233, \ 1.2942,  \ 1.3449)$ & $(1, \  0.9912, \ 0.5783, \ 1.1312)$ \\
  ${E_8}_{(ii)}$ &  $(1, \ 0.9133, \ 1.4136, \ 1.5196)$ & $(1, \ 0.9663, \ 0.4898, \ 1.0809)$
  \end{tabular}
  \]
(2) \ If  $M=E_8/SO(10)\times SU(3)\times U(1)$, i.e. the flag manifold of Type  I correpsonding to  ${E_8}_{(i)}$,
 then $M$ admits (up to scale) five $E_8$-invariant Einstein metrics.  
 One is a K\"ahler-Einstein metric given by $g=(1, 2, 3, 4)$, and the other four are non-K\"ahler  given approximatelly
as follows:
\[
\begin{tabular}{cc}
 $g_1=(1, \ 0.6496, \ 1.1094, \ 1.0610)$, & $g_3=(1, \ 1.0970, \ 0.7703, \ 1.2969)$, \\
 $g_2=(1, \ 1.1560, \ 1.0178, \ 0.2146)$, & $g_4=(1, \ 0.7633, \ 1.0090, \ 0.1910)$. 
  \end{tabular}
\]
\end{theorem}

\markboth{Andreas Arvanitoyeorgos and Ioannis Chrysikos}{Invariant Einstein metrics  on  flag manifolds with four isotropy summands}
\section{Invariant Einstein metrics on generalized flag manifolds of type II}
\markboth{Andreas Arvanitoyeorgos and Ioannis Chrysikos}{Invariant Einstein metrics  on  flag manifolds with four isotropy summands}

\subsection{Calculation of the Ricci tensor}
In this section we will investigate $G$-invariant  Einstein metrics on flag manifolds $M=G/K$ of Type II.
We will apply Proposition \ref{Ricc} to  compute the Ricci tensor  for a   $G$-invariant Riemannian metric on $M$ detrermined by (\ref{metrI}).
 \begin{prop}\label{componentsII}
 Let $M=G/K$ be a generalized flag manifold of Type IIa. Then
the components $r_{i}$ of the Ricci tensor   associated to the   metric 
$\left\langle \ , \ \right\rangle$  given in (\ref{metrI}), are the following:
\begin{equation}\label{compII}
  \ \  \left.
\begin{tabular}{l}
$r_1=\displaystyle\frac{1}{2x_1}+  \frac{c_{12}^3}{2d_1}\Big( \frac{x_1}{x_2x_3}- \frac{x_2}{x_1x_3}- \frac{x_3}{x_1x_2}\Big)$ \\
$r_2=\displaystyle\frac{1}{2x_2}+  \frac{c_{12}^3}{2d_2}\Big( \frac{x_2}{x_1x_3}- \frac{x_1}{x_2x_3}- \frac{x_3}{x_1x_2}\Big)+  \frac{c_{23}^4}{2d_2}\Big( \frac{x_2}{x_3x_4}- \frac{x_4}{x_2x_3}- \frac{x_3}{x_2x_4}\Big)$ \\
$r_3=\displaystyle\frac{1}{2x_3}+  \frac{c_{12}^3}{2d_3}\Big( \frac{x_3}{x_1x_2}- \frac{x_2}{x_1x_3}- \frac{x_1}{x_2x_3}\Big)+  \frac{c_{23}^4}{2d_3}\Big( \frac{x_3}{x_2x_4}- \frac{x_4}{x_2x_3}- \frac{x_2}{x_3x_4}\Big)$ \\
$r_4=\displaystyle\frac{1}{2x_4}+  \frac{c_{23}^4}{2d_4}\Big( \frac{x_4}{x_2x_3}- \frac{x_3}{x_2x_4}- \frac{x_2}{x_3x_4}\Big),$ \\
\end{tabular}\right\} 
\end{equation}
where $c_{12}^3=[123]$ and $c_{23}^4=[234]$. 

If $M=G/K$ is of Type IIb, then
the components $r_{i}$ of the Ricci tensor $\Ric$ associated to the  
$\left\langle \ , \ \right\rangle$ given in (\ref{metrI}), are the following:
\begin{equation}\label{compIII}
 \ \ \left.
\begin{tabular}{l}
$r_1=\displaystyle\frac{1}{2x_1}+  \frac{c_{12}^3}{2d_1}\Big( \frac{x_1}{x_2x_3}- \frac{x_2}{x_1x_3}- \frac{x_3}{x_1x_2}\Big)+  \frac{c_{13}^4}{2d_1}\Big( \frac{x_1}{x_3x_4}- \frac{x_4}{x_1x_3}- \frac{x_3}{x_1x_4}\Big)$ \\
$r_2=\displaystyle\frac{1}{2x_2}+  \frac{c_{12}^3}{2d_2}\Big( \frac{x_2}{x_1x_3}- \frac{x_1}{x_2x_3}- \frac{x_3}{x_1x_2}\Big)$ \\
$r_3=\displaystyle\frac{1}{2x_3}+  \frac{c_{12}^3}{2d_3}\Big( \frac{x_3}{x_1x_2}- \frac{x_2}{x_1x_3}- \frac{x_1}{x_2x_3}\Big)+  \frac{c_{13}^4}{2d_3}\Big( \frac{x_3}{x_1x_4}- \frac{x_4}{x_1x_3}- \frac{x_1}{x_3x_4}\Big)$ \\
$r_4=\displaystyle\frac{1}{2x_4}+  \frac{c_{13}^4}{2d_4}\Big( \frac{x_4}{x_1x_3}- \frac{x_3}{x_1x_4}- \frac{x_1}{x_3x_4}\Big),$ \\
\end{tabular}\right\}
\end{equation}
where $c_{12}^3=[123]$ and $c_{13}^4=[134]$.
\end{prop}    
   
\begin{proof}
In order to apply Proposition \ref{Ricc} we need to  check which triples 
$[ijk]$ do not vanish.  It is sufficient to compute the brackets
 $[\fr{m}_{i}, \fr{m}_{j}]$ between the real irreducible submodules $\fr{m}_{i}$ of $\fr{m}$.
According to (\ref{bas}), each  real submodule $\fr{m}_{i}$ associated to the positive $\fr{t}$-root $\xi_{i}$ 
 can be expressed in  terms of  root vectors $E_{\pm\al}$  $(\al\in R_M^+)$, such that $\kappa(\al)=\xi_{i}$.  
So from (\ref{str}) we can compute the  brackets $[\fr{m}_{i}, \fr{m}_{j}]$, 
for suitable root vectors $E_{\al}$.  

Let $M=G/K$ be a flag manifold  of Type IIa. Then 
we obtain that $[\fr{m}_i, \fr{m}_i]\subset\fr{k}$ for  $1\leq i\leq 4$, and
\[
\begin{tabular}{lll}
$[\fr{m}_1, \fr{m}_2]\subset\fr{m}_3$, &   
$[\fr{m}_1, \fr{m}_3]\subset\fr{m}_2,$ & 
$[\fr{m}_1, \fr{m}_4]\subset\fr{k},$ \\
$[\fr{m}_2, \fr{m}_3]\subset\fr{m}_1\oplus\fr{m}_4,$ & 
$[\fr{m}_2, \fr{m}_4]\subset\fr{m}_3,$ & 
$[\fr{m}_3, \fr{m}_4]\subset\fr{m}_2.$
\end{tabular}
\]
Using the  
definition of the triples $[ijk]$ and the fact that 
$B(\fr{m}_i, \fr{m}_j)=B(\fr{m}_i, \fr{k})=0$, 
it follows that the only non-zero $[ijk]$'s  are $[123]$, $[234]$, and their symmetries.
A straightforward application of   relation (\ref{ricc}) gives  now $(\ref{compII})$.

Let now $M=G/K$ be a flag manifold of Type IIb.  
Then we obtain the relations $[\fr{m}_i, \fr{m}_i]\subset\fr{k}$, and
 \[
\begin{tabular}{lll}
$[\fr{m}_1, \fr{m}_2]\subset\fr{m}_3$, &   
$[\fr{m}_1, \fr{m}_3]\subset\fr{m}_2\oplus\fr{m}_4,$ & 
$[\fr{m}_1, \fr{m}_4]\subset\fr{m}_3,$ \\
$[\fr{m}_2, \fr{m}_3]\subset\fr{m}_1,$ & 
$[\fr{m}_2, \fr{m}_4]\subset\fr{k},$ & 
$[\fr{m}_3, \fr{m}_4]\subset\fr{m}_1.$
\end{tabular}
\]
We can easily conclude that the only non-zero $[ijk]$'s are $[123]$, $[134],$  
and their symmetries, 
thus by applying (\ref{ricc})  we obtain   $(\ref{compIII})$.
\end{proof}

The next table gives the dimensions  $d_{i}=\dim\fr{m}_{i}$ 
of the irreducible submodules $\fr{m}_i$.
 
\medskip
 \begin{center}
\begin{tabular}{|l|c|c|c|c|}
\hline 
$G$ & $d_1$ & $d_2$ & $d_3$ & $d_4$\\
\hline 
$SO(2\ell+1)$ & $2$ & $2(2\ell-3)$ & $2(2\ell-3)$ & $2$\\
\hline
$Sp(\ell)$ & $2p(\ell-p)$ & $(\ell-p)(\ell-p+1)$ & $2p(\ell-p)$ & $p(p+1)$\\
\hline 
$SO(2\ell)_{(i)}$ & $2$ & $4(\ell-2)$ & $4(\ell-2)$ & $2$\\
\hline
$SO(2\ell)_{(ii)}$ & $2p(\ell-p)$ & $(\ell-p)(\ell-p-1)$ & $2p(\ell-p)$ & $p(p-1)$\\
\hline
$E_6$ & $2$ & $20$ & $20$ & $10$\\
\hline
$E_7$ & $2$ & $32$ & $32$ & $20$\\
\hline
\end{tabular}
\end{center}

\smallskip
\begin{center}
{\sc{Table 8.}} \ The  dimensions of the isotropy submodules for $M=G/K$ of Type II. 
\end{center}
\medskip

 By taking into account the explicit form of the K\"ahler-Einstein metrics in  Theorem \ref{The4},  
 and substituting these in (\ref{compII}) and (\ref{compIII}), 
 we can find the   values of the unknown triples $[ijk]$ in Proposition \ref{componentsII}.
Note that since we are looking for only two unknowns, 
these triples will be the solutions 
of any of the systems  
\[
\left\{ 
\begin{tabular}{l}
 $r_1=r_2$ \\
 $r_2=r_3$ 
 \end{tabular}\right\},
 \qquad 
 \left\{ 
\begin{tabular}{l}
 $r_2=r_3$ \\
 $r_3=r_4$ 
 \end{tabular}\right\},
 \qquad 
 \left\{ 
\begin{tabular}{l}
 $r_3=r_4$ \\
 $r_4=r_1$ 
 \end{tabular}\right\}.
 \qquad 
 \]
 Therefore, we obtain the following:
\begin{lemma}\label{Lem2}
(1) \ For a flag manifold $M=G/K$ of Type IIa, the non-zero numbers 
$c_{12}^{3}=[123]$ and $c_{23}^{4}=[234]$ are given as follows:
\[
 \begin{tabular}{|c|c|c|} 
 \hline
 $G$ & $c_{12}^{3}=[123]$ & $c_{23}^{4}=[234]$ \\
 \hline 
 $  SO(2\ell+1) $ & $ \displaystyle \frac{2\ell-3}{2\ell-1} $ & $\displaystyle\frac{2\ell-3}{2\ell-1}$ \\ 
 \hline
$  SO(2\ell)_{(i)}$  & $ \displaystyle\frac{\ell-2}{\ell-1}  $ & $\displaystyle \frac{\ell-2}{\ell-1}$ \\
\hline
$ E_6  $  &  $ 5/6 $ & $  5/2 $  \\ 
\hline
 $ E_7  $ & $ 8/9 $ & $ 40/9 $ \\ 
 \hline
 \end{tabular}
 \]
(2) \  For a flag manifold $M=G/K$ of Type IIb, the non-zero numbers 
$c_{12}^{3}=[123]$ and $c_{13}^{4}=[134]$ are given as follows:
\[
 \begin{tabular}{|c|c|c|} 
 \hline
 $G$ & $c_{12}^{3}=[123]$ & $c_{13}^{4}=[134]$ \\
 \hline
    $ Sp(\ell) $ & $ \displaystyle\frac{p(\ell-p)(\ell-p+1)}{2(\ell+1)} $ &  $ \displaystyle\frac{p(p+1)(\ell-p)}{2(\ell+1)}$  \\
    \hline
 $ SO(2\ell)_{(ii)}$ & $ \displaystyle\frac{p(\ell-p)(\ell-p-1)}{2(\ell-1)} $ & $ \displaystyle\frac{p(p-1)(\ell-p)}{2(\ell-1)}$\\
 \hline
 \end{tabular}
 \]
\end{lemma}

\subsection{Solutions of the Einstein equation}
 An invariant  metric $g=(x_1, x_2, x_3, x_4)$ on  a generalized 
 flag manifold $M=G/K$    
 of Type II  is Einstein if and only if, it is a solution of system (\ref{systemI})
    where in this case the components $r_i$ $(i=1, \ldots, 4)$ are determined by Proposition \ref{componentsII}.

 $\bullet$ {\bf Type IIa.}
 Let $M=G/K$  be a generalized flag manifold of Type IIa.
 System (\ref{systemI}) reduces  to the following system of non linear 
 polynomial equations:
 \begin{equation}\label{systemIIa}
 \left.
 \begin{tabular}{r}
  $ d_1d_2x_3x_4(x_2-x_1)+c_{12}^3x_1^2x_4(d_1+d_2)-c_{12}^3x_2^2x_4(d_1+d_2)-c_{12}^3x_2^2x_4(d_2-d_1)$\\
  $ -c_{23}^4d_1x_1(x_2^2-x_3^2-x_4^2)=0$\\
  $d_2d_3x_1x_4(x_3-x_2)+(x_2^2-x_3^2)(d_2+d_3)(c_{12}^3x_4+c_{23}^4x_1)$\\
  $ -x_1x_4(d_3-d_2)(c_{12}^3x_1+c_{23}^4x_4)=0$\\
  $d_3d_4x_1x_2(x_4-x_3)+c_{12}^3d_4x_4(x_3^2-x_2^2-x_1^2)+c_{23}^4x_1(d_3+d_4)(x_3^2-x_4^2)$\\
  $+c_{23}^4x_2^2x_1(d_3-d_4)=0.$ 
  \end{tabular}\right\}
  \end{equation}   
Because of  Table 8 and Lemma \ref{Lem2}, the coefficients of system (\ref{systemIIa})
	are  completely determined.  

According to \cite[p.~1053]{Gr}  the number of   invariant {\it complex} 
Einstein metrics
on the  flag manifolds of Type 
IIa corresponding to the exceptional Lie groups $E_6$ and $E_7$, 
is twelve, i.e. $\mathcal{E}(M)=12$.  
By setting $x_1=1$ in   system (\ref{systemIIa}) we can get approximate values for all 
these complex solutions.
We can sharpen this result and  prove that these 
exceptional flags admit  eight {\it real} invariant Einstein metrics, explicity 
given in the following theorem. 

\begin{theorem}\label{The6}
  Let $M=G/K$ be a generalized flag manifold of Type IIa, associated 
  to the exceptional Lie groups
$E_6$ and $E_7$.
Then $M$ admits  (up to scale),  precisely eight invariant Einstein metrics.   
These metrics are approximately  given as  follows:
 \begin{itemize}
 \item  $E_6$  
  \begin{tabular}{ll}
   $(a) \ (1,  0.568845,  0.568845,  0.452648),$ & $(b) \ (1, 3.81171,  3.81171,  7.45484),$ \\
   $(c) \ (1, 4.93397, 4.93397,  3.34633)$, & $(d) \ (1, 0.685474,  0.685474,  1.19063),$ \\
   $(e) \ (1,  0.636364,  0.363636,  0.272727),$ & $(f) \ (1, 0.363636, 0.636364,  0.272727),$ \\
   $(g) \ (1, \ 4, \ 5,  \ 9),$ & $(h) \ (1, \ 5, \ 4, \ 9).$
   \end{tabular} \\\\
  \item  $ E_7$  
  \begin{tabular}{ll}
   $(a) \ (1,  7.46064,  7.46064,  5.7877),$ & $(b) \ (1, 5.79359,  5.79359,  11.4613),$ \\
   $(c) \ (1, 0.704472,  0.704472,  1.27517)$, & $(d) \ (1, 0.579765,  0.579765,   0.505408),$ \\
   $(e) \ (1,  0.352941,  0.647059,  0.294118),$ & $(f) \ (1, 0.647059, 0.352941,  0.294118),$ \\
   $(g) \ (1, \ 6, \ 7,  \ 13),$ & $(h) \ (1, \ 7, \ 6, \ 13).$
   \end{tabular}
\end{itemize}
In both cases, the metrics (e), (f), (g) and (h) are K\"ahler-Einstein metrics.
\end{theorem}
 
We will now discuss the flag manifolds $M=G/K$   of Type IIa corresponding to  
the classical Lie groups.

Set $G=SO(2\ell+1)$, that is $M=SO(2\ell+1)/U(1)\times U(1)\times SO(2\ell-3)$.   
 By use of Table 8 and Lemma \ref{Lem2}  system (\ref{systemIIa}) 
 reduces   to 
 \begin{equation}\label{syst1}
 \left.
 \begin{tabular}{r}
$2(2\ell-1)(x_2-x_1)+2(\ell-1)x_4(x_1^2-x_2^2)-2(\ell-2)x_4x_3^2 -x_1(x_2^2-x_3^2-x_4^2)=0$\\
$(2\ell-1)x_1x_4(x_3-x_2)+(x_1+x_4)(x_2^2-x_3^2)=0$\\
$2(2\ell-1)x_1x_2(x_4-x_3)+x_4(x_3^2-x_1^2-x_2^2) +2(\ell-1)x_1(x_3^2-x_4^2)+2(\ell-2)x_1x_2^2=0.$
\end{tabular}\right\} 
  \end{equation} 
According to  \cite[p.~1053]{Gr}, $M$ admits 
ten {\it complex} Einstein metrics, i.e. $\mathcal{E}(M)=10$.

\begin{theorem}\label{The7}
The space $M=SO(2\ell+1)/U(1)\times U(1)\times SO(2\ell-3)$ for $\ell\geq 3$  
admits (up to a scale) precisely eight $SO(2\ell+1)$-invariant 
   Einstein metrics. 
Four of them are K\"ahler, given by 
\[ 
 \begin{tabular}{ll}
   $(a) \ (1,   \ \ell-3/2,  \ \ell-1/2,  \ 2(\ell-1)),$ & $(b) \ (1, \ \ell-1/2, \  \ell-3/2, \ 2(\ell-1)),$ \\
   $(c) \ ( 2(\ell-1),  \ \ell-1/2, \ \ell-3/2, \ 1)$, & $(d) \ ( 2(\ell-1), \ \ell-3/2,  \ \ell-1/2, \ 1),$ 
   \end{tabular}
   \]
and the other four are non-K\"ahler. Two of them are 
given explicity  as follows:
\begin{equation}\label{sol1}
   x_1=x_4=1, \quad   x_2=x_3= \displaystyle\frac{2\ell-1\pm\sqrt{4\ell^2- 12\ell + 5 }}{4}.
 \end{equation}
 \end{theorem}
 
 \begin{proof}
In order to find real solutions of  (\ref{syst1}) we distinguish the following cases. 

Let $x_1=x_4=1$ and $x_2=x_3$.  Then the second equation of (\ref{syst1}) 
is satisfied, and both the  first and  third equations  reduce to the   equation,  
 \[
 4x_3^2-2(2\ell-1)x_3+2\ell-1=0.
 \]
For $\ell\geq 3$ we get two real solutions given by $x_3=\displaystyle\frac{2\ell-1\pm\sqrt{4 \ell^2- 12\ell + 5 }}{4}$.

Let $x_1=1, x_1\neq x_4$ and $x_2=x_3$.  Then the second equation  in (\ref{syst1})  is 
satisfied,  and the other two reduce to the system
 \begin{equation}\label{syst2}
 \left.
 \begin{tabular}{r}
$4x_3^2-2(2\ell-1)x_3+x_4+2(\ell-1)=0$\\
$4x_3^2-2(2\ell-1)x_3x_4+2(\ell-1)x_4^2+x_4=0.$ 
\end{tabular}\right\} 
  \end{equation} 
By solving system (\ref{syst2}) we obtain two new real solutions.

Let $x_1=x_4=1$ and $x_2\neq x_3$. Then  system (\ref{syst1})  reduces to 
 \begin{equation}\label{syst3}
 \left.
 \begin{tabular}{r}
$(2\ell-5)x_3^2+(2\ell-1)x_2^2-4(\ell-2)x_2x_3+(4\ell-2)x_3-(2\ell-1)=0$\\
$2(x_2+x_3)-(2\ell-1)=0 $ \\
$(2\ell-5)x_2^2+(2\ell-1)x_3^2-4(\ell-2)x_2x_3+(4\ell-2)x_2-(2\ell-1)=0.$
\end{tabular}\right\} 
  \end{equation} 
By subtracting the first and the third equation of (\ref{syst3}) we get the second  equation.
By solving the system consists of the later and   the  first or  
third equation, we get the same two complex solutions.
 Since $\mathcal{E}(M)=10$ we conclude that    $M$
does not admit any other (non-K\"ahler) real Einstein metric.
\end{proof}

A similar method can be used to solve the Einstein equation for 
the flag manifold of Type IIa corresponding to $G=SO(2\ell)$, i.e. 
the space
$M=G/K=SO(2\ell)/U(1)\times U(1)\times SO(2\ell-4)$.  
Thus we  obtain the following theorem.

\begin{theorem}\label{The77}
 The space $M=SO(2\ell)/U(1)\times U(1)\times SO(2\ell-4)$ for $\ell\geq 3$   
admits (up to a scale) precisely eight $SO(2\ell+1)$-invariant 
  Einstein metrics. 
Four of them are K\"ahler, given by 
\[ 
 \begin{tabular}{ll}
   $(a) \ (1, \ \ell - 2, \ \ell - 1, \ 2\ell - 3),$ & $(b) \ (1, \ \ell-1, \  \ell-2, \ 2\ell-3 ),$ \\
   $(c) \ (  2\ell-3,  \ \ell-2 \ \ell-1, \ 1)$, & $(d) \ ( 2\ell-3, \   \ell-1, \ \ell-2,    \ 1),$ 
   \end{tabular}
   \]
and the other four are non-K\"ahler. Two of them are given 
explicity  as follows:
\begin{equation}\label{sol2}
   x_1=x_4=1, \quad   x_2=x_3= \displaystyle\frac{\ell-1 \pm \sqrt{\ell^2 - 4\ell +3}}{2}.
 \end{equation}
  \end{theorem}

Note that for $\ell=3$ we obtain the full flag manifold 
$M=SO(6)/U(1)\times U(1)\times SO(2)\cong SO(6)/SO(2)\times SO(2)\times SO(2)$,
which is a normal homogeneous Einstein manifold (according to \cite[p.~568]{Wa1}). Indeed, for 
$\ell=3$  the Einstein metric (\ref{sol2}) reduces to the normal metric $g=(1, 1, 1, 1)$.
 
 \smallskip
$\bullet$ {\bf Type IIb.}
We now come to  generalized flag manifolds $M=G/K$ of Type IIb.  
An invariant metric $g=(x_1, x_2, x_3, x_4)$ of $M$ is 
Einstein if and only if, it is a solution of system (\ref{systemI}), 
where $r_{i}$ $(i=1, \ldots, 4)$ are given by (\ref{compIII}).
In this case system  (\ref{systemI}) reduces to the following system of 
non linear polynomial equations:
 \begin{equation}\label{syst4}
 \left.
 \begin{tabular}{r}
  $d_1d_2x_3x_4(x_2-x_1)+c_{12}^3x_4(d_1+d_2)(x_1^2-x_2^2)+c_{12}^3x_3^2x_4(d_1-d_2)$\\
  $+c_{13}^4d_2x_2(x_1^2-x_3^2-x_4^2)=0 $\\
  $d_2d_3x_1x_4(x_3-x_2)+c_{12}^3x_4(d_2+d_3)(x_2^2-x_3^2)+c_{12}^3x_1^2x_4(d_2-d_3)$\\
  $-c_{13}^4d_2x_2(x_3^2-x_1^2-x_4^2)=0$\\
  $d_3d_4x_1x_2(x_4-x_3)+c_{13}^4x_2(d_3+d_4)(x_3^2-x_4^2)+c_{13}^4x_1^2x_2(d_3-d_4)$\\
  $+c_{12}^3d_4x_4(x_3^2-x_2^2-x_1^2)=0,$ 
  \end{tabular}\right\}
  \end{equation} 
 	where $d_{i}$ $(i=1, \ldots, 4)$ and $c_{12}^{3}, c_{13}^{4}$ 
 	are determined from Table 8 and  Lemma \ref{Lem2}, respectively.

Let $G=SO(2\ell)$, that is $M=SO(2\ell)/U(p)\times U(\ell-p)$ 
with $\ell\geq 4$ and  $2\leq p \leq \ell-2$.   
Then system (\ref{syst4})   becomes
 \begin{equation}\label{syst5}
 \left.
 \begin{tabular}{r}
$4(\ell-1) x_3 x_4 (x_2-x_1)+ (\ell+p-1)x_4(x_1^2-x_2^2)- (\ell-3p-1)x_3^2x_4$\\
$ +(p-1)x_2(x_1^2-x_3^2-x_4^2)=0$\\
$4(\ell-1)x_1x_4(x_3-x_2)+ (\ell+p-1)x_4(x_2^2-x_3^2)+ (\ell-3p-1)x_1^2x_4$\\
$ -(p-1)x_2(x_3^2-x_1^2-x_4^2)=0$\\
$4(\ell-1)x_1x_2(x_4-x_3)+ (2\ell-p-1)x_2(x_3^2-x_4^2)+ (2\ell-3p+1)x_1^2x_2$\\
$ +(\ell-p-1)x_4(x_3^2-x_1^2-x_2^2)=0,$\\
\end{tabular}\right\} 
  \end{equation} 
System (\ref{syst5})  is quite complicated, so we consider the following cases.

Let  $x_2=x_4=1$  and $x_1=x_3$.  Then system (\ref{syst5}) 
reduces to  
\begin{equation}\label{gen1}
 \left.
 \begin{tabular}{r}
$4(\ell-p-1)x_1^2-4(\ell-1)x_1+(\ell+2p-2)=0$\\
$4(p-1)x_1^2-4(\ell-1)x_1+(3\ell-2p-2)=0$\\
\end{tabular}\right\} 
  \end{equation} 
By comparing these equations  it follows that system (\ref{gen1}) is solvable only when  $\ell=2p$.   

If  $x_2=x_4=1$ and  $x_1\neq x_3$,  then system (\ref{syst5}) reduces to the 
 following system of three equations and two unknowns
 \begin{equation}\label{gen2}
 \left.
 \begin{tabular}{r}
$4(\ell-1)x_1x_3-4(\ell-1)x_3+(\ell-2p-2)x_3^2-(\ell+2p-2)x_1^2+(\ell+2p-2)=0$\\
$4(\ell-1)x_1x_3-4(\ell-1)x_1-(\ell+2p-2)x_3^2+(\ell-2p-2)x_1^2+(\ell+2p-2)=0$\\
$4(\ell-1)x_1x_3-4(\ell-1)x_1-(3\ell-2p-2)x_3^2-(\ell-2p+2)x_1^2+(3\ell-2p-2)=0$
\end{tabular}\right\} 
  \end{equation}  
   By  using the second and the third equation of (\ref{gen2})
 we conclude that this sytem is also solvable only when $\ell=2p$. This mean that we have to exam  the case $\ell=2p$ separately.  We obtain the following theorem:
 
\begin{theorem}\label{The8}
The flag manifold $M=SO(4p)/U(p)\times U(p)$ $(p\geq 2)$ 
admits  at least six (up to scale) $SO(4p)$-invariant   Einstein metrics.
Four of them are K\"ahler,  given by
\[ 
 \begin{tabular}{ll}
   $(a) \ (p, \  p-1, \ 2p-1, \ 3p-1),$ & $(b) \ (p, \ 3p-1, \  2p-1, \ p-1 ),$ \\
   $(c) \ (  2p-1,  \ 3p-1, \ p, \ p-1)$, & $(d) \ ( 2p-1, \   p-1, \ p,    \ 3p-1),$ 
   \end{tabular}
   \]
   and the  other two are non-K\"ahler,  explicity given by  
   \begin{equation}\label{sol3}
x_2=x_4=1, \quad x_1=x_3=\displaystyle\frac{2p-1\pm\sqrt{2p-1}}{2(p-1)}.
\end{equation}  
   In the special case where $2\leq p\leq 6$, 
   $M$ admits precisely eight (up to scale) $SO(4p)$-invariant Einstein metrics.  
   In this case,   
   two new (non-K\"ahler)   Einstein metrics are
     given  by   
     \begin{equation}\label{sol4}
 \left.
 \begin{tabular}{l}
$x_2=1,   \quad x_4=\displaystyle\frac{ 7 p^3 - p^2 - 3 p + 1 \pm 
 2 (2 p - 1)\sqrt{
  2 p (-p^3 + 7 p^2 - 5 p + 1)}}{(p-1) (3p-1)^2} $\\\\
$x_1=x_3=\displaystyle\frac{4 p^2-2 p   \pm 
 \sqrt{2(7 p^3-5 p^2+p-p^4})}{2 (3p^2-4p+1)}=\sqrt{\frac{p}{2(p-1)}x_4}.$ 
\end{tabular}\right\} 
  \end{equation}  
\end{theorem}
\begin{proof}
 For $\ell=2p$ we obtain the flag manifold  
  $M=SO(4p)/U(p)\times U(p)$, 
 and according to  \cite{Gr}  it is $\mathcal{E}(M)=10$. 
 
 For $\ell=2p$
 system  (\ref{gen1})  reduces to the equation
 \[
 4(p-1)x_1^2-4(2p-1)x_1+4p-2=0.
 \]
For $p\geq 2$ we 
find two real solutions, given by $x_1=\displaystyle\frac{2p-1\pm\sqrt{2p-1}}{2(p-1)}$. 
  
For $\ell=2p$ system (\ref{gen2}) reduces to
\[
 \left.
 \begin{tabular}{r}
$4(2p-1)x_1x_3-4(2p-1)x_3-2(2p-1)x_1^{2}-2x_3^{2}+4p-2=0$\\
$4(2p-1)x_1x_3-4(2p-1)x_1-2(2p-1)x_3^{2}-2x_1^{2}+4p-2=0,$ 
  \end{tabular}\right\} 
  \]
and by solving it we obtain two complex solutions.  

 Finally, we consider the case $x_2=1, x_2\neq x_4$ and $x_1=x_3$.  
Then   for $\ell=2p$ system (\ref{syst5}) reduces to  
\begin{equation}\label{syst6}
 \left.
 \begin{tabular}{r}
$ 4(p-1)x_1^2-4(2p-1)x_1+(p-1)x_4+3p-1=0$\\
$4(p-1)x_1^2+(3p-1)x_4^2-4(2p-1)x_1x_4+(p-1)x_4=0.$
\end{tabular}\right\} 
  \end{equation} 
 By solving (\ref{syst6}) for $2\leq p\leq 6$ we obtain two new  real  solutions given by
 \[
 \left.
 \begin{tabular}{l}
$  x_4=\displaystyle\frac{ 7 p^3 - p^2 - 3 p + 1 \pm 
 2 (2 p - 1)\sqrt{
  2 p (-p^3 + 7 p^2 - 5 p + 1)}}{(p-1) (3p-1)^2} $\\\\
$x_1=\displaystyle\frac{4 p^2-2 p   \pm 
 \sqrt{2(7 p^3-5 p^2+p-p^4})}{2 (3p^2-4p+1)}=\sqrt{\frac{p}{2(p-1)}x_4}.$ 
\end{tabular}\right\} 
  \]
  
Since $\mathcal{E}(M)=10$ we 
conclude that when  $2\leq p\leq 6$ there no other  
real (non-K\"ahler) invariant  Einstein metrics.  
 \end{proof}

We remark that the K\"ahler-Einstein metrics $(a)-(d)$, and the non-Kahler metrics (\ref{sol3}) and (\ref{sol4}) 
were also obtained by the first author in \cite[Theorem.~9]{Arv}, but here we give a corrected version.
 
 
 Let now return to the general space  $M=SO(2\ell)/U(p)\times U(\ell-p)$ and 
 examine system (\ref{syst5})  when $x_2=1, x_2\neq x_4$, and $x_1=x_3$.
 In this case  we obtain the system
\begin{equation}\label{sakane}
 \left.
 \begin{tabular}{r}
$4(\ell-p-1)x_1^2-4(\ell-1)x_1+(p-1)x_4+(\ell+p-1)=0$\\
$ 4(\ell-1)x_1x_4-4(p-1)x_1^2-(2\ell-p-1)x_4^2-(\ell-p-1)x_4=0.$\\
\end{tabular}\right\}
\end{equation}  
 
 \begin{lemma}\label{Lem3}
 If $\ell\geq 4$ and $2\leq p\leq \ell-2$, then system (\ref{sakane}) admits at least two positive real solutions.
 \end{lemma}
 
 \begin{proof}
  The first equation
 of (\ref{sakane}) gives that $x_4=\displaystyle\frac{(2x_1-1)(\ell+p-1-2(\ell-p-1)x_1)}{p-1}$, 
 so for $\displaystyle\frac{1}{2} < x_1 < \displaystyle\frac{\ell+p-1}{2(\ell-p-1)}$, it follows that $x_4>0$.
 By substituting  this value  into the second equation of (\ref{sakane})  we obtain the   equation 
 \begin{eqnarray}\label{function}
 F(x_1) &=& -8(\ell-p-1)^{2}(2\ell-p-1)x_1^{4}+8(\ell-1)(4\ell-3p-1)(\ell-p-1)x_1^{3} \\ \nonumber
 && -2(12\ell^3-11p\ell^2-25\ell^2-2p^{2}\ell+20p\ell+14\ell+2p^3-2p^2-6p-2)x_1^{2} \\ \nonumber
 && +4(\ell-1)(2\ell^2-2\ell-p^2+p)x_1+(1-\ell)\ell(\ell+p-1)=0.
 \end{eqnarray}
 Note that $F(1/2)=\displaystyle\frac{-(p-1)^{3}}{2}<0$, and 
 $F(\displaystyle\frac{\ell+p-1}{2(\ell-p-1)})=-\displaystyle\frac{(p-1)^{3}(\ell+p-1)^{2}}{2(\ell-p-1)^{2}}<0$.
 Set 
 \begin{equation}\label{zeta}
 \zeta=\displaystyle\frac{1}{2}\Big(\frac{1}{2}+\frac{\ell+p-1}{2(\ell-p-1)}\Big).
 \end{equation}
 Then $\displaystyle\frac{1}{2}<\zeta<\displaystyle\frac{\ell+p-1}{2(\ell-p-1)}$ and we
  claim that $F(\zeta)>0$. Indeed, by substituting $\zeta$  into (\ref{function})  we obtain that
\begin{equation}\label{fun}
F(\zeta)=\displaystyle\frac{(\ell-1)\cdot Q(\ell, p)}{2(\ell-p-1)^2},
\end{equation}
where   $Q(\ell, p)=-2p^3+2\ell p^2-2p^2-3\ell p+3p+\ell-1 = (\ell-1-p)(2p-1)(p-1)-p(3p-1)$. 
Thus, for $\ell-1-p\geq 2$,  we see that 
\[
Q(\ell, p)\geq  2(2p-1)(p-1)-p(3p-1)=p^{2}-5p+2,
\]
so $Q(\ell,p)>0$ for  $p\geq 5$, and $F(\zeta)>0$ for $5\leq p\leq \ell-3$. 
We  conclude that equation (\ref{function}) 
has at least two positive   solutions $ x_1^a$ and  $ x_1^b$, 
with  $\displaystyle\frac{1}{2} <  x_1^a <\zeta$,
and   $\zeta <  x_1^b< \displaystyle\frac{\ell+p-1}{2(\ell-p-1)}$. 

For $\ell-1-p\geq 3$ we obtain that
\[
Q(\ell, p)\geq  3(2p-1)(p-1)-p(3p-1)=3p^{2}-8p+3,
\]
so $Q(\ell,p)>0$   for $p\geq 3$, and  it is $F(\zeta)>0$ for $3\leq p\leq \ell-4$. 
We conclude that (\ref{function}) has at least two positive  solutions $x_1^a$ and  $x_1^b$,  
  with  $\displaystyle\frac{1}{2} < x_1^a <\zeta$,  and   $\zeta < x_1^b< \displaystyle\frac{\ell+p-1}{2(\ell-p-1)}$.

We now examine the case  $p=2$.  Then (\ref{zeta}) 
gives that $\zeta=\displaystyle\frac{\ell+1}{2(\ell-3)}$, and (\ref{fun}) reduces to  $F(\zeta)=\displaystyle \frac{(\ell-1)(3\ell-19)}{2(\ell-3)^{2}}$, 
which is positive for $\ell\geq 7$.  
Thus, for  $p=2$ and $\ell\geq 7$,  
equation (\ref{function}) admits at least  two solutions $x_1^a, x_1^b$ 
with $\displaystyle\frac{1}{2} < x_1^a <\displaystyle\frac{\ell+1}{2(\ell-3)}$ 
 and   $\displaystyle\frac{\ell+1}{2(\ell-3)} < x_1^b< \displaystyle\frac{\ell+p-1}{2(\ell-p-1)}$.   
 For $\ell=4, 5, 6$ and $p=2$, we can easily see that 
  (\ref{function}) has  
four positive solutions which satisfy 
$\displaystyle\frac{1}{2} < x_1 < \displaystyle\frac{\ell+p-1}{2(\ell-p-1)}$. 

The other cases we have to check that 
$F(\zeta)>0$, are $(\ell, p)=(6, 3)$,  $(\ell, p)=(7, 4)$, and $p=\ell-2$.
For the first two cases we see that  (\ref{function}) has four solutions, 
which satisfy 
$\displaystyle\frac{1}{2} < x_1 < \displaystyle\frac{\ell+p-1}{2(\ell-p-1)}$.  
Finally, we consider the case $p=\ell-2$.  
Then,   $\displaystyle\frac{\ell+p-1}{2(\ell-p-1)}$
is given by $\displaystyle\frac{1}{2}(2\ell-3)$, and we see 
that $\displaystyle\frac{1}{2}(2\ell-3)>\displaystyle\frac{\ell}{2}$, for $\ell\geq 4$.  
By substituting $x_1=\ell/2$ in (\ref{function}), 
we get $F(\ell/2)=\ell(\ell-3)^{2}>0$. 
Therefore $F(x_1)=0$ has at least two solutions for  
$\displaystyle\frac{1}{2} < x_1 <  \displaystyle\frac{1}{2}(2\ell-3)$.  
We remark that when   $p=\ell-2$ and  $\ell\leq 6$, then 
equation $F(x_1)=0$ has   four solutions, but for $\ell\geq 7$ the equation 
$F(x_1)=0$ might has only 2 positive real solutions.
 \end{proof}
 
 From Lemma \ref{Lem3} it follows that $M=SO(2\ell)/U(p)\times U(\ell-p)$ admits 
 at least two (real) $G$-invariant Einstein metrics of the form $g=(x_1, 1, x_1, x_4)$.
 Since any K\"ahler-Einstein metric on $M$ is given by four distinct parameters $x_1, x_2, x_3, x_4$, 
 we obtain the following existence theorem.
 
 \begin{theorem}\label{existence}
 Let $M=SO(2\ell)/U(p)\times U(\ell-p)$ with $\ell\geq 4$ and $2\leq p\leq \ell-2$.  Then $M$ admits at least two 
   non-K\"ahler $SO(2\ell)$-invariant Einstein metrics.  
 \end{theorem}

 A similar analysis can be applied to the symplectic
 flag manifold $M=Sp(\ell)/U(p)\times U(\ell-p)$ ($1\leq p\leq \ell-1$). 
 Then system (\ref{syst4}) becomes 
\begin{equation}\label{syst8}
 \left.
 \begin{tabular}{r}
$4(\ell+1) x_3 x_4 (x_2-x_1)+ (\ell+p+1)x_4(x_1^2-x_2^2)- (\ell-3p+1)x_3^2x_4$\\
$ +(p+1)x_2(x_1^2-x_3^2-x_4^2)=0$\\
$4(\ell+1)x_1x_4(x_3-x_2)+ (\ell+p+1)x_4(x_2^2-x_3^2)+ (\ell-3p+1)x_1^2x_4$\\
$ -(p+1)x_2(x_3^2-x_1^2-x_4^2)=0$\\
$4(\ell+1)x_1x_2(x_4-x_3)+ (2\ell-p+1)x_2(x_3^2-x_4^2)+ (2\ell-3p-1)x_1^2x_2$\\
$ +(\ell-p+1)x_4(x_3^2-x_1^2-x_2^2)=0.$ 
\end{tabular}\right\} 
  \end{equation} 
   
 For $\ell=2p$, that is $M=Sp(2p)/U(p)\times U(p)$, we obtain the following theorem.
   
\begin{theorem}\label{The9}
Let $M=Sp(2p)/U(p)\times U(p)$, $p\geq 1$.  Then $M$ admits precisely
six $Sp(2p)$-invariant Einstein metrics.  Four of them are K\"ahler, 
given by
 \[ 
 \begin{tabular}{ll}
   $(a) \ (p, \  p+1, \ 2p+1, \ 3p+1),$ & $(b) \ (p, \ 3p+1, \  2p+1, \ p+1 ),$ \\
   $(c) \ (  2p+1,  \ 3p+1, \ p, \ p+1)$, & $(d) \ ( 2p+1, \   p+1, \ p,    \ 3p+1).$ 
   \end{tabular}
   \] 
   The two non-K\"ahler 
metrics are given by 
\begin{equation}\label{sol5}
x_2=x_4=1, \quad x_1=\frac{6p^3+11p^2+ 6p +1\pm A}{2(p+1)^2 (3p+1)}, \quad x_3=\frac{6p^3+11p^2+ 6p +1\mp A}{2(p+1)^2 (3p+1)},
\end{equation}
where $A=\sqrt{(p+1)^3 (6 p^2+5p+1)}$.
\end{theorem}

\markboth{Andreas Arvanitoyeorgos and Ioannis Chrysikos}{Homogeneous Einstein metrics  on  flag manifolds  with four isotropy summands}
\section{The isometric problem}
\markboth{Andreas Arvanitoyeorgos and Ioannis Chrysikos}{Homogeneous Einstein metrics  on  flag manifolds  with four isotropy summands}

 We will examine the isometric problem for the homogeneous Einstein metrics 
stated in Theorems \ref{The5}, \ref{The6},  \ref{The7}, \ref{The77} \ref{The8} and \ref{The9}.
 In general, this is not  a trivial problem.  We follow the method presented in  
  \cite[p.~ 22]{Nik} (see also \cite[p.~8]{LNF}).
  
 Let $M=G/K$ be a generalized flag manifold with 
 $\fr{m}=\oplus_{i=1}^{4}\fr{m}_i$, $d_{i}=\dim \fr{m}_{i}$, and  $d=\sum_{i=1}^{4}d_{i}=\dim M$. 
 For any $G$-invariant Einstein metric $g=(x_1, x_2, x_3, x_4)$ on  $M$ 
 we determine a  scale invariant given by
 $H_{g}=V^{1/d}S_g$, where  $S_g$ is the scalar curvature of $g$, 
 and $V=V_{g}/V_{B}$ is the quotient of the volumes $V_{g}=\prod_{i=1}^{4}x_{i}^{d_i}$ 
 of the given metric $g$, and $V_{B}$ the volume of the normal metric 
 induced by the negative of the Killing form of $G$. We normalize $V_{B}=1$, 
 so $H_{g}=V_{g}^{1/d}S_g$.   
 The scalar curvature  $S_g$ of a $G$-invariant metric  $g$ on   $M$  
 is given by the following well known  formula (\cite{Wa2}):
    \begin{equation}\label{sc}
  S_g=\sum_{i=1}^{4}d_{i}\cdot r_{i}=\frac{1}{2}\sum_{i=1}^{4}\frac{d_{i}}{x_{i}}-\frac{1}{4}\sum_{1\leq i, j, k\leq 4}[ijk]\frac{x_{k}}{x_{i}x_{j}},
  \end{equation}
  where the components $r_{i}$ of the Ricci tensor  are given by one of the expressions (\ref{compI}),   (\ref{compII}), or (\ref{compIII}).
 The scalar curvature is a homogeneous  polynomial of degree $-1$  on the variables $x_{i}$ ($i=1, \ldots, 4$).  
 The volume $V_{g}$ is a  monomial of degree $d$,
 so   $H_{g}=V_{g}^{1/d}S_g$ is a homogeneous polynomial of degree 0. Therefore,
 $H_{g}$  is invariant under a common scaling of the variables $x_i$. 
  
 If two metrics are isometric then they have the same scale invariant, so if
  the scale invariants $H_{g}$ and $H_{g'}$ 
 are different, then the metrics $g$ and $g'$ can not  be isometric. 
But if $H_{g}=H'_{g}$
 we can not immediately conclude if the metrics $g$ and $g'$ are isometric or not. 
  For such a case we have to look at the   group of automorphisms of $G$ and check if there is an automorphism 
  which permutes the isotopy summands and takes one metric to another.  
  This usually arises for the K\"ahler-Einstein metrics. 
   Recall that  the K\"ahler-Einstein metrics which correspond to 
   equivalent invariant complex structures on $M$ are isometric 
  (\cite[Remark 8.96]{Be}).  
   For an application of this   case we refer to \cite[p.~ 315]{Kim}.  
 
 Let $M=G/K$ be a generalized flag manifold of Type I, and let $g=(x_1, x_2, x_3, x_4)$ 
 be a $G$-invariant metric on $M$. From (\ref{relations}) and (\ref{sc}) we easily obtain that
\begin{eqnarray*}
  S_g&=&\frac{1}{2}\sum_{i=1}^{4}\frac{d_i}{x_i} - \frac{[123]}{2}(\frac{x_1}{x_2x_3} + \frac{x_2}{x_1x_3} + \frac{x_3}{x_1x_2}) - 
 \frac{[134]}{2}(\frac{x_1}{x_3x_4} + \frac{x_3}{x_1x_4}+\frac{x_4}{x_1x_3})\nonumber \\ 
 && -\frac{[112]}{4}(\frac{x_2}{x_{1}^{2}} + \frac{2}{x_2})-\frac{[224]}{4}(\frac{x_4}{x_{2}^{2}} + \frac{2}{x_4}),
\end{eqnarray*}
    where the dimensions $d_{i}$ are given in Table 5, and the triples $[123], [134], [112]$, and $[224]$ are given in Table 7.
   For the invariant Einstein metrics presented in Theorem \ref{The5} we obtain  the following approximate values of the scale invariant
    $H_{g}$.
   \smallskip
   
\smallskip
  \begin{center}
\begin{tabular}{|c||c|c|c|c|}
\hline 
$\mbox{Einstein metrics}$ & $\mbox{Case of} \ F_4$ & $\mbox{Case of} \ E_7$ & $\mbox{Case of} \ {E_{8}}_{(ii)}$ & $\mbox{Case of} \ {E_{8}}_{(i)}$\\
\hline 
$g=(1, 2, 3,4)$ & $H_{g}\cong15.5381$ & $H_{g}\cong38.8641$ & $H_{g}\cong72.1927$ & $H_{g}\cong70.9532$ \\
\hline
$g_1$ & $H_{g_1}\cong15.7376$ & $H_{g_1}\cong39.0998$ & $H_{g_1}\cong72.8754$ & $H_{g_1}\cong70.6326$\\
\hline
$g_2$ & $H_{g_2}\cong15.7255$ & $H_{g_2}\cong38.9954$ & $H_{g_2}\cong72.6779$ & $H_{g_2}\cong77.6071$\\
\hline 
$g_3$ & $-$ & $-$ & $-$ & $H_{g_3}\cong70.6696$\\
\hline
$g_4$ & $-$ & $-$ & $-$ & $H_{g_4}\cong77.3436$\\
\hline
\end{tabular}
\end{center}
 \medskip
\smallskip 
 \begin{center}
{\sc{Table 9.}} {The  constants $H_{g}$  corresponding to Einstein metrics on $M=G/K$ of Type I.}
\end{center} 
\medskip
 
 From Table 9 it follows that the invariant Einstein metrics on a flag manifold of Type I given in Theorem \ref{The5} are not isometric.

Let $M=G/K$  be a generalized flag manifold of Type IIa, and 
let $g=(x_1, x_2, x_3, x_4)$ be a  $G$-invariant  Riemannian metric   on $M$.  
The scalar curvature of $g$ is given by
\[
S_g=\frac{1}{2}\sum_{i=1}^{4}\frac{d_i}{x_i} - \frac{[123]}{2}(\frac{x_1}{x_2x_3} + \frac{x_2}{x_1x_3} + \frac{x_3}{x_1x_2}) - 
 \frac{[234]}{2}(\frac{x_2}{x_3x_4} + \frac{x_3}{x_2x_4}+\frac{x_4}{x_2x_3}),
\]
where $d_{i}$ are given in Table 8 and $[123]$, $[234]$ are determined by Lemma \ref{Lem2}.

Let $M=E_{6}/SU(5)\times U(1)\times U(1)$. For the (non-K\"ahler)  
Einstein metrics $(a)-(d)$ given in Theorem \ref{The6}, 
we obtain that $H_{(a)}\cong 21.0363$, 
$H_{(b)}\cong 20.9202$, $H_{(c)}\cong 20.5771$, and $H_{(d)}\cong 21.1831$, respectively.  
Thus, these Einstein metrics are not isometric.
For the  K\"ahler-Einstein metrics $(e)-(h)$ we obtain 
that $H_{(e)}=H_{(f)}\cong 21.146 $ and $H_{(g)}=H_{(h)}\cong 20.9279$,  
so the metrics $(e)$ and $(f)$ can not be isometric to  the metrics $(g)$ and $(h)$.  
In \cite[p.~51, Table 4]{N} it was shown that
 $M$  admits two pairs of equivalent invariant complex structures, 
thus  there are two pairs of isometric  K\"ahler-Einstein metrics, 
namely $(e),(f)$, and $(g), (h)$.
Similar results are valid for the  invariant Einstein metrics of 
the flag manifold $E_7/SO(10)\times U(1)\times U(1)$ of
Theorem \ref{The6}.

 We now examine the space $M=SO(2\ell+1)/U(1)\times U(1)\times SO(2\ell-3)$ of Theorem \ref{The7}. 
 In this case the scale invariants 
 are functions of the parameter $\ell$.
 By computing these for the
 two non-K\"ahler Einstein metrics given in (\ref{sol1})
 we conclude that these metrics are non isometric.  
 The K\"ahler-Einstein metrics $(a)$ and $(b)$ have  the same scale invariants $H_{(a)}=H_{(b)}$
 and the same is true for the  K\"ahler-Einstein metrics  $(c)$ and $(d)$, i.e. $H_{(c)}=H_{(d)} $. 
  However,   $H_{(a)}\neq H_{(c)}$, 
 therefore     none of the metrics $(a), (b)$ are isometric to  the metrics $(c), (d)$.   
It is known (\cite[p.~47, Theorem 5]{N}) that  
 $M$ admits two pairs of equivalent complex structures,  thus there are two 
 pairs of isometric K\"ahler-Einstein metrics.  The above analysis  implies that the metrics $(a)$ and $(b)$ are isometric, 
 and the same is true  for the metrics  $(c)$ and $(d)$.  
 Similar results are valid for the invariant Einstein metrics of
 the flag $SO(2\ell)/U(1)\times U(1)\times SO(2\ell-4)$
 of Theorem \ref{The77}.
 
Finally, let  $M=G/K$ be a generalized flag manifold of Type IIb, and let $g=(x_1, x_2, x_3, x_4)$ 
 be a $G$-invariant metric on $M$. In this case the  scalar curvature is given by  
\[
S_g=\frac{1}{2}\sum_{i=1}^{4}\frac{d_i}{x_i} - \frac{[123]}{2}(\frac{x_1}{x_2x_3} + \frac{x_2}{x_1x_3} + \frac{x_3}{x_1x_2}) - 
 \frac{[134]}{2}(\frac{x_1}{x_3x_4} + \frac{x_3}{x_1x_4}+\frac{x_4}{x_1x_3}),
\]
where $d_{i}$ are given in Table 8, and $[123]$, $[134]$  by Lemma \ref{Lem2}.

Let $M=SO(4p)/U(p)\times U(p)$, $(p\geq 2)$.  
Then for any $G$-invariant metric $g$ on $M$  the scale invariant $H_{g}$ is a function of $p$.
The functions corresponding to the homogeneous Einstein metrics given in (\ref{sol3}) 
are different, so these metrics are not isometric.
The scale invariants  corresponding to the K\"ahler-Einstein metrics $(a)-(d)$ 
stated in Theorem \ref{The8}  are equal to each other.  
   In \cite[Theorem ~6]{N} it was shown that 
   $M$ admits two equivalent complex structures, 
thus there are two pairs of isometric K\"ahler-Einstein metrics.  
At this point we are unable to determine these pairs.

Let $M=Sp(2p)/U(p)\times U(p)$ $(p\geq 1)$.   For the (non-K\"ahler) Einstein metrics given in (\ref{sol5}),
one can easily see that the scale invariant functions are equal, but we are not able to    conclude  
whether these metrics are isometric or not. 
For the corresponding K\"ahler-Einstein metrics $(a)-(d)$ 
we also obtain equal scale invariant functions.  
According to \cite[Theorem~ 5]{N} $M$ admits two pairs of equivalent complex structures, 
 therefore there are two pairs of isometric K\"ahler-Einstein metrics.

  \section*{Acknowlegments}  
 This work is the main part of the second author's Ph.D. thesis at the University of Patras,
under the direction of the first author.  The  second author  wishes to thank Yusuke Sakane for several useful discussions concerning this project.
 He also acknowledges   Megan Kerr and  Yurii Nikonorov for   their suggestions about the isometric problem.

\end{document}